%% file: real-main-lncs.tex
\documentclass[envcountsect,envcountsame]{llncs}%

\usepackage{ifthen}
\newcommand{\techRep}{true} 
\newcommand{\iftechrep}{\ifthenelse{\equal{\techRep}{true}}}

\usepackage{macros}
\usepackage[inner=3.5cm,outer=3.5cm,top=3.9cm,bottom=2.7cm]{geometry} 

\newcommand{\geoproofs}{false}

\allowdisplaybreaks[4]
\title{Computing the Least Fixed Point \\ of Positive Polynomial Systems \thanks{This work was partially supported by the DFG project
        {\em Algorithms for Software Model Checking}.}}

\author{Javier Esparza \and Stefan Kiefer \and Michael Luttenberger}
\institute{Institut f{\"u}r Informatik, Technische Universit{\"a}t
M{\"u}nchen, 85748 Garching, Germany
\email{\{esparza,kiefer,luttenbe\}@model.in.tum.de}}

\begin{document}
\maketitle

\begin{abstract}
\input{abstract}
\end{abstract}
\input{sec1-intro}
\input{sec2-prelim}
\input{sec3-fundamental}
\input{sec4-strongly-connected}
\input{sec5-decomposed}
\input{sec6-upper-bounds}

\input{sec-geometry}
\input{sec-conclusions}
\iftechrep{\input{acknowledgments}}{} 

\appendix

\input{app5-decomposed}
\ifthenelse{\equal{\geoproofs}{false}}{\input{app-geo}}{}
\iftechrep{}{\input{acknowledgments}} 
\bibliographystyle{alpha} 
\bibliography{db}
\end{document}

%% file: abstract.tex
We consider equation systems of the form
$X_1 = f_1(X_1, \ldots, X_n)$, $\ldots, X_n = f_n(X_1, \ldots, X_n)$
where $f_1, \ldots, f_n$ are polynomials with {\em positive} real coefficients.
In vector form we denote such an equation system by $\vX = \vf(\vX)$ and call $\vf$ a system of positive polynomials, short SPP.
Equation systems of this kind appear naturally in the analysis of stochastic models like stochastic
context-free grammars (with numerous applications to natural language processing and computational biology),
probabilistic programs with procedures, web-surfing models with back buttons, and branching processes.
The least nonnegative solution $\mu\vf$ of an SPP equation $\vX = \vf(\vX)$ is of central interest for these models.
Etessami and Yannakakis \cite{EYstacs05Extended} have suggested a particular version of Newton's method to approximate $\mu\vf$.

We extend a result of Etessami and Yannakakis and show that Newton's method starting at $\vzero$ always converges to $\mu\vf$.
We obtain lower bounds on the convergence speed of the method. For so-called {\em strongly connected} SPPs
we prove the existence of a threshold $k_\vf \in \Nat$ such that for every $i \geq 0$ the ($k_\vf+i$)-th iteration of
Newton's method has at least $i$ valid bits of $\fix{\vf}$.
The proof yields an explicit bound for $k_\vf$ depending only on syntactic parameters of $\vf$.
We further show that for arbitrary SPP equations Newton's method still converges linearly:
there exists a threshold $k_\vf$ and an $\alpha_\vf > 0$ such that for every $i \geq 0$
the ($k_\vf+ \alpha_\vf \cdot i$)-th
iteration of Newton's method has at least $i$ valid bits of $\fix{\vf}$. The proof yields an explicit
bound for $\alpha_\vf$; the bound is exponential in the number of equations in $\vX = \vf(\vX)$, but we
also show that it is essentially optimal. The proof does not yield any bound for $k_\vf$,
it only proves its existence. Constructing a bound for $k_\vf$ is still an open problem.
Finally, we also provide a geometric interpretation of Newton's method for SPPs.

%% file: sec1-intro.tex
\section{Introduction}\label{sec:intro}

We consider equation systems of the form
$$\begin{array}{rcl}
X_1 & = & f_1(X_1, \ldots, X_n) \\
    & \vdots & \\
X_n& = & f_n(X_1, \ldots, X_n)
\end{array}$$
\noindent where $f_1, \ldots, f_n$ are polynomials with {\em positive} real
coefficients.
In vector form we denote such an equation system by $\vX = \vf(\vX)$.
The vector $\vf$ of polynomials is called a {\em system of
positive polynomials}, or {\em SPP} for short.
Figure~\ref{fig:quadric} shows the graph of a 2-dimensional SPP equation system $\vX = \vf(\vX)$.

\begin{figure}[htp]
\begin{center}
{
 \psfrag{X1 = f1(X1,X2)}{\hspace{-6mm}$X_1 = f_1(X_1,X_2)$}
 \psfrag{X2 = f2(X1,X2)}{$X_2 = f_2(X_1,X_2)$}
 \psfrag{muf}{$\mu\vf$}
 \psfrag{0}{$0$}
 \psfrag{0.2}{$0.2$}
 \psfrag{0.4}{$0.4$}
 \psfrag{0.5}{$0.5$}
 \psfrag{0.6}{$0.6$}
 \psfrag{0.8}{$0.8$}
 \psfrag{1}{$1$}
 \psfrag{\2610.4}{$-0.4$}
 \psfrag{\2610.2}{$-0.2$}
 \psfrag{02h}{$0.2$}
 \psfrag{04h}{$0.4$}
 \psfrag{06h}{$0.6$}
 \psfrag{02v}{$0.2$}
 \psfrag{X1}{$X_1$}
 \psfrag{X2}{$X_2$}
  \scalebox{1.0}{ \includegraphics{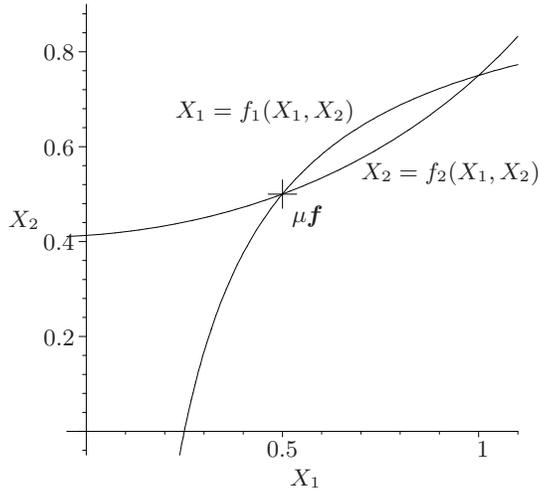}}
}
\end{center}
\caption{Graphs of the equations $X_1 = f_1(X_1,X_2)$ and $X_2 = f_2(X_1,X_2)$ with
         $f_1(X_1,X_2) = X_1X_2 + \frac{1}{4}$ and
         $f_2(X_1,X_2) = \frac{1}{6}X_1^2 + \frac{1}{9}X_1X_2 + \frac{2}{9}X_2^2 + \frac{3}{8}$.
         There are two real solutions in $\R^2$, the least one is labelled with~$\mu\vf$.}
\label{fig:quadric}
\end{figure}

Equation systems of this kind appear naturally in the analysis of
stochastic context-free grammars (with numerous applications to natural language
processing \cite{ManningSchuetze:book,geman02probabilistic} and
computational biology
\cite{Sakabikaraetal,Durbinetal:book,DowellEddy,KnudsenHein}),
probabilistic programs with procedures
\cite{EKM:prob-PDA-PCTL,BKS:pPDA-temporal,EYstacs05Extended,EY:RMC-LTL-complexity,EKM:prob-PDA-expectations,EY:RMC-LTL-QUEST,EY:RMC-RMDP},
and web-surfing models with back buttons
\cite{FaginetalSTOC,Faginetal}. More generally, they play
an important r\^ole in the theory of {\em branching
processes} \cite{Harris63,AthreyaNey:book}, stochastic processes describing the evolution of
a population whose individuals can die and reproduce.
The probability of extinction of the population is the least solution
of such a system, a result whose history goes back to \cite{WatGal1874}.

Since SPPs have positive coefficients, $\vx \leq \vx'$
implies $\vf(\vx) \leq \vf(\vx')$ for \mbox{$\vx,\vx'\in\Rp^n$}, i.e., the
functions $f_1, \ldots, f_n$ are monotonic. This allows us to apply
Kleene's theorem (see for instance \cite{Kui}), and conclude that a feasible
system $\vX = \vf(\vX)$, i.e., one having
at least one nonnegative solution, has a smallest solution $\mu\vf$. It
follows easily from standard Galois theory that $\mu\vf$
can be irrational and non-expressible by radicals. The problem
of deciding, given an SPP and a rational vector $\vv$ encoded in binary,
whether $\mu\vf \leq \vv$ holds, is known
to be in PSPACE, and to be at least as hard as two relevant problems:
SQUARE-ROOT-SUM and PosSLP. SQUARE-ROOT-SUM is a
well-known problem of computational
geometry, whose membership in NP is a long standing open question.
PosSLP is the problem of deciding, giving a division-free
straight-line program, whether it produces a positive integer (see~\cite{EYstacs05Extended}
for more details). PosSLP has been
recently shown to play a central r\^ole in understanding the
Blum-Shub-Smale model of computation,
where each single arithmetic operation over the reals can be carried out
exactly and in constant time~\cite{AllenderBKM09-journal}.

For the practical applications mentioned above the complexity of determining if
$\mu\vf$ exceeds a given bound is less relevant than the complexity of, given
$i \in \Nat$, computing $i$ {\em valid bits} of $\fix{\vf}$, i.e., computing a
vector $\vv$ such that $\abs{\fix{\vf}_j - v_j} / \abs{\fix{\vf}_j} \le
2^{-i}$ for every $1 \le j \le n$. Given an SPP $\vf$ and $i \in \Nat$, deciding whether the
first $i$ bits of a component of $\mu\vf$, say $\mu\vf_1$, are $0$, remains as hard as SQUARE-ROOT-SUM and PosSLP.
The reason is that
in~\cite{EYstacs05Extended} both problems are reduced to the following one:
given $\epsilon > 0$ and an SPP $\vf$ for which it is known that either
$\fix\vf_1 = 1$ or $\mu\vf_1 \leq \epsilon$, decide which of the two is the case. So it suffices
to take $\epsilon = 2^{-i}$.

In this paper we study the problem of computing $i$ valid bits in the Blum-Shub-Smale model.
Since the least fixed point of a feasible SPP~$\vf$
is a solution of $\vF(\vX) = \vzero$ for $\vF(\vX) = \vf(\vX) - \vX$, we can try to
apply (the multivariate version of) \emph{Newton's method}
\cite{OrtegaRheinboldt:book}:
starting at some $\xs{0} \in \mathbb{R}^n$ (we use uppercase to denote variables
and lowercase to denote values),
 compute the sequence
  $$
  \xs{k+1} := \xs{k} - (\vF'(\xs{k}))^{-1}\vF(\xs{k})
  $$
\noindent where $\vF'(\vX)$ is the Jacobian matrix of partial derivatives.
A first difficulty is that the method might not even be well-defined, because
$\vF'(\xs{k})$ could be singular for some~$k$. However, Etessami and Yannakakis
have recently studied SPPs derived from probabilistic
pushdown automata (actually, from an equivalent model called recursive
Markov chains) \cite{EYstacs05Extended}, and shown that a particular
version of Newton's method always converges,
namely a version which decomposes the SPP into
{\em strongly connected components} (SCCs)\footnote{
Loosely speaking, a subset of
variables and their associated equations form an SCC,
if the value of any variable in the subset
influences the value of all variables in the subset,
see \S~\ref{sec:prelim} for details.}
and applies Newton's method to them in a bottom-up fashion. Our first result
generalizes Etessami and Yannakakis': the ordinary Newton method converges
for arbitrary SPPs, provided that they are clean (which can be easily achieved).

While these results show that Newton's method can be an adequate algorithm for solving SPP equations,
they provide no information on the number of iterations needed to compute
$i$ valid bits. To the best of our knowledge (and perhaps surprisingly),
the rest of the literature does not contain relevant information either:
it has not considered SPPs explicitly, and the
existing results have very limited interest for SPPs, since
they do not apply even for very simple and relevant SPP
cases (see {\em Related work} below). In this paper we obtain upper bounds on the number of iterations
that Newton's method needs to produce
$i$ valid bits, first for strongly connected and then for arbitrary SPP equations.

For strongly connected SPP equations $\vX=\vf(\vX)$ we prove the existence of a
threshold $k_\vf$ such that for every $i \geq 0$ the ($k_\vf+i$)-th iteration of
Newton's method has at least $i$ valid bits of $\fix{\vf}$.
So, loosely speaking, after $k_\vf$ iterations Newton's method is guaranteed to
compute at least 1 new bit of the solution per iteration;
we say that Newton's method converges at least {\em linearly} with
{\em rate 1}. Moreover,
we show that the threshold $k_\vf$ can be chosen as
$$k_\vf = \lceil 4mn + 3n \max\{ 0, - \log \mumin \} \rceil$$
\noindent where $n$ is the number of polynomials of the strongly connected SPP, $m$ is
such that all coefficients of the SPP
can be given as ratios of $m$-bit integers, and $\mumin$ is
the minimal component of the least fixed point~$\mu\vf$.

Notice that $k_\vf$ depends on $\fix{\vf}$, which is what
Newton's method should compute. For this reason
we also obtain bounds on $k_\vf$ depending only on $m$ and $n$. We
show that for arbitrary strongly connected SPP equations $k_\vf = 4mn2^n$ is also
a valid threshold. For SPP equations coming from stochastic models,
such as the ones listed above,
we do far better. First, we show that if every procedure has a non-zero
probability of terminating (a condition that always holds for back-button processes
\cite{FaginetalSTOC,Faginetal}), then a valid threshold is $k_\vf = 2m(n+1)$.
Since one iteration requires $\bigo(n^3)$ arithmetic operations in a system of $n$ equations,
we immediately obtain an upper bound on the time complexity of Newton's method
in the Blum-Shub-Smale model: for back-button processes, $i$ valid bits can be computed
in time $\bigo(mn^4 + in^3)$.
Second, we observe that, since $\xs{k} \leq \xs{k+1} \le \fix{\vf}$ holds for every
$k \geq 0$, as Newton's method proceeds it provides better and better lower
bounds for $\mumin$ and thus for~$k_\vf$.
We exhibit an SPP for which, using this fact and
our theorem, we can prove that no component of the
solution reaches the value~1. This cannot be proved by just computing more
iterations, no matter how many.

For general SPP equations, not necessarily strongly connected, we show that Newton's
method still converges linearly. Formally,
we show the existence of a threshold $k_\vf$ and a real number $0 < \alpha_\vf$ such that for
every $i \geq 0$ the ($k_\vf+ \alpha_\vf \cdot i$)-th iteration of
Newton's method has at least $i$ valid bits of $\fix{\vf}$. So, loosely
speaking, after the first $k_\vf$ iterations Newton's method
computes new bits of $\fix{\vf}$ at a rate of at least $1/\alpha_\vf$ bits per iteration.
Unlike the strongly connected
case, the proof does not provide any bound on the threshold $k_\vf$: with respect
to the threshold the proof is non-constructive, and finding a bound on $k_\vf$ is
still an open problem. However, the proof does provide a bound for $\alpha_\vf$, it shows
$\alpha_\vf \leq n \cdot 2^n$ for an SPP with $n$ polynomials. We also exhibit a
family of SPPs for which more than $i \cdot 2^{n-1}$ iterations are needed to
compute $i$ bits. So  $\alpha_\vf \leq n \cdot 2^n$ for every system $\vf$, and
there exists a family of systems for which $2^{n-1} \leq \alpha_\vf$.

Finally, the last result of the paper concerns the geometric interpretation of
Newton's method for SPP equations. We show that, loosely speaking, the Newton
approximants stay within the hypervolume limited by the hypersurfaces
corresponding to each individual equation. This means that a simple
geometric intuition of how Newton's method works, extracted from the
case of 2-dimensional SPPs, is also correct for arbitrary dimensions. As a byproduct
we also obtain a new variant of Newton's method.

\iftechrep{\paragraph{Related work.}}{\paragraph{Related work}}

There is a large body of literature on the convergence speed
of Newton's method for arbitrary systems of differentiable
functions. A comprehensive reference is Ortega and Rheinboldt's
book \cite{OrtegaRheinboldt:book} (see also Chapter 8 of Ortega's
course \cite{Ortega:book} or Chapter~5 of \cite{Kelley:book} for a
brief summary). Several theorems (for instance Theorem 8.1.10 of
\cite{Ortega:book}) prove that the number of valid bits grows
linearly, superlinearly, or even exponentially in the number of
iterations, but only under the hypothesis that $\vF'(\vx)$ is
non-singular everywhere, in a neighborhood of $\mu \vf$, or at
least at the point $\mu \vf$ itself. However, the matrix $\vF'(\mu
\vf)$ can be singular for an SPP, even for the 1-dimensional SPP
$f(X) = 1/2 X^2 + 1/2$.

The general case in which $\vF'(\mu \vf)$ may be singular for the
solution $\mu \vf$ that the method converges to has been thoroughly
studied. In a seminal paper \cite{Reddien:NewtonSingular}, Reddien
shows that under certain conditions, the main ones being that the
kernel of $\vF'(\mu\vf)$ has dimension 1 and that the initial
point is close enough to the solution, Newton's method gains 1 bit
per iteration. Decker and Kelly obtain results for kernels of
arbitrary dimension, but they require a certain linear map
$B(\vX)$ to be non-singular for all $\vx \neq \vzero$
\cite{Decker:NewtonI}. Griewank observes in
\cite{Griewank:NullspaceNewton} that the non-singularity of
$B(\vX)$ is in fact a strong condition which, in particular, can
only be satisfied by kernels of even dimension. He presents a
weaker sufficient condition for linear convergence requiring
$B(\vX)$ to be non-singular only at the initial point $\xs{0}$,
i.e., it only requires to make ``the right guess'' for $\xs{0}$.
Unfortunately, none of these results can be directly applied to
arbitrary SPPs. The possible dimensions of the kernel of
$\vF'(\mu \vf)$ for an SPP $\vf(\vX)$ are to the best of our
knowledge unknown, and deciding this question seems as hard as
those related to the convergence rate\footnote{More precisely,
SPPs with kernels of arbitrary dimension exist, but the cases we
know of can be trivially reduced to SPPs with kernels of
dimension 1.}. Griewank's result does not apply to the decomposed
Newton's method either because the mapping $B(\xs{0})$ is always
singular for $\xs{0} = \vzero$.

Kantorovich's famous theorem (see e.g.\ Theorem 8.2.6 of
\cite{OrtegaRheinboldt:book} and~\cite{PotraPok} for an
improvement) guarantees global convergence and only requires
$\vF'$ to be non-singular at~$\xs{0}$. However, it also requires
to find a Lipschitz constant for $\vF'$ on a suitable region and
some other bounds on $\vF'$. These latter conditions are far too
restrictive for the applications mentioned above. For instance,
the stochastic context-free grammars whose associated SPPs
satisfy Kantorovich's conditions cannot exhibit two productions $X
\rightarrow aYZ$ and $W \rightarrow \varepsilon$ such that $Prob(X
\rightarrow aYZ) \cdot Prob(W \rightarrow \varepsilon) \geq 1/4$.
This class of grammars is too contrived to be of use.

Summarizing, while the convergence of Newton's method for systems
of differentiable functions has been intensely studied, the case of SPPs
does not seem to have been considered yet. The results obtained for other
classes have very limited applicability to SPPs: either they
do not apply at all, or only apply to contrived SPP subclasses.
Moreover, these results only provide information about the
growth rate of the number of accurate bits, but not about the
number itself. For the class of strongly connected SPPs, our
thresholds lead to {\em explicit} lower bounds for the number of accurate
bits depending only on syntactical parameters: the number of equations and the size of
the coefficients. For arbitrary SPPs we prove the existence of
a threshold, while finding explicit lower bounds remains an open problem.

\iftechrep{\paragraph{Structure of the paper.}}{\paragraph{Structure of the paper}}

\S~\ref{sec:prelim} defines SPPs and briefly describes their
applications to stochastic systems. \S~\ref{sec:summary} presents a short summary
of our main theorems. \S~\ref{sec:well-defined} proves some fundamental
properties of Newton's method for SPP equations. \S~\ref{sec:scSPPs}
and \S~\ref{sec:decomposed} contain our results on the convergence speed for
strongly connected and general SPP equations, respectively.
\S~\ref{sec:upper-bounds} shows that the bounds are essentially tight.
\S~\ref{sec:geo} presents our results about the geometrical interpretation
of Newton's method, and \S~\ref{sec:conclusions} contains conclusions.
\iftechrep{}{A slightly extended version of this paper is available as technical report~\cite{EKL10:SIAMtechRep}.}

%% file: sec2-prelim.tex
\newcommand{\Vector}[1]{{\left[#1\right]}}
\newcommand{\nsmin}{\nsc{k}_{\mathit{min}}}

\section{Preliminaries}\label{sec:prelim}
In this section we introduce our notation used in the following and formalize the concepts mentioned in the introduction.

\subsection{Notation}
As usual,  $\R$ and $\N$ denote the set of real, respectively
natural numbers. We assume $0\in \N$. $\R^n$ denotes the set of
$n$-dimensional real valued column vectors and $\Rp^n$ the
subset of vectors with nonnegative components. We use bold
letters for vectors, e.g.\ $\vx\in\R^n$, where we assume that $\vx$
has the components $x_1,\ldots,x_n$. Similarly, the
$i$-th component of a function $\vf:\R^n \to \R^n$ is
denoted by~$f_i$.
We define $\vzero := (0, \ldots, 0)^\top$ and $\vone := (1, \ldots, 1)^\top$ where the superscript ${}^\top$
 indicates the transpose of a vector or a matrix.
%
Let $\norm{\cdot}$ denote some norm on~$\R^n$.
Sometimes we use explicitly the maximum norm~$\norm{\cdot}_\infty$ with $\norm{\vx}_\infty := \max_{1\le i \le n} \abs{x_i}$.

The partial order $\le$ on $\R^n$ is defined as usual
  by setting $\vx \le \vy$ if $x_i \le y_i$ for all $1 \le i \le n$.
Similarly, $\vx < \vy$ if $\vx \le \vy$ and $\vx \neq \vy$.
Finally, we write $\vx \prec \vy$ if $x_i < y_i$ for all
$1 \le i \le n$, i.e., if every component of $\vx$ is smaller than the corresponding
component of~$\vy$.

We use $X_1,\ldots,X_n$ as variable identifiers and arrange them into the vector~$\vX$.
In the following $n$ always denotes the number of variables, i.e., the dimension of~$\vX$.
While $\vx,\vy,\ldots$ denote arbitrary elements in $\R^n$,
 we write $\vX$ if we want to emphasize that a function is given w.r.t.\ these variables.
Hence, $\vf(\vX)$ represents the function itself, whereas $\vf(\vx)$ denotes its value for some~$\vx\in\R^n$.

If $S \subseteq \{1, \ldots, n\}$ is a set of components and $\vx$ a vector,
 then by $\vx_S$ we mean the vector obtained by restricting $\vx$ to the components in $S$.

Let $S \subseteq \{1, \ldots, n\}$ and $\overline{S} = \{1,\ldots,n\} \setminus S$.
Given a function $\vf(\vX)$ and a vector $\vx_S$,
 then $\vf[S / \vx_S]$ is obtained by replacing, for each $s \in S$, each occurrence of~$\vX_s$ by~$\vx_s$
 and removing the $s$-component.
In other words, if $\vf(\vX) = \vf(\vX_S,\vX_{\overline{S}})$
 then $\vf[S / \vx_S](\vy_{\overline{S}}) = \vf_{\overline{S}}(\vx_S,\vy_{\overline{S}})$.
For instance, if $\vf(X_1,X_2) = (X_1 X_2 + \frac{1}{2}, X_2^2 + \frac{1}{5})^\top$,
  then $\vf[\{2\}/\frac{1}{2}] : \R\to\R, X_1 \mapsto \frac{1}{2} X_1 + \frac{1}{2}$.

$\R^{m\times n}$ denotes the set of matrices having $m$ rows and $n$ columns.
The transpose of a vector or matrix is indicated by the superscript $^\top$.
The identity matrix of $\R^{n\times n}$ is denoted by $\Id$.

The {\em formal Neumann series} of $A\in\R^{n\times n}$ is defined by $A^\ast = \sum_{k\in\N} A^k$.
It is well-known that $A^\ast$ exists if and only if the spectral radius of $A$ is less than $1$, i.e.\
$\max\{ \abs{\lambda} \mid \mathbb{C}\ni\lambda \text{ is an eigenvalue of } A \} < 1$.
If $A^\ast$ exists then $A^\ast = ( \Id - A )^{-1}$.

The partial derivative of a function $f(\vX):\R^n\to\R$ w.r.t.\
the variable $X_i$ is denoted by $\pd{f}{X_i}$. 
The gradient $\nabla f$ of $f(\vX)$ is then defined to be the (row) vector
\[
  \nabla f := \left( \pd{f}{X_1}, \ldots, \pd{f}{X_n} \right).
\]
The {\em Jacobian} of a function $\vf(\vX)$ with $\vf : \R^n \to \R^m$ is the matrix $\vf'(\vX)$ defined by
\[
  \vf'(\vX) =
  \begin{pmatrix}
  \pd{f_1}{X_1} & \ldots & \pd{f_1}{X_n} \\
     \vdots     &        & \vdots\\
  \pd{f_m}{X_1} & \ldots & \pd{f_m}{X_n} \\
  \end{pmatrix} \:,
\]
i.e., the $i$-th row of $\vf'$ is the gradient of $f_i$.
\subsection{Systems of Positive Polynomials} \label{subsec:SPPs}
\iftechrep{}{\mbox{}\par}
\begin{definition}
A function $\vf(\vX)$ with $\vf : \Rp^n \to \Rp^n$ is a {\em system of positive polynomials (SPP)},
 if every component $f_i(\vX)$ is a polynomial in the variables $X_1,\ldots,X_n$ with coefficients in $\Rp$.
We call an SPP $\vf(\vX)$ {\em feasible} if $\vy = \vf(\vy)$ for some $\vy \in \Rp^n$.
An SPP is called {\em linear} (resp.\ {\em quadratic}) if all polynomials have degree at most $1$ (resp.\ $2$).
\end{definition}
\begin{fact}
Every SPP $\vf$ is monotone on $\Rp^n$, i.e.\ for $\vzero \le \vx \le \vy$ we have \mbox{$\vf(\vx) \le \vf(\vy)$}.
\end{fact}

We will need the following lemma, a version of Taylor's theorem.
\begin{lemma}[Taylor]\label{lem:taylor}
 Let $\vf$ be an SPP and $\vx, \vu \ge \vzero$.
 Then
  \[ \vf(\vx) + \vf'(\vx)\vu \le \vf(\vx+\vu) \le \vf(\vx) + \vf'(\vx+\vu)\vu\;.
  \]
\end{lemma}
\begin{proof}
 It suffices to show this for a multivariate polynomial $f(\vX)$ with nonnegative coefficients.
 Consider $g(t) = f(\vx + t \vu)$.
 We then have
 \[ f(\vx+\vu) = g(1) = g(0) + \int_{0}^1 g'(s)~ds = f(\vx) + \int_{0}^{1} f'(\vx + s \vu) \vu~ds.\]
 The result follows as $f'(\vx) \le f'(\vx + s \vu ) \le f'(\vx+\vu)$ for $s\in[0,1]$.
\qed
\end{proof}

Since every SPP is continuous, Kleene's fixed-point theorem
(see e.g.~\cite{Kui}) applies.
\begin{theorem}[Kleene's fixed-point theorem]\label{thm:kleene}
Every feasible SPP $\vf$ has a least fixed point $\mu\vf$ in $\Rp^n$
i.e., $\mu\vf = \vf(\mu\vf)$ and, in addition, $\vy = \vf(\vy)$ implies
$\mu\vf \le \vy$.
Moreover, the sequence $(\ks{k}_{\vf})_{k\in\N}$ with $\ks{k}_{\vf} = \vf^k(\vzero)$ (where $\vf^k$ denotes the $k$-fold iteration of~$\vf$)
is monotonically increasing with respect
to $\le$ (i.e.\ $\ks{k}_{\vf} \le \ks{k+1}_{\vf})$
and converges to~$\mu\vf$.
\end{theorem}

In the following we call $(\ks{k}_{\vf})_{k\in\N}$ the {\em Kleene sequence}
of $\vf(\vX)$, and drop the subscript whenever $\vf$ is clear from the context.
Similarly, we sometimes write $\vmu$ instead of $\fix{\vf}$.

An SPP $\vf(\vX)$ is {\em clean} if for all variables $X_i$ there is a $k\in\N$ such that $\ksc{k}_i > 0$.
It is easy to see that we have $\ksc{k}_i = 0$ for all $k\in\N$ if $\ksc{n}_i = 0$.
So we can ``clean'' an SPP~$\vf(\vX)$ in time linear in the size of~$\vf$
 by determining the components $i$ with $\ksc{n}_i = 0$ and removing them.



We will also need the notion of {\em dependence} between variables.
\begin{definition} \label{def:contain}
A polynomial $f(\vX)$ {\em contains} a variable~$X_i$ if $\pd{f}{X_i}(\vX)$ is not the zero-polynomial.
\end{definition}

\begin{definition} \label{def:depend}
 Let $\vf(\vX)$ be an SPP.
 A component~$i$ {\em depends directly} on a component~$k$ if $f_i(\vX)$ contains $X_k$.
 A component~$i$ {\em depends} on~$k$ if either $i$ depends directly on~$k$ or there is a component~$j$ such that
  $i$ depends on~$j$ and $j$ depends on~$k$.
 The components $\{1,\ldots,n\}$ can be partitioned into strongly connected components (SCCs)
  where an SCC $S$ is a maximal set of components such that each component in $S$ depends on each other component in $S$.
 An SCC is called {\em trivial} if it consists of a single component that does not depend on itself.
 An SPP is {\em strongly connected} (short: an {\em scSPP}) if $\{1,\ldots,n\}$ is a non-trivial SCC.
\end{definition}

\subsection{Convergence Speed}

We will analyze the convergence speed of Newton's method.
To this end we need the notion of {\em valid bits}.
\begin{definition} \label{def:valid-bits}
 Let $\vf$ be a feasible SPP.
 A vector~$\vx$ has $i$ {\em valid bits} of the least fixed point~$\mu\vf$ if
  \[
   \frac{\abs{\mu\vf_j - x_j}}{\abs{\mu\vf_j}} \le 2^{-i}
  \]
 for every $1 \le j \le n$.
 Let $(\xs{k})_{k\in\N}$ be a sequence with $\vzero \le \xs{k} \le \mu\vf$.
Then the {\em convergence order} $\beta : \N \to \N$ of the sequence~$(\xs{k})_{k\in\N}$ is defined as follows:
 $\beta(k)$ is the greatest natural number $i$ such that $\xs{k}$ has $i$ valid bits
  (or $\infty$ if such a greatest number does not exist).
 We will always mean the convergence order of the Newton sequence~$(\ns{k})_{k\in\N}$, unless explicitly stated otherwise.
\end{definition}

We say that a sequence has linear, exponential, logarithmic, etc.\ convergence order if the
function $\beta(k)$ grows linearly, exponentially, or logarithmically in $k$, respectively.

\medskip

\begin{remark}
Our definition of convergence order differs from the one commonly
used in numerical analysis (see e.g.~\cite{OrtegaRheinboldt:book}), where
``quadratic convergence'' or ``Q-quadratic convergence''
means that the error $e'$ of the new approximant (its distance to the least fixed point
according to some norm) is bounded by $c \cdot e^2$, where $e$ is the error of the old
approximant and $c > 0$ is some constant. We consider our notion more natural from
a computational point of view, since it directly relates the number of iterations to
the accuracy of the approximation. Notice that ``quadratic convergence'' implies  exponential
convergence order in the sense of Definition~\ref{def:valid-bits}. In the following we
avoid the notion of ``quadratic convergence''.
\end{remark}

\subsection{Stochastic Models} \label{subsec:stochastic-models}
As mentioned in the introduction, several problems concerning stochastic models can be reduced to problems
 about the least fixed point~$\fix{\vf}$ of an SPP~$\vf$.
In these cases, $\fix{\vf}$ is a vector of probabilities, and so $\mu\vf \le \vone$.

\subsubsection{Probabilistic Pushdown Automata}
Our study of SPPs was initially motivated by the verification of probabilistic pushdown automata.
A \emph{probabilistic pushdown automaton (pPDA)} is a tuple $\pPDA = (Q,\Gamma,\delta,\Prob)$
 where $Q$ is a finite set of \emph{control states}, $\Gamma$ is a finite \emph{stack alphabet},
 $\delta \subseteq Q \times \Gamma \times Q \times \Gamma^*$ is a  finite \emph{transition relation}
 (we write $pX \tran{} q \alpha$ instead of $(p,X,q,\alpha) \in \delta$),
 and $\Prob$ is a function which to each transition $pX \tran{} q\alpha$ assigns
 its probability $\Prob(pX \tran{} q\alpha) \in (0,1]$ so that for all $p \in Q$ and $X \in \Gamma$ we have
 $\sum_{pX \tran{} q\alpha} \Prob(pX \tran{} q\alpha) = 1$.
We write $pX \tran{x} q\alpha$ instead of $\Prob(pX \tran{} q\alpha) = x$.
A {\em configuration} of $\pPDA$ is a pair $qw$, where $q$ is a control state and $w \in \Gamma^*$ is a {\em stack content}.
A pPDA $\pPDA$ naturally induces a possibly infinite Markov chain with the
 configurations as states and transitions given by:
$pX \beta \tran{x} q \alpha \beta$ for every $\beta \in \Gamma^*$ if{}f $pX \tran{x} q\alpha$.
We assume w.l.o.g.\ that if $pX \tran{x} q\alpha$ is a transition then $|\alpha|\leq 2$.

pPDAs and the equivalent model of recursive Markov chains have been very thoroughly studied
 \cite{EKM:prob-PDA-PCTL,BKS:pPDA-temporal,EYstacs05Extended,EY:RMC-LTL-complexity,EKM:prob-PDA-expectations,EY:RMC-LTL-QUEST,EY:RMC-RMDP}.
This work has shown that the key to the analysis of pPDAs are the {\em termination probabilities} $\Pro{pXq}$,
 where $p$ and $q$ are states, and $X$ is a stack letter, defined as follows
 (see e.g.\ \cite{EKM:prob-PDA-PCTL} for a more formal definition):
$\Pro{pXq}$ is the probability that, starting at the configuration $pX$,
 the pPDA eventually reaches the configuration $q\varepsilon$ (empty stack).
It is not difficult to show that the vector of these probabilities is the least solution of the SPP equation system
 containing the equation
$$
      \Var{pXq}  =  \sum_{pX \tran{x} rYZ}
                  x \cdot \sum_{t \in Q} \Var{rYt} \cdot \Var{tZq}
                \quad + \quad \sum_{pX \tran{x} rY}
                  x \cdot \Var{rYq}
                \quad + \quad  \sum_{pX \tran{x} q \varepsilon} x
$$
 for each triple $(p,X,q)$.
Call this quadratic SPP the {\em termination SPP} of the pPDA
 (we assume that termination SPPs are clean, and it is easy to see that they are always feasible).

%
%


\subsubsection{Strict pPDAs and Back-Button Processes}

A pPDA is {\em strict} if for all $pX \in Q  \times \Gamma$ and all $q \in Q$
 the transition relation contains a pop-rule $pX \tran{x} q \epsilon$ for some $x > 0$.
Essentially, strict pPDAs model programs in which every procedure has at least one terminating execution that
 does not call any other procedure.
The termination SPP of a strict pPDA
 satisfies $\vf(\vzero) \gg \vzero$.


In~\cite{FaginetalSTOC,Faginetal} a class of stochastic processes is introduced to model the behavior of web-surfers
 who from the current webpage $A$ can decide either to follow a link to another page, say $B$, with probability $\ell_{AB}$,
 or to press the ``back button'' with nonzero probability $b_A$.
These back-button processes correspond to a very special class of strict pPDAs having one single control state
 (which in the following we omit),
 and rules of the form $A \tran{b_A} \varepsilon$ (press the back button from $A$)
 or $A \tran{\ell_{AB}} BA$ (follow the link from $A$ to $B$, remembering $A$ as destination of pressing the back button at~$B$).
The termination probabilities are given by an SPP equation system containing the equation
\[
  \Var{A} \quad = \quad b_A + \displaystyle{\sum_{A \tran{\ell_{AB}} BA} \ell_{AB}\Var{B}\Var{A}}
          \quad = \quad b_A + \Var{A} \displaystyle{\sum_{A \tran{\ell_{AB}} BA}} \ell_{AB} \Var{B}
\]
 for every webpage $A$.
In \cite{FaginetalSTOC,Faginetal} those termination probabilities are called {\em revocation} probabilities.
The revocation probability of a page~$A$ is the probability that, when currently visiting webpage~$A$
 and having $H_0 H_1 \ldots H_{n-1} H_n$ as the browser history of previously visited pages,
 then during subsequent surfing from~$A$ the random user eventually returns to webpage $H_n$ with
 $H_0 H_1 \ldots H_{n-1}$ as the remaining browser history.
%
%
%

\begin{example} \label{ex:back-button}
 Consider the following equation system.
 \[
  \begin{pmatrix}
   X_1 \\ X_2 \\ X_3
  \end{pmatrix}
  =
  \begin{pmatrix}
   0.4 X_2 X_1 + 0.6 \\
   0.3 X_1 X_2 + 0.4 X_3 X_2 + 0.3 \\
   0.3 X_1 X_3 + 0.7
  \end{pmatrix}
 \]
 The least solution of the system gives the revocation probabilities of a back-button process with three web-pages.
 For instance, if the surfer is at page 2 it can choose between following links to pages 1 and 3 with probabilities 0.3 and 0.4,
  respectively, or pressing the back button with probability 0.3.
\end{example}

\section{Newton's Method and an Overview of Our Results}
\label{sec:summary}

In order to approximate the least fixed point $\mu\vf$ of an SPP~$\vf$ we employ Newton's method:
\begin{definition} \label{def:newton-operator}
 Let $\vf$ be a clean and feasible SPP.
 The Newton operator~$\Ne_{\vf}$ is defined as follows:
 \[
  \Ne_{\vf}(\vX) := \vX + \left(\Id - \vf'(\vX)\right)^{-1} ( \vf(\vX) - \vX  )
 \]
 The sequence $(\ns{k}_\vf)_{k\in\N}$ with $\ns{k}_\vf = \Ne_{\vf}^k(\vzero)$ (where $\Ne_\vf^k$ denotes the $k$-fold iteration of~$\Ne_\vf$)
  is called {\em Newton sequence}.
 We drop the subscript of $\Ne_{\vf}$ and~$\ns{k}_\vf$ when $\vf$ is understood.
\end{definition}

The main results of this paper concern the application of Newton's method to
SPPs. We summarize them in this section.

{\bf Theorem~\ref{thm:well-defined}} states that the Newton sequence $(\ns{k})_{k\in\N}$ is well-defined
(i.e., the inverse matrices  $\left(\Id - \vf'(\ns{k})\right)^{-1}$ exist for every $k \in \N$),
monotonically increasing and bounded from above by $\mu\vf$ (i.e.\ $\ns{k} \le \vf(\ns{k}) \le \ns{k+1} \le \mu\vf$), and converges to $\mu\vf$.
This theorem generalizes the result of Etessami and Yannakakis in~\cite{EYstacs05Extended} to arbitrary
clean and feasible SPPs and to the ordinary Newton's method.

For more quantitative results on the convergence speed it is convenient to focus on
quadratic SPPs.
{\bf Theorem~\ref{thm:reduction-quadratic}}
 shows that any clean and feasible SPP can be syntactically transformed into a quadratic SPP without changing
the least fixed point and without accelerating Newton's method.
This means, one can perform Newton's method on the original (possibly non-quadratic) SPP and convergence will be at least as
fast as for the corresponding quadratic SPP.

For quadratic $n$-dimensional SPPs,
 one iteration of Newton's method involves $\bigo(n^3)$ arithmetical operations and $\bigo(n^3)$ operations in the Blum-Shub-Smale model.
Hence, a bound on the number of iterations needed to compute a given number of valid bits immediately leads to a bound on the number of operations.
In \S~\ref{sec:scSPPs} we obtain such bounds for
{\em strongly connected} quadratic SPPs. We give different thresholds
for the number of iterations, and show that when any of these thresholds is reached,
Newton's method gains at least one valid bit for each iteration. More
precisely, {\bf Theorem \ref{thm:estimate-cor}} states the following. Let $\vf$ be a quadratic, clean and feasible scSPP,
let $\mumin$ and $\mumax$ be the minimal and maximal component of $\fix{\vf}$, respectively,
and let the coefficients of~$\vf$ be given as ratios of $m$-bit integers.
Then $\beta( k_\vf + i) \ge i$ holds for all $i\in\N$ and for any of the following choices of~$k_\vf$:
\begin{itemize}
  \item[1.] $\displaystyle 4mn + \lceil 3n \max\{ 0, - \log \mumin \} \rceil$;
  \item[2.] $\displaystyle 4mn2^n$;
  \item[3.] $\displaystyle 7mn$ if $\vf$ satisfies $\vf(\vzero) \succ \vzero$;
  \item[4.] $\displaystyle 2m(n + 1)$ if $\vf$ satisfies  both $\vf(\vzero) \succ \vzero$ and $\mumax \le 1$.
\end{itemize}

We further show that Newton iteration can also be used to
obtain a sequence of {\em upper} approximations of~$\mu\vf$.
Those upper approximations converge to~$\mu\vf$, asymptotically as fast as the Newton sequence.
More precisely, {\bf Theorem~\ref{thm:proximity-2}} states the following:
Let $\vf$ be a quadratic, clean and feasible scSPP, let $\cmin$ be the smallest nonzero coefficient of~$\vf$, and
let $\mumin$ be the minimal component of~$\mu\vf$. Further,
for all Newton approximants $\ns{k}$ with $\ns{k} \succ \vzero$, let $\nsmin$ be the
smallest coefficient of~$\ns{k}$. Then
   \[
    \ns{k} \le \mu\vf \le \ns{k} + \Vector{\frac{\norm{\ns{k} - \ns{k-1}}_\infty}{\left( \cmin \cdot \min\{\nsmin, 1\} \right)^n}}
   \]
\noindent  where $\Vector{s}$ denotes the vector $\vx$ with $x_j = s$ for all $1 \le j \le n$.

In \S~\ref{sec:decomposed} we turn to general (not necessarily strongly connected) clean and feasible SPPs.
We show in {\bf Theorem \ref{thm:conv-speed-general}} that Newton's method still converges linearly.
Formally, the theorem proves that for every quadratic, clean and feasible SPP $\vf$, there is a
threshold $k_\vf \in \Nat$ and $\alpha_\vf > 0$ such that $\beta(k_\vf + \alpha_\vf \cdot i) \ge i$ for all $i\in\N$.
With respect to the threshold our proof is purely existential and does not provide any bound for $k_\vf$.
For $\alpha_\vf$ we show an upper bound of $n \cdot 2^n$, i.e., asymptotically at most
$n \cdot 2^n$ extra iterations are needed in order to get one new valid bit.
\S~\ref{sec:upper-bounds} exhibits a family of SPPs in which one new bit requires at least $2^{n-1}$ iterations, implying
that the bound on $\alpha_\vf$ is essentially tight.

Finally, \S~\ref{sec:geo} gives a geometrical interpretation of
Newton's method on quadratic SPP equations. Let $R$ be the region bounded by
the coordinate axes and by the quadrics corresponding to the individual equations.
{\bf Theorem \ref{cor:R-newton-kleene}} shows that
all Kleene and Newton approximations lie within $R$, i.e.:
$\ns{i}, \ks{i} \in R$ for every $i\in\N$.

%% file: sec3-fundamental.tex
\section{Fundamental Properties of Newton's Method}\label{sec:well-defined}

\subsection{Effectiveness} \label{sub:effectiveness}

Etessami and Yannakakis \cite{EYstacs05Extended} suggested to use Newton's method for SPPs.
More precisely, they showed that the sequence obtained by applying Newton's method to the equation system $\vX = \vf(\vX)$
 converges to~$\mu\vf$ as long as $\vf$ is strongly connected.
We extend their result to arbitrary SPPs,
 thereby reusing and extending several proofs of~\cite{EYstacs05Extended}.

In Definition~\ref{def:newton-operator} we defined the Newton operator $\Ne_{\vf}$ and the associated Newton sequence $(\ns{k})_{k\in\N}$.
In this section we prove the following fundamental theorem on the Newton sequence.

\begin{theorem} \label{thm:well-defined}
 Let $\vf$ be a clean and feasible SPP.
 Let the Newton operator $\Ne_{\vf}$ be defined as in Definition~\ref{def:newton-operator}:
 \[
 \Ne_{\vf}(\vX) := \vX + (\Id-\vf'(\vX))^{-1} ( \vf(\vX) - \vX )
 \]
 \begin{itemize}
  \item[1.] Then the Newton sequence $(\ns{k})_{k\in\N}$ with $\ns{k} = \Ne_{\vf}^k(\vzero)$
   is well-defined (i.e., the matrix inverses exist),
   monotonically increasing, bounded from above by $\mu\vf$ (i.e.\ $\ns{k} \le \vf(\ns{k}) \le \ns{k+1} \le \mu\vf$),
   and converges to $\mu\vf$.
  \item[2.]
   We have $(\Id-\vf'(\ns{k}))^{-1} = \vf'(\ns{k})^*$ for all $k\in\N$.
   \\ We also have $(\Id-\vf'(\vx))^{-1} = \vf'(\vx)^*$ for all $\vx \prec \mu\vf$.
 \end{itemize}
\end{theorem}

The proof of Theorem~\ref{thm:well-defined} consists of three steps.
In the first proof step we study a sequence generated by a somewhat weaker version of the Newton operator and obtain the following:
\begin{proposition}
\label{prop:newton-stacs07}
 Let $\vf$ be a feasible SPP.
 Let the operator $\hNe_{\vf}$ be defined as follows:
 \[
 \hNe_{\vf}(\vX) := \vX + \sum_{d=0}^\infty \left(\vf'(\vX)^d ( \vf(\vX) - \vX ) \right) \:.
 \]
 Then the sequence $(\ns{k})_{k\in\N}$ with $\ns{k} := \hNe_{\vf}^k(\vzero)$
  is monotonically increasing, bounded from above by $\mu\vf$ (i.e.\ $\ns{k} \le \vf(\ns{k}) \le \ns{k+1} \le \mu\vf$)
  and converges to $\mu\vf$.
\end{proposition}

In a second proof step, we show another intermediary proposition, namely that the star of the Jacobian matrix $\vf'$ converges
 for all Newton approximants:
\begin{proposition} \label{prop:no-infty-entries}
 Let $\vf$ be clean and feasible.
 Then the matrix series $\vf'(\ns{k})^* := \Id + \vf'(\ns{k}) + \vf'(\ns{k})^2 + \cdots$
  converges in~$\Rp$ for all Newton approximants~$\ns{k}$,
  i.e., there are no $\infty$ entries.
\end{proposition}

In the third and final step we show that Propositions \ref{prop:newton-stacs07} and~\ref{prop:no-infty-entries} imply Theorem~\ref{thm:well-defined}.

\iftechrep{\subsubsection{First Step.}}{\subsubsection{First Step}}

For the first proof step (i.e., the proof of Proposition~\ref{prop:newton-stacs07}) we will need the following generalization of Taylor's theorem.
\begin{lemma}\label{lem:gen-taylor}
 Let $\vf$ be an SPP, $d \in \N$, and $\vzero \le \vu$, and $\vzero \le \vx \le \vf(\vx)$. Then
   \[ \vf^d(\vx + \vu) \ge \vf^d(\vx) + \vf'(\vx)^d \vu \:.\]
 In particular, by setting $\vu := \vf(\vx) - \vx$ we get
  \[ \vf^{d+1}(\vx) - \vf^{d}(\vx) \ge \vf'(\vx)^d (\vf(\vx) - \vx) \:.\]
\end{lemma}
\begin{proof}
 By induction on $d$.
 For $d=0$ the statement is trivial.
 Let $d \ge 0$.
 Then, by Taylor's theorem (Lemma~\ref{lem:taylor}), we have:
  \begin{align*}
    \vf^{d+1}(\vx + \vu)
    &  =  \vf(\vf^d(\vx + \vu)) \\
    & \ge \vf(\vf^d(\vx) + \vf'(\vx)^d \vu)                 && \text{(induction hypothesis)} \\
    & \ge \vf^{d+1}(\vx) + \vf'(\vf^d(\vx)) \vf'(\vx)^d \vu && \text{(Lemma~\ref{lem:taylor})} \\
    & \ge \vf^{d+1}(\vx) + \vf'(\vx)^{d+1} \vu              && \text{($\vf^d(\vx) \ge \vx$)}
  \end{align*}
\end{proof}

Lemma~\ref{lem:gen-taylor} can be used to prove the following.
\begin{lemma} \label{lem:step-stays-below-mu}
 Let $\vf$ be a feasible SPP.
 Let $\vzero \le \vx \le \mu\vf$ and $\vx \le \vf(\vx)$.
 Then
  \[
   \vx + \sum_{d=0}^\infty \left(\vf'(\vx)^d ( \vf(\vx) - \vx ) \right) \le \mu\vf\;.
  \]
\end{lemma}
\begin{proof}
 Observe that
  \begin{equation} \label{eq:proof-lem-step-stays-below-mu}
   \lim_{d\to\infty} \vf^d(\vx) = \mu\vf
  \end{equation}
  because $\vzero \le \vx \le \mu\vf$ implies $\vf^d(\vzero) \le \vf^d(\vx) \le \mu\vf$ and
  and as $(\vf^d(\vzero))_{d\in\N}$ converges to~$\mu\vf$ by Theorem~\ref{thm:kleene}, so does $(\vf^d(\vx))_{d\in\N}$.
 We have:
 \begin{align*}
  \vx + \sum_{d=0}^\infty \left( \vf'(\vx)^d (\vf(\vx) - \vx) \right)
    & \le \vx + \sum_{d=0}^\infty \left( \vf^{d+1}(\vx) - \vf^d(\vx) \right) && \text{(Lemma~\ref{lem:gen-taylor})} \\
    & =   \lim_{d \to \infty} \vf^d(\vx) \\
    & =   \mu\vf  && \text{(by~\eqref{eq:proof-lem-step-stays-below-mu})}
 \end{align*}
\qed
\end{proof}

Now we can prove Proposition~\ref{prop:newton-stacs07}.
\begin{proof}[of Proposition~\ref{prop:newton-stacs07}]
 First we prove the following inequality by induction on~$k$:
 \[
  \ns{k} \le \vf(\ns{k}) 
 \]
 The induction base ($k=0$) is easy.
 For the step, let $k \ge 0$.
 Then
 \begin{align*}
  \ns{k+1} & =   \ns{k} + \sum_{d=0}^\infty \left( \vf'(\ns{k})^d (\vf(\ns{k}) - \ns{k}) \right) \\
           & =   \vf(\ns{k}) + \sum_{d=1}^\infty \left( \vf'(\ns{k})^d (\vf(\ns{k}) - \ns{k}) \right) \\
           & \le \vf(\ns{k}) + \vf'(\ns{k}) \sum_{d=0}^\infty \left( \vf'(\ns{k})^d (\vf(\ns{k}) - \ns{k}) \right) \\
           & \le \vf\left(\ns{k} + \sum_{d=0}^\infty \left( \vf'(\ns{k})^d (\vf(\ns{k}) - \ns{k}) \right)\right)
                   && \text{(Lemma~\ref{lem:taylor})} \\
           & =   \vf(\ns{k+1}) \;.
 \end{align*}

 Now, the inequality $\ns{k} \le \mu\vf$ follows from Lemma~\ref{lem:step-stays-below-mu}
  by means of a straightforward induction proof.
 Hence, it follows $\vf(\ns{k}) \le \vf(\mu\vf) = \mu\vf$.
 Further we have
  \begin{equation} \label{eq:proof-stacs07-3}
   \begin{split}
    \vf(\ns{k}) & =   \ns{k} + (\vf(\ns{k}) - \ns{k}) \\
                & \le \ns{k} + \sum_{d=0}^\infty \left( \vf'(\ns{k})^d (\vf(\ns{k}) - \ns{k}) \right) = \ns{k+1}\;.
   \end{split}
  \end{equation}
 So it remains to show that $(\ns{k})_{k\in\N}$ converges to~$\mu\vf$.
 As we have already shown $\ns{k} \le \mu\vf$ it suffices to show that $\ks{k} \le \ns{k}$
  because $(\ks{k})_{k\in\N}$ converges to~$\mu\vf$ by Theorem~\ref{thm:kleene}.
 We proceed by induction on~$k$.
 The induction base ($k=0$) is easy.
 For the step, let $k \ge 0$.
 Then
 \begin{align*}
  \ks{k+1} & =   \vf(\ks{k}) \\
           & \le \vf(\ns{k}) && \text{(induction hypothesis)} \\
           & \le \ns{k+1}    && \text{(by~\eqref{eq:proof-stacs07-3})}
 \end{align*}
 This completes the proof of Proposition~\ref{prop:newton-stacs07} and, hence, the first step towards the proof of Theorem~\ref{thm:well-defined}.
\qed
\end{proof}

\iftechrep{\subsubsection{Second Step.}}{\subsubsection{Second Step}}

For the second proof step (i.e., the proof of Proposition~\ref{prop:no-infty-entries})
 it is convenient to move to the {\em extended reals}~$\Rinf$, i.e., we extend~$\Rp$ by an element~$\infty$
 such that addition satisfies $a + \infty = \infty + a = \infty$ for all $a \in \Rp$
 and multiplication satisfies $0 \cdot \infty = \infty \cdot 0 = 0$ and $a \cdot \infty = \infty \cdot a = \infty$ for all $a \in\Rp$.
In $\Rinf$, one can rewrite $\hNe(\ns{k}) = \ns{k} + \sum_{d=0}^\infty \left( \vf'(\ns{k})^d ( \vf(\ns{k}) - \ns{k} ) \right)$
  as $\ns{k} + \vf'(\ns{k})^* ( \vf(\ns{k}) - \ns{k} )$.
Notice that Proposition~\ref{prop:no-infty-entries} does not follow trivially from Proposition~\ref{prop:newton-stacs07},
 because $\infty$ entries of~$\vf'(\ns{k})^*$ could be cancelled out by matching $0$ entries of $\vf(\ns{k})- \ns{k}$.

For the proof of Proposition~\ref{prop:no-infty-entries} we need several lemmata.
%
The following lemma assures that a starred matrix has an $\infty$ entry if and only if it has an $\infty$ entry on the diagonal.
\begin{lemma}\label{lem:consider-only-diag}
 Let $A = (a_{ij}) \in \Rp^{n \times n}$.
 Let $A^*$ have an $\infty$ entry.
 Then $A^*$ also has an $\infty$ entry on the diagonal, i.e., $\left[A^*\right]_{ii} = \infty$ for some $1 \le i \le n$.
\end{lemma}
\begin{proof}
 By induction on $n$.
 The base case $n=1$ is clear.
 For $n > 1$ assume w.l.o.g.\ that $\left[A^*\right]_{1n} = \infty$.
 We have
 \begin{equation}\label{eq:path-star}
  \left[A^*\right]_{1n} = \left[A^*\right]_{11} \sum_{j=2}^n a_{1j} \left[A_{[2..n,2..n]}^*\right]_{jn} \:,
 \end{equation}
 where by $A_{[2..n,2..n]}$ we mean the square matrix obtained from $A$ by erasing the first row and the first column.
 To see why \eqref{eq:path-star} holds,
  think of $\left[A^*\right]_{1n}$ as the sum of weights of paths from $1$ to $n$ in the complete graph over the vertices $\{1,\ldots,n\}$.
 The weight of a path $P$ is the product of the weight of $P$'s edges,
  and $a_{i_1i_2}$ is the weight of the edge from $i_1$ to~$i_2$.
 Each path $P$ from $1$ to $n$ can be divided into two subpaths $P_1,P_2$ as follows.
 The second subpath $P_2$ is the suffix of $P$ leading from $1$ to $n$ and not returning to $1$.
 The first subpath $P_1$, possibly empty, is chosen such that $P = P_1P_2$.
 Now, the sum of weights of all possible $P_1$ equals $\left[A^*\right]_{11}$,
  and the sum of weights of all possible $P_2$ equals $\sum_{j=2}^n a_{1j} \left[(A_{[2..n,2..n]})^*\right]_{jn}$.
 So \eqref{eq:path-star} holds.

 As $\left[A^*\right]_{1n} = \infty$, it follows that either $\left[A^*\right]_{11}$ or some $\left[(A_{[2..n,2..n]})^*\right]_{jn}$ equals $\infty$.
 In the first case, we are done.
 In the second case, by induction, there is an $i$ such that $\left[(A_{[2..n,2..n]})^*\right]_{ii} = \infty$.
 But then also $\left[A^*\right]_{ii} = \infty$,
  because every entry of $\left[(A_{[2..n,2..n]})\right]^*$ is less than or equal to the corresponding entry of $A^*$.
\qed
\end{proof}

The following lemma treats the case that $\vf$ is strongly connected (cf.~\cite{EYstacs05Extended}).
\begin{lemma} \label{lem:no-infty-in-SCC}
 Let $\vf$ be clean, feasible and non-trivially strongly connected.
 Let $\vzero \le \vx \prec \mu\vf$.
 Then $\vf'(\vx)^*$ does not have $\infty$ as an entry.
\end{lemma}
\begin{proof}
 By Theorem~\ref{thm:kleene} the Kleene sequence $(\ks{i})_{i\in\N}$ converges to~$\mu\vf$.
 Furthermore, $\ks{i} \prec \mu\vf$ holds for all $i$, because, as every component depends non-trivially on itself,
  any increase in any component results in an increase of the same component in a later Kleene approximant.
 So, we can choose a Kleene approximant $\vy = \ks{i}$ such that $\vx \le \vy \prec \mu\vf$.
 Notice that $\vy \le \vf(\vy)$.
 By monotonicity of $\vf'$ it suffices to show that $\vf'(\vy)^*$ does not have $\infty$ as an entry.
 By Lemma~\ref{lem:gen-taylor} (taking $\vx := \vy$ and $\vu := \mu\vf - \vy$) we have
 \[
  \vf'(\vy)^d (\mu \vf - \vy) \le \mu \vf - \vf^d(\vy) \;.
 \]
 As $d \to \infty$, the right hand side converges to $\vzero$, because, by Kleene's theorem, $\vf^d(\vy)$ converges to~$\mu\vf$.
 So the left hand side also converges to~$\vzero$.
 Since $\mu\vf - \vy \succ \vzero$, every entry of $\vf'(\vy)^d$ must converge to $\vzero$.
 Then, by standard facts about matrices (see e.g.\ \cite{Lancaster:book}), the spectral radius of~$\vf'(\vy)$ is less than~$1$,
  i.e., $|\lambda| < 1$ for all eigenvalues $\lambda$ of~$\vf'(\vy)$.
 This, in turn, implies that the series $\vf'(\vy)^* = \Id + \vf'(\vy) + \vf'(\vy)^2 + \cdots$ converges in~$\Rp$,
  see \cite{Lancaster:book}, page~531.
 In other words, $\vf'(\vy)^*$ and hence $\vf'(\vx)^*$ do not have $\infty$ as an entry.
\qed
\end{proof}

The following lemma states that Newton's method can only terminate in a component~$s$
 after certain other components~$\ell$ have reached~$\mu\vf_\ell$.

\begin{lemma}\label{lem:EY-not-terminating}
 Let $1 \le s,\ell \le n$.
 Let the term $\left[\vf'(\vX)^*\right]_{ss}$ contain the variable $X_\ell$.
 Let $\vzero \le \vx \le \vf(\vx) \le \fix{\vf}$ and $x_s < \fix{\vf}_s$ and $x_\ell < \fix{\vf}_\ell$.
 Then $\hNe(\vx)_s < \fix{\vf}_s$.
\end{lemma}
\begin{proof}
 This proof follows closely a proof of~\cite{EYstacs05Extended}.
 Let $d \ge 0$ such that $\left[\vf'(\vX)^{d}\right]_{ss}$ contains~$X_\ell$.
 Let $m' \ge 0$ such that $\vf^{m'}(\vx) \succ \vzero$ and $\vf^{m'}(\vx)_\ell > x_\ell$.
 Such an $m'$ exists because with Kleene's theorem the sequence $(\vf^k(\vx))_{k\in\Nat}$ converges to~$\fix{\vf}$.
 Notice that our choice of~$m'$ guarantees $\left[\vf'(\vf^{m'}(\vx))^d\right]_{ss} > \left[\vf'(\vx)^d\right]_{ss}$.

 Now choose $m \ge m'$ such that $\vf^{m+1}(\vx)_s > \vf^m(\vx)_s$.
 Such an $m$ exists because the sequence $(\vf^k(\vx)_s)_{k\in\Nat}$ never reaches $\fix{\vf}_s$.
 This is because $s$ depends on itself (since $\left[\vf'(\vX)^*\right]_{ss}$ is not constant~$0$),
  and so every increase of the $s$-component results in
  an increase of the $s$-component in some later iteration of the Kleene sequence.

 Now we have
 \[
  \begin{array}{ll@{\qquad}l}
   \multicolumn{2}{l}{\vf^{d+m+1}(\vx) - \vf^{d+m}(\vx)} \\
    \quad & \ge \vf'(\vf^m(\vx))^d (\vf^{m+1}(\vx) - \vf^m(\vx)) & \text{(Lemma~\ref{lem:gen-taylor})}\\
    & \ge^* \vf'(\vx)^d (\vf^{m+1}(\vx) - \vf^m(\vx))\\
    & \ge \vf'(\vx)^d \vf'(\vx)^m (\vf(\vx) - \vx) & \text{(Lemma~\ref{lem:gen-taylor})}\\
    &  =  \vf'(\vx)^{d+m} (\vf(\vx) - \vx) \:.\\
  \end{array}
 \]

 The inequality marked with $*$ is strict in the $s$-component, due to the choice of $d$ and $m$ above.
 So, with $b = d+m$ we have:
  \begin{equation}\label{eq:ineq-b}
   (\vf^{b+1}(\vx) - \vf^b(\vx))_s > (\vf'(\vx)^b (\vf(\vx) - \vx))_s
  \end{equation}
 Again by Lemma~\ref{lem:gen-taylor}, inequality~\eqref{eq:ineq-b} holds for all $b\in\N$, but with $\ge$ instead of~$>$.
 Therefore:
  \[
   \begin{array}{ll@{\qquad}l}
    \fix{\vf}_s
     & = \bigl( \vx + \sum_{i=0}^\infty (\vf^{i+1}(\vx) - \vf^i(\vx)) \bigr)_s & \text{(Kleene)} \\
     & > \bigl( \vx + \vf'(\vx)^* (\vf(\vx) - \vx) \bigr)_s                   & \text{(inequality~\eqref{eq:ineq-b})} \\
     & = \bigl( \hNe(\vx) \bigr)_s                                               &
   \end{array}
  \]
\qed
\end{proof}

Now we are ready to prove Proposition~\ref{prop:no-infty-entries}.

\begin{proof}[of Proposition~\ref{prop:no-infty-entries}]
 Using Lemma~\ref{lem:consider-only-diag} it is enough to show that \mbox{$\left[\vf'(\ns{k})^*\right]_{ss} \ne \infty$} for all $s$.
 If the $s$-component constitutes a trivial SCC, then $\left[\vf'(\ns{k})^*\right]_{ss} = 0 \ne \infty$.
 So we can assume in the following that the $s$-component belongs to a non-trivial SCC, say~$S$.
 Let $\vX_L$ be the set of variables contained by the term $\left[\vf'(\vX)^*\right]_{ss}$.
 For any $t \in S$ we have $\left[\vf'(\vX)^*\right]_{ss} \ge \left[\vf'(\vX)^*\right]_{st} \left[\vf'(\vX)^*\right]_{tt} \left[\vf'(\vX)^*\right]_{ts}$.
 Neither $\left[\vf'(\vX)^*\right]_{st}$ nor $\left[\vf'(\vX)^*\right]_{ts}$ is constant zero, because $S$ is non-trivial.
 Therefore, $\left[\vf'(\vX)^*\right]_{ss}$ contains all variables that $\left[\vf'(\vX)^*\right]_{tt}$ contains, and vice versa, for all $t \in S$.
 So, $\vX_L$ is, for all $t \in S$, exactly the set of variables contained by $\left[\vf'(\vX)^*\right]_{tt}$.

 We distinguish two cases.

 \vspace{2mm}

 \noindent {\bf Case 1:}
  There is a component $\ell \in L$ such that the sequence $(\nsc{k}_\ell)_{k\in\N}$ does not terminate,
   i.e., $\nsc{k}_\ell < \mu\vf_\ell$ holds for all $k$.
  Then, by Lemma~\ref{lem:EY-not-terminating}, the sequence $(\nsc{k}_s)_{k\in\N}$ cannot reach $\mu\vf_s$ either.
  In fact, we have $\ns{k}_S \prec \mu\vf_S$.
  Let $M$ denote the set of those components that the $S$-components depend on, but do not depend on~$S$.
  In other words, $M$ contains the components that are ``lower'' in the DAG of SCCs than~$S$.
  Define $\vg(\vX_S)$ := $\vf_S(\vX) [M / \fix{\vf}_M]$.
  Then $\vg(\vX_S)$ is an scSPP with $\fix{\vg} = \fix{\vf}_S$.
  As $\ns{k}_S \prec \fix{\vg}$, Lemma~\ref{lem:no-infty-in-SCC} is applicable,
   so $\vg'(\ns{k}_S)^*$ does not have $\infty$ as an entry.
  With $\left[\vf'(\ns{k})^*\right]_{SS} \le \vg'(\ns{k}_S)^*$, we get $\left[\vf'(\ns{k})^*\right]_{ss} < \infty$, as desired.

 \vspace{3mm}

 \noindent {\bf Case 2:}
  For all components $\ell \in L$ the sequence $(\nsc{k}_\ell)_{k\in\N}$ terminates.
  Let $i\in\N$ be the least number such that $\nsc{i}_\ell = \mu\vf_\ell$ holds for all $\ell \in L$.
  By Lemma~\ref{lem:EY-not-terminating} we have $\nsc{i}_s < \mu\vf_s$.
  But as, according to Proposition~\ref{prop:newton-stacs07}, $(\nsc{k}_s)_{k\in\N}$ converges to~$\mu\vf_s$,
   there must exist a $j \ge i$ such that $0 < \left(\vf'(\ns{j})^* (\vf(\ns{j}) - \ns{j})\right)_s < \infty$.
  So there is a component $u$ with $0 < \left[\vf'(\ns{j})^*\right]_{su}(\vf(\ns{j}) - \ns{j})_u < \infty$.
  This implies $0 < \left[\vf'(\ns{j})^*\right]_{su} < \infty$, therefore also $\left[\vf'(\ns{j})^*\right]_{ss} < \infty$.
  By monotonicity of $\vf'$, we have $\left[\vf'(\ns{k})^*\right]_{ss} \le \left[\vf'(\ns{j})^*\right]_{ss} < \infty$ for all $k \le j$.
  On the other hand, since $\left[\vf'(\vX)^*\right]_{ss}$ contains only $L$-variables and $\ns{k}_L = \mu\vf_L$ holds for all $k \ge j$,
   we also have $\left[\vf'(\ns{k})^*\right]_{ss} = \left[\vf'(\ns{j})^*\right]_{ss} < \infty$ for all $k \ge j$.
\qed
\end{proof}

This completes the second intermediary step towards the proof of Theorem~\ref{thm:well-defined}.

\iftechrep{\subsubsection{Third and Final Step.}}{\subsubsection{Third and Final Step}}

Now we can use Proposition~\ref{prop:newton-stacs07} and Proposition~\ref{prop:no-infty-entries} to complete the proof of Theorem~\ref{thm:well-defined}.

\begin{proof}[of Theorem~\ref{thm:well-defined}]
 By Proposition~\ref{prop:no-infty-entries} the matrix $\vf'(\ns{k})^*$ has no $\infty$ entries.
 Then we clearly have $\vf'(\ns{k})^* (\Id - \vf'(\ns{k})) = \Id$, so $(\Id - \vf'(\ns{k}))^{-1} = \vf'(\ns{k})^*$,
  which is the first claim of part~2.\ of the theorem.
 Hence, we also have
  \begin{align*}
   \hNe(\ns{k}) & = \ns{k} + \sum_{d=0}^\infty \left(\vf'(\ns{k})^d ( \vf(\ns{k}) - \ns{k} ) \right) \\
                & = \ns{k} + \vf'(\ns{k})^* ( \vf(\ns{k}) - \ns{k} ) \\
                & = \ns{k} + (\Id - \vf'(\ns{k}))^{-1} ( \vf(\ns{k}) - \ns{k} ) \\
                & = \Ne(\ns{k}) \;,
  \end{align*}
  so we can replace $\hNe$ by~$\Ne$.
 Therefore, part~1.\ of the theorem is implied by Proposition~\ref{prop:newton-stacs07}.
 It remains to show $(\Id-\vf'(\vx))^{-1} = \vf'(\vx)^*$ for all $\vx \prec \mu\vf$.
 It suffices to show that $\vf'(\vx)^*$ has no $\infty$ entries.
 By part~1.\ the sequence $(\ns{k})_{k\in\N}$ converges to~$\mu\vf$.
 So there is a $k'$ such that $\vx \le \ns{k'}$.
 By Proposition~\ref{prop:no-infty-entries}, $\vf'(\ns{k'})^*$ has no $\infty$ entries,
  so, by monotonicity, $\vf'(\vx)^*$ has no $\infty$ entries either.
\qed
\end{proof}

\subsection{Monotonicity}
\mbox{}
\begin{lemma}[Monotonicity of the Newton operator]\label{lem:newton-mon} 
 Let $\vf$ be a clean and feasible SPP.
 Let $\vzero\le \vx \le \vy \le \vf(\vy) \le \mu\vf$ and let $\Ne_{\vf}(\vy)$ exist.
 Then
 \[
  \Ne_{\vf}(\vx) \le \Ne_{\vf}(\vy)\;.
 \]
\end{lemma}
\begin{proof}
For $\vx \le \vy$ we have $\vf'(\vx) \le \vf'(\vy)$ as every entry of $\vf'(\vX)$ is a monotone polynomial.
Hence, $\vf'(\vx)^\ast \le \vf'(\vy)^\ast$.
With this at hand we get:
\begin{align*}
 \Ne_{\vf}(\vy) & =   \vy + \vf'(\vy)^\ast ( \vf(\vy) - \vy ) && \text{(Theorem~\ref{thm:well-defined})} \\
                & \ge \vy + \vf'(\vx)^\ast ( \vf(\vy) - \vy ) && \text{($\vf'(\vy)^\ast \ge \vf'(\vx)^\ast$)} \\
                & \ge \vy + \vf'(\vx)^\ast ( \vf(\vx) + \vf'(\vx)(\vy-\vx) - \vy ) && \text{(Lemma~\ref{lem:taylor})} \\
                & =   \vy + \vf'(\vx)^\ast ( (\vf(\vx)-\vx) - (\Id - \vf'(\vx))(\vy-\vx) ) \\
                & =   \vy + \vf'(\vx)^\ast ( \vf(\vx) - \vx ) - (\vy-\vx) && \text{($\vf'(\vx)^\ast =$} \\
                  &&& \quad (\Id - \vf'(\vx))^{-1}) \\
                & =   \Ne_{\vf}(\vx)                          && \text{(Theorem~\ref{thm:well-defined})}  \\
\end{align*}
\qed
\end{proof}

\subsection{Exponential Convergence Order in the Nonsingular Case}

If the matrix \mbox{$\Id - \vf'(\mu\vf)$} is nonsingular, Newton's method has exponential convergence order
 in the sense of Definition~\ref{def:valid-bits}.
\iftechrep{
This is, in fact, a well known general property of Newton's method, see, e.g., Theorem 4.4 of~\cite{SuliMayers:book}.
For completeness, we show that Newton's method for ``nonsingular'' SPPs has exponential convergence order,
 see Theorem~\ref{thm:quadratic-convergence} below.

\begin{lemma} \label{lem:qc-new-error}
 Let $\vf$ be a clean and feasible SPP.
 Let $\vzero \le \vx \le \mu\vf$ such that $\vf'(\vx)^*$ exists.
 Then there is a bilinear function $B: \Rp^n \times \Rp^n \to \Rp^n$ with
  \[
   \mu\vf - \Ne(\vx) \le \vf'(\vx)^* B(\mu\vf - \vx, \mu\vf - \vx) \;.
  \]
\end{lemma}
\begin{proof}
 Write $\vd := \mu\vf - \vx$.
 By Taylor's theorem (cf.\ Lemma~\ref{lem:taylor}) we obtain
  \begin{equation} \label{eq:proof-lem-qc-new-error}
   \vf(\vx + \vd) \le \vf(\vx) + \vf'(\vx) \vd + B (\vd, \vd)
  \end{equation}
  for the bilinear map~$B(\vX) := \vf''(\mu\vf)(\vX,\vX)$, where $\vf''(\mu\vf)$ denotes the rank-3 tensor of the second partial derivatives evaluated at $\mu\vf$
  \cite{OrtegaRheinboldt:book}.
 We have
 \begin{align*}
  \mu\vf - \Ne(\vx)
   & = \vd - \vf'(\vx)^* (\vf(\vx) - \vx) \\
   & = \vd - \vf'(\vx)^* (\vd + \vf(\vx) - (\vx+\vd)) \\
   & = \vd - \vf'(\vx)^* (\vd + \vf(\vx) - \vf(\vx+\vd)) && \text{($\vx + \vd = \mu\vf = \vf(\mu\vf)$)} \\
   & \le \vd - \vf'(\vx)^* \left(\vd - \vf'(\vx) \vd - B(\vd, \vd) \right)
            && \text{(by~\eqref{eq:proof-lem-qc-new-error})} \\
   & = \vd - \vf'(\vx)^* \left((\Id - \vf'(\vx)) \vd  - B(\vd, \vd) \right) \\
   & = \vd - \vd + \vf'(\vx)^* B(\vd, \vd)
            && \text{($\vf'(\vx)^* = (\Id - \vf'(\vx))^{-1}$)} \\
   & = \vf'(\vx)^* B(\vd, \vd)
 \end{align*}
\qed
\end{proof}
}{} 

Define for the following lemmata $\Ds{k} := \mu\vf - \ns{k}$, i.e., $\Ds{k}$ is the error after $k$ Newton iterations.
The following lemma bounds $\norm{\Ds{k+1}}$ in terms of $\norm{\Ds{k}}^2$ if $\Id - \vf'(\mu\vf)$ is nonsingular.

\begin{lemma} \label{lem:qc-square}
 Let $\vf$ be a clean and feasible SPP such that $\Id - \vf'(\mu\vf)$ is nonsingular.
 Then there is a constant $c > 0$ such that
  \[
   \norm{\Ds{k+1}}_\infty \le c \cdot \norm{\Ds{k}}_\infty^2 \text{ for all $k\in\N$.}
  \]
\end{lemma}
\iftechrep{
\begin{proof}
 As $\Id - \vf'(\mu\vf)$ is nonsingular, we have, by Theorem~\ref{thm:well-defined}, $(\Id - \vf'(\vx))^{-1} = \vf'(\vx)^*$
  for all $\vzero \le \vx \le \mu\vf$.
 By continuity, there is a $c_1 > 0$ such that $\norm{\vf'(\vx)^*} \le c_1$ for all $\vzero \le \vx \le \mu\vf$.
 Similarly, there is a $c_2 > 0$ such that $\norm{B(\vx,\vx)} \le c_2 \norm{\vx}^2$ for all $\vzero \le \vx \le \mu\vf$,
  because $B$ is bilinear.
 So it follows from Lemma~\ref{lem:qc-new-error} that $\norm{\Ds{k+1}} \le c_1c_2 \norm{\Ds{k}}^2$.
\qed
\end{proof}
}{ 
\begin{proof}
 See \cite{EKL10:SIAMtechRep} or Theorem 4.4 of~\cite{SuliMayers:book}.
\qed
\end{proof}
}

Lemma~\ref{lem:qc-square} implies that Newton's method has an exponential convergence order in the nonsingular case.
More precisely:
\begin{theorem} \label{thm:quadratic-convergence}
 Let $\vf$ be a clean and feasible SPP such that $\Id - \vf'(\mu\vf)$ is nonsingular.
 Then there is a constant $k_\vf \in \N$ such that
  \[
   \beta( k_\vf + i) \ge 2^i \text{ for all $i\in\N$.}
  \]
\end{theorem}
\begin{proof}
\newcommand{\tkf}{\widetilde{k}_\vf}
 We first show that there is a constant $\tkf \in \N$ such that 
  \begin{equation}
   \norm{\Ds{\tkf + i}}_\infty \le 2^{-2^i} \text{ for all $i\in\N$.} \label{eq:qc-error-decays-exponentially}
  \end{equation}
 We can assume w.l.o.g.\ that $c \ge 1$ for the $c$ from Lemma~\ref{lem:qc-square}.
 As the $\Ds{k}$ converge to $\vzero$, we can choose $\tkf \in \N$ large enough such that
  $d := -\log \norm{\Ds{\tkf}} - \log c \ge 1$.
 As $c,d \ge 1$, it suffices to show the following inequality:
 \begin{equation*} \label{eq:lem-qc-error-decays-exponentially}
  \norm{\Ds{\tkf + i}} \le \frac{2^{-d\cdot 2^i}}{c}\;.
 \end{equation*}
 We proceed by induction on~$i$.
 For $i=0$, the inequality above 
  follows from the definition of~$d$.
 Let $i\ge 0$.
 Then
 \begin{align*}
  \norm{\Ds{\tkf + i + 1}} & \le c \cdot \norm{\Ds{\tkf + i}}^2 && \text{(Lemma~\ref{lem:qc-square})} \\
                           & \le c \cdot \frac{2^{-d\cdot 2^i \cdot 2}}{c^2} && \text{(induction hypothesis)} \\
                           &  = \frac{2^{-d\cdot 2^{i+1}}}{c} \;.
 \end{align*}
 Hence, \eqref{eq:qc-error-decays-exponentially} is proved.

 Choose $m \in \N$ large enough such that 
  $2^{m+i} + \log(\mu\vf_j) \ge 2^i$ holds for all components~$j$.
 Thus
 \begin{align*}
  \Dsc{\tkf + m + i}_j / \mu\vf_j
   & \le 2^{-2^{m+i}} / \mu\vf_j && \text{(by~\eqref{eq:qc-error-decays-exponentially})} \\
   &  =  2^{-(2^{m+i} + \log(\mu\vf_j))} \\
   & \le 2^{-2^i} && \text{(choice of~$m$)}\;.
 \end{align*}
 So, with $k_\vf := \tkf + m$, the approximant~$\ns{k_\vf + i}$ has at least $2^i$ valid bits of~$\mu\vf$.
\qed
\end{proof}

This type of analysis has serious shortcomings.
In particular, Theorem~\ref{thm:quadratic-convergence} excludes the case where $\Id - \vf'(\mu\vf)$ is singular.
We will include this case in our convergence analysis in \S~\ref{sec:scSPPs} and \S~\ref{sec:decomposed}.
Furthermore, and maybe more severely, Theorem~\ref{thm:quadratic-convergence} does not give any bound on~$k_\vf$.
We solve this problem for strongly connected SPPs in \S~\ref{sec:scSPPs}.

\subsection{Reduction to the Quadratic Case} \label{sub:reduction-quadratic-case}

In this section we reduce SPPs to quadratic SPPs, i.e., to SPPs in which every polynomial $f_i(\vX)$ has degree at most~$2$,
 and show that the convergence on the quadratic SPP is no faster than on the original SPP.
In the following sections we will obtain convergence speed guarantees of Newton's method on quadratic SPPs.
Hence, one can perform Newton's method on the original SPP and, using the results of this section,
 convergence is at least as fast as on the corresponding quadratic SPP.

The idea to reduce the degree of our SPP~$\vf$ is to introduce
auxiliary variables that express quadratic subterms.
This can be done repeatedly until all polynomials in the system have
reached degree at most~$2$.
The construction is very similar to the one that transforms a
context-free grammar into another grammar in Chomsky normal form.
The following theorem shows that the transformation does not
accelerate the convergence of Newton's method.

\begin{theorem} \label{thm:reduction-quadratic}
 Let $\vf(\vX)$ be a clean and feasible SPP such that $f_s(\vX) = g(\vX) + h(\vX) X_i X_j$ for some $1\le i,j,s \le n$,
  where $g(\vX)$ and $h(\vX)$ are polynomials with nonnegative coefficients. Let
 $\widetilde{\vf}(\vX,Y)$ be the SPP given by
 \[
 \begin{array}{rcl@{\hspace{0.3cm}}l}
 \widetilde{f}_\ell(\vX,Y) & = & f_\ell(\vX) & \mbox{for every $\ell \in \{1, \ldots, s-1\}$}\\
 \widetilde{f}_s(\vX,Y) & = & g(\vX) + h(\vX)Y \\
 \widetilde{f}_\ell(\vX,Y) & = & f_\ell(\vX) & \mbox{for every $\ell \in \{s+1, \ldots, n\}$}\\
 \widetilde{f}_{n+1}(\vX,Y) & = & X_i X_j.
 \end{array}
 \]
 Then the function $b: \R^n \to \R^{n+1}$ given by $b(\vX) = ( X_1,\ldots,X_n, X_i X_j)^\top$
  is a bijection between the set of fixed points of $\vf(\vX)$ and $\widetilde{\vf}(\vX,Y)$.
 Moreover, $\tns{k} \le (\nsc{k}_1,\ldots,\nsc{k}_n, \nsc{k}_i \nsc{k}_j)^\top$ for all $k\in\N$,
  where $\tns{k}$ and $\ns{k}$ are the Newton approximants of $\widetilde{\vf}$ and $\vf$, respectively.
\end{theorem}
\begin{proof}
 \newcommand{\de}{\vdelta}
 \newcommand{\tde}{\widetilde{\vdelta}}
 \newcommand{\tdes}{\widetilde{\delta}}
 \newcommand{\tvf}{\widetilde{\vf}}
 We first show the claim regarding~$b$:
 if $\vx$ is a fixed point of $\vf$, then $b(\vx)=(\vx, x_i x_j)$ is a fixed point of $\tvf$.
 Conversely, if $(\vx, y)$ is a fixed point of $\tvf$,
  then we have $y = x_i x_j$ implying that $\vx$ is a fixed point of $\vf$. 
 Therefore, the least fixed point $\mu \vf$ of $\vf$ determines $\mu \tvf$, and vice versa.

 Now we show that the Newton sequence of $\vf$ converges at least as fast as the Newton sequence of $\tvf$.
 In the following we write $\vY$ for the $(n+1)$-dimensional vector of variables $(X_1, \ldots, X_n, Y)^\top$
  and, as usual, $\vX$ for $(X_1, \ldots, X_n)^\top$.
 For an $(n+1)$-dimensional vector $\vx$, we let $\vx_{[1,n]}$ denote its restriction to the $n$ first components,
  i.e., $\vx_{[1,n]} := (x_1,\ldots,x_n)^\top$.
 Note that $\vY_{[1,n]} = \vX$.
 Let $\ve_s$ denote the unit vector $(0, \ldots, 0,1,0 \ldots 0)^\top$, where the ``$1$'' is on the $s$-th place.
 We have:
 \begin{align*}
  \tvf(\vY) & =
  \begin{pmatrix}
   \vf(\vX) + \ve_s h(\vX) (Y - X_iX_j) \\
   X_iX_j
  \end{pmatrix}
  \intertext{and}
  \tvf'(\vY) & =
  \begin{pmatrix}
   \vf'(\vX) + \ve_s \pd{h(\vX) (Y - X_iX_j)}{\vX} & \quad \ve_s h(\vX) \\
   \pd{X_iX_j}{\vX}                                & \quad 0
  \end{pmatrix}
 \end{align*}
 We need the following lemma.
 \begin{lemma} \label{lem:reduction-quadratic-updates-similar}
  Let $\vz \in \Rp^{n}$, $\de  = \big(\Id - \vf'(\vz)\big)^{-1} (\vf(\vz) - \vz)$ and
      $\tde = \left(\Id - \tvf'(\vz, z_iz_j)\right)^{-1} (\tvf(\vz, z_iz_j) - (\vz,z_iz_j)^\top)$.
  Then $\de = \tde_{[1,n]}$.
 \end{lemma}

 {\em Proof of the lemma.}
  \begin{align*}
   \tvf'(\vz, z_iz_j) & =
   \begin{pmatrix}
    \vf'(\vz) + \ve_s h(\vz) \pd{(Y - X_iX_j)}{\vX}|_{\vY = (\vz,z_iz_j)} & \quad \ve_s h(\vz) \\
    \pd{X_iX_j}{\vX}|_{\vY = (\vz,z_iz_j)}                                & \quad 0
   \end{pmatrix}
   \\ & =
   \begin{pmatrix}
    \vf'(\vz) - \ve_s h(\vz) \pd{(X_iX_j)}{\vX}|_{\vX = \vz} & \quad \ve_s h(\vz) \\
    \pd{X_iX_j}{\vX}|_{\vX = \vz}                            & \quad 0
   \end{pmatrix}
  \end{align*}
  We have $(\Id - \tvf'(\vz, z_iz_j)) \tde = (\tvf(\vz, z_iz_j) - (\vz,z_iz_j)^\top)$, or equivalently:
  \[
    \begin{pmatrix}
         \Id - \vf'(\vz) + \ve_s h(\vz) \pd{(X_iX_j)}{\vX}|_{\vX = \vz} & \quad - \ve_s h(\vz) \\
          - \pd{X_iX_j}{\vX}|_{\vX = \vz}                               & \quad   1 \\
    \end{pmatrix}
    \cdot
    \begin{pmatrix}
      \tde_{[1,n]}\\
      \tdes_{n+1} \\
    \end{pmatrix}
    =
    \begin{pmatrix}
      \vf(\vz) - \vz \\
      0 \\
    \end{pmatrix}
  \] 
  Multiplying the last row by $\ve_s h(\vz)$ and adding to the first $n$ rows yields:
  \[
   \left( \Id - \vf'(\vz) \right) \tde_{[1,n]} = \vf(\vz) - \vz
  \]
  So we have $\tde_{[1,n]} = \left( \Id - \vf'(\vz) \right)^{-1} \left(\vf(\vz) - \vz\right) = \de$, which proves the lemma.
 \qed

 Now we proceed by induction on $k$ to show $\tns{k}_{[1,n]} \le \ns{k}$,
  where $\tns{k}$ is the Newton sequence for $\tvf$.
 By definition of the Newton sequence this is true for $k=0$.
 For the step, let $k \ge 0$ and define $\vu := (\tns{k}_{[1,n]}, \tnsc{k}_i\cdot \tnsc{k}_j)^\top$.
 Then we have:
 \begin{align*}
  \tns{k+1}_{[1,n]} & = \Ne_{\tvf}( \tns{k} )_{[1,n]} \\
                    & \stackrel{(\ast)}{\le} \Ne_{\tvf}( \vu )_{[1,n]}
                      && \text{(see below)} \\
                    & = \tns{k}_{[1,n]} + \left((\Id - \tvf'(\vu))^{-1} (\tvf(\vu) - \vu)\right)_{[1,n]} \\
                    & = \tns{k}_{[1,n]} + (\Id - \vf'(\tns{k}_{[1,n]}))^{-1} (\vf(\tns{k}_{[1,n]}) - \tns{k}_{[1,n]})
                      && \text{(Lemma~\ref{lem:reduction-quadratic-updates-similar})} \\
                    & = \Ne_{\vf}( \tns{k}_{[1,n]} ) \\
                    & \le \Ne_{\vf}(\ns{k})                     && \text{(induction)} \\
                    & = \ns{k+1}
 \end{align*}

 At the inequality marked with~$(*)$ we used the monotonicity of $\Ne_{\tvf}$ (Lemma~\ref{lem:newton-mon}) combined with
  Theorem~\ref{thm:well-defined}, which states $\tns{k} \le \tvf(\tns{k})$,
  hence in particular $\tnsc{k}_{n+1} \le \tnsc{k}_i \tnsc{k}_j$.
 This concludes the proof of Theorem~\ref{thm:reduction-quadratic}.
\qed
\end{proof}

%% file: sec4-strongly-connected.tex
\section{Strongly Connected SPPs} \label{sec:scSPPs}

In this section we study the convergence speed of Newton's method on strongly connected SPPs,
 short scSPPs, see Definition~\ref{def:depend}.

\subsection{Cone Vectors}
Our convergence speed analysis makes crucial use of the existence of {\em cone vectors}.
\begin{definition} \label{def:cone-vector}
 Let $\vf$ be an SPP.
 A vector $\vd \in \Rp^n$ is a {\em cone vector} if $\vd \succ \vzero$ and $\vf'(\mu\vf) \vd \le \vd$.
\end{definition}

\newcommand{\lemsemiconevectorexists}{
\begin{lemma} \label{lem:semi-cone-vector-exists}
 Any clean and feasible scSPP $\vf$ has a vector $\vd > \vzero$ with \mbox{$\vf'(\mu\vf) \vd \le \vd$}.
\end{lemma}
\begin{proof}
 Consider the Kleene sequence $(\ks{k})_{k\in\N}$.
 Since $\vf$ is strongly connected, we have $\vzero \le \ks{k} \prec \mu\vf$ for all $k \in \N$.
 By Theorem~\ref{thm:well-defined}.2., the matrices $(\Id - \vf'(\ks{k}))^{-1} = \vf'(\ks{k})^*$ exist for all $k$.
 Let $\norm{\cdot}$ be any norm.
 Define the vectors
  \[
   \des{k} := \frac{\vf'(\ks{k})^* \vone}{\norm{\vf'(\ks{k})^* \vone}} \;.
  \]
 Notice that for all $k\in\N$ we have
  $(\Id - \vf'(\ks{k})) \des{k} = \frac{1}{\norm{\vf'(\ks{k})^*\vone}} \cdot \vone  \ge \vzero$.
 Furthermore we have $\des{k} \in C$, where $C := \{\vx \ge \vzero \mid \norm{\vx} = 1\}$ is compact.
 So the sequence $(\des{k})_{k\in\N}$ has a convergent subsequence, whose limit, say $\vd$, is also in~$C$.
 In particular $\vd > \vzero$.
 As $(\ks{k})_{k\in\N}$ converges to~$\mu\vf$ and $(\Id - \vf'(\ks{k})) \des{k} \ge \vzero$,
  it follows by continuity $(\Id - \vf'(\mu\vf)) \vd \ge \vzero$.
\qed
\end{proof}
}
\newcommand{\lemvdvzeroisenough}{
\begin{lemma}\label{lem:vd>vzero-is-enough}
 Let $\vf$ be a clean and feasible scSPP and let $\vd > \vzero$ with $\vf'(\mu\vf)\vd \le \vd$.
 Then $\vd$ is a cone vector, i.e., $\vd \succ \vzero$.
\end{lemma}
\begin{proof}
 Since $\vf$ is an SPP, every component of $\vf'(\mu\vf)$ is nonnegative.
 So,
 \[
 \vzero \le \vf'(\mu\vf)^n \vd \le \vf'(\mu\vf)^{n-1}\vd \le
 \ldots \le \vf'(\mu\vf)\vd \le \vd.
 \]
 Let w.l.o.g.\ $d_1 > 0$.
 As $\vf$ is strongly connected, there is for all $j$ with $1 \le j \le n$ an $r_j \le n$ such that $(\vf'(\mu\vf)^{r_j})_{j1} > 0$.
 Hence, $(\vf'(\mu\vf)^{r_j} \vd)_j > 0$ for all $j$.
 With above inequality chain, it follows that $d_j \ge (\vf'(\mu\vf)^{r_j} \vd)_j > 0$.
 So, $\vd \succ \vzero$.
\qed
\end{proof}
}

\iftechrep{
We will show that any scSPP has a cone vector, see Proposition~\ref{prop:cone-vector-exists} below.
As a first step, we show the following lemma.

\lemsemiconevectorexists

\lemvdvzeroisenough

The following proposition follows immediately by combining Lemmata \ref{lem:semi-cone-vector-exists} and~\ref{lem:vd>vzero-is-enough}.
\begin{proposition} \label{prop:cone-vector-exists}
 Any clean and feasible scSPP has a cone vector.
\end{proposition}

We remark that using Perron-Frobenius theory \cite{book:BermanP} there is a simpler proof for Proposition~\ref{prop:cone-vector-exists}:
By Theorem~\ref{thm:well-defined} $\vf'(\vx)^*$ exists for all $\vx \prec \mu\vf$.
So, by fundamental matrix facts \cite{book:BermanP}, the spectral radius of $\vf'(\vx)$ is less than $1$ for all $\vx \prec \mu\vf$.
As the eigenvalues of a matrix depend continuously on the matrix, the spectral radius of $\vf'(\mu\vf)$, say~$\rho$, is at most~$1$.
Since $\vf$ is strongly connected, $\vf'(\mu\vf)$ is irreducible, and so Perron-Frobenius theory guarantees the existence
 of an eigenvector $\vd \succ \vzero$ of $\vf'(\mu\vf)$ with eigenvalue~$\rho$.
So we have $\vf'(\mu\vf) \vd = \rho \vd \le \vd$, i.e., the eigenvector~$\vd$ is a cone vector.
}{
\begin{proposition} \label{prop:cone-vector-exists}
 Any clean and feasible scSPP has a cone vector.
\end{proposition}
\begin{proof}
By Theorem~\ref{thm:well-defined} $\vf'(\vx)^*$ exists for all $\vx \prec \mu\vf$.
So, by fundamental matrix facts \cite{book:BermanP}, the spectral radius of $\vf'(\vx)$ is less than $1$ for all $\vx \prec \mu\vf$.
As the eigenvalues of a matrix depend continuously on the matrix, the spectral radius of $\vf'(\mu\vf)$, say~$\rho$, is at most~$1$.
Since $\vf$ is strongly connected, $\vf'(\mu\vf)$ is irreducible, and so Perron-Frobenius theory guarantees the existence
 of an eigenvector $\vd \succ \vzero$ of $\vf'(\mu\vf)$ with eigenvalue~$\rho$.
So we have $\vf'(\mu\vf) \vd = \rho \vd \le \vd$, i.e., the eigenvector~$\vd$ is a cone vector.
\qed
\end{proof}
}
\iftechrep{}{\par In~\cite{EKL10:SIAMtechRep} we prove Proposition~\ref{prop:cone-vector-exists} independently of Perron-Frobenius theory.}
\subsection{Convergence Speed in Terms of Cone Vectors}

Now we show that cone vectors play a fundamental role for the convergence speed of Newton's method.
The following lemma gives a lower bound of the Newton approximant~$\ns{1}$ in terms of a cone vector.

\begin{lemma}\label{lem:blub}
 Let $\vf$ be a feasible (not necessarily clean) SPP such that $\vf'(\vzero)^*$ exists.
 Let $\vd$ be a cone vector of~$\vf$.
 Let $\vzero \ge \mu\vf - \lambda \vd$ for some $\lambda \ge 0$.
 Then
 \[
   \Ne(\vzero) \ge \mu \vf - \frac{1}{2} \lambda \vd\;.
 \]
\end{lemma}
\begin{proof}
\newcommand{\Ts}[1]{\vec{T}^{(#1)}}
\newcommand{\Tsc}[1]{T^{(#1)}}
\newcommand{\Ls}[1]{L^{(#1)}}
\newcommand{\Xs}[1]{\vec{X}^{(#1)}}
 We write $\vf(\vX)$ as a sum
 $\vf(\vX) = \vc + \sum_{k=1}^{D} \Ts{k}(\vX,\ldots,\vX)$,
 where $D$ is the degree of~$\vf$ and, for all $k \in \{1,\ldots,D\}$ and all $i \in \{1,\ldots,n\}$,
  the component~$\Tsc{k}_i$ of~$\Ts{k}$ is the symmetric $k$-linear form associated to the degree-$k$ terms of~$f_i$.
 Let $\Ls{k}: (\R^n)^{k-1} \to \R^{n\times n}$ such that $\Ts{k}(\Xs{1}, \ldots, \Xs{k}) = \Ls{k}(\Xs{1}, \ldots, \Xs{k-1}) \cdot \Xs{k}$.
 Now we can write
  $$\vf(\vX) = \vc + \sum_{k=1}^{D} \Ls{k}(\vX,\ldots,\vX)\vX \text{ \quad and \quad} \vf'(\vX) = \sum_{k=1}^{D} k \cdot \Ls{k}(\vX,\ldots,\vX)\;.$$
 We write $L$ for $\Ls{1}$, and $\vh(\vX)$ for $\vf(\vX) - L\vX -\vc$.
 We have:
  \begin{align*}
          \frac{\lambda}{2} \vd
    &=    \frac{\lambda}{2} (L^*\vd - L^*L\vd)                   && \text{($L^* = \Id + L^*L$)}\\
    &\ge  \frac{\lambda}{2} (L^*\vf'(\fix{\vf})\vd - L^*L\vd)    && \text{($\vf'(\fix{\vf})\vd \le \vd$)}\\
    &=    \frac{\lambda}{2} L^*\vh'(\fix{\vf})\vd                && \text{($\vf'(\vx) = \vh'(\vx) + L$)}\\
    &=    L^* \frac{1}{2} \vh'(\fix{\vf}) \lambda \vd \\
    &\ge  L^* \frac{1}{2} \vh'(\fix{\vf}) \fix{\vf}            && \text{($\lambda \vd \ge \fix{\vf}$)}\\
    &=    L^* \frac{1}{2} \sum_{k=2}^{D} k \cdot \Ls{k}(\mu\vf,\ldots,\mu\vf) \fix{\vf} \\
    &\ge  L^* \sum_{k=2}^{D} \Ls{k}(\mu\vf,\ldots,\mu\vf) \fix{\vf} \\
    &=    L^* \vh(\fix{\vf}) \\
    &=    L^* (\vf(\fix{\vf}) - L \fix{\vf} - \vc)             && \text{($\vf(\vx) = \vh(\vx) + L\vx + \vc$)}\\
    &=    L^* \fix{\vf} - L^* L \fix{\vf} - L^* \vc            && \text{($\vf(\fix{\vf}) = \fix{\vf}$)}\\
    &=    \fix{\vf} - L^* \vc                                  && \text{($L^* = \Id + L^*L$)}\\
    &=    \fix{\vf} - \Ne(\vzero)                              && \text{($\Ne(\vzero) = \vf'(\vzero)^*\vf(\vzero) = L^* \vc$)}
     & \text{\qed}
  \end{align*}
\end{proof}

We extend Lemma~\ref{lem:blub} to arbitrary vectors~$\vx$ as follows.
\begin{lemma}\label{lem:blub2}
 Let $\vf$ be a feasible (not necessarily clean) SPP.
 Let $\vzero \le \vx \le \mu\vf$ and $\vx \le \vf(\vx)$ such that $\vf'(\vx)^*$ exists.
 Let $\vd$ be a cone vector of~$\vf$.
 Let $\vx \ge \mu\vf - \lambda \vd$ for some $\lambda \ge 0$.
 Then
 \[
   \Ne(\vx) \ge \mu \vf - \frac{1}{2} \lambda \vd\;.
 \]
\end{lemma}
\begin{proof}
 Define $\vg(\vX) := \vf(\vX + \vx) - \vx$.
 We first show that $\vg$ is an SPP (not necessarily clean).
 The only coefficients of~$\vg$ that could be negative are those of degree 0.
 But we have $\vg(\vzero) = \vf(\vx) - \vx \ge \vzero$, and so these coefficients are also nonnegative.

 It follows immediately from the definition that $\mu\vf - \vx \ge \vzero$ is the least fixed point of~$\vg$.
 Moreover, $\vg$ satisfies $\vg'( \mu\vf - \vx) \vd \le \vd$, and so $\vd$ is also a cone vector of~$\vg$.
 Finally, we have $\vzero \ge \mu\vf - \vx - \lambda \vd = \mu\vg - \lambda \vd$.
 So, Lemma~\ref{lem:blub} can be applied as follows.
 \begin{align*}
   \Ne_\vf(\vx) & =   \vx + \vf'(\vx)^*    (\vf(\vx) - \vx) \\
                & =   \vx + \vg'(\vzero)^* (\vg(\vzero) - \vzero) \\
                & =   \vx + \Ne_\vg(\vzero) \\
                & \ge \vx + \mu\vg - \frac{1}{2} \lambda \vd && \text{(Lemma~\ref{lem:blub})} \\
                & =   \mu\vf - \frac{1}{2} \lambda \vd
 \end{align*}
\qed
\end{proof}

By induction we can extend this lemma to the whole Newton sequence:
\begin{lemma} \label{lem:blubblub}
 Let $\vd$ be a cone vector of a clean and feasible SPP $\vf$ and let \mbox{$\lmax = \max_j\{ \frac{\mu\vf_j}{d_j} \}$}.
 Then
 \[
   \ns{k} \ge \mu\vf - 2^{-k} \lmax \vd\;.
 \]
\end{lemma}

\begin{figure}[ht]
 \begin{center}
 {
  \psfrag{X1 = f1(X1,X2)}{$X_1 = f_1(\vX)$}
  \psfrag{X2 = f2(X1,X2)}{$X_2 = f_2(\vX)$}
  \psfrag{muf}{$\fix{\vf}=\vr(0)$}
  \psfrag{0}{$0$}
  \psfrag{0.2}{}
  \psfrag{0.4}{}
  \psfrag{0.6}{}
  \psfrag{\2610.4}{$-0.4$}
  \psfrag{\2610.2}{$-0.2$}
  \psfrag{02h}{$0.2$}
  \psfrag{04h}{$0.4$}
  \psfrag{06h}{$0.6$}
  \psfrag{02v}{$0.2$}
  \psfrag{X1}{$X_1$}
  \psfrag{X2}{$X_2$}
  \psfrag{glmax}{$\vr(\lmax)$}
   \scalebox{0.9}{ \includegraphics{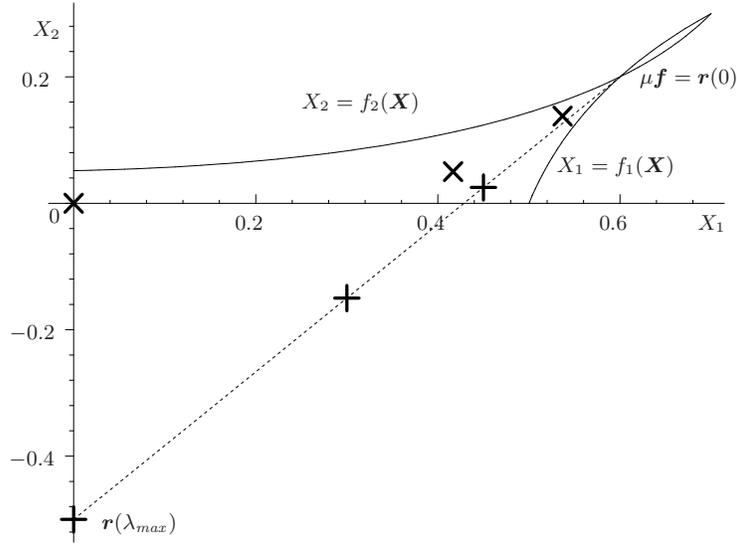}}
 }
 \end{center}
 \caption{Illustration of Lemma~\ref{lem:blubblub}:
          The points (shape: $\boldsymbol{+}$) on the ray~$r$ along a cone vector are lower bounds
           on the Newton approximants (shape: $\boldsymbol{\times}$).}
 \label{fig:blubblub}
\end{figure}
Before proving the lemma we illustrate it by a picture.
The dashed line in Figure~\ref{fig:blubblub} is the ray $\vr(t) = \mu\vf - t\vd$ along a cone vector~$\vd$.
Notice that $\vr(0)$ equals $\mu\vf$ and $\vr(\lmax)$ is the greatest point on the ray that is below $\vzero$.
The figure also shows the Newton iterates $\ns{k}$ for $0 \le k \le 2$ (shape: $\boldsymbol{\times}$) and the corresponding
 points $\vr(2^{-k} \lmax)$ (shape: $\boldsymbol{+}$) located on the ray $\vr$.
Observe that $\ns{k} \ge \vr(2^{-k}\lmax)$, as claimed by Lemma~\ref{lem:blubblub}.

\begin{proof}[of Lemma~\ref{lem:blubblub}]
 By induction on~$k$.
 For the induction base ($k=0$) we have for all components~$i$:
  \[
   \left( \mu\vf - \lmax \vd \right)_i = \left( \mu\vf - \max_j \left\{ \frac{\mu\vf_j}{d_j} \right\} \vd \right)_i
    \le \mu\vf_i - \frac{\mu\vf_i}{d_i} d_i = 0\;,
  \]
 so $\ns{0} = \vzero \ge \mu\vf - \lmax \vd$.

 For the induction step, let $k \ge 0$.
 By induction hypothesis we have $\ns{k} \ge \mu\vf - 2^{-k} \lmax \vd$.
 So we can apply Lemma~\ref{lem:blub2} to get
  \[
   \ns{k+1} = \Ne(\ns{k}) \ge \mu\vf - \frac{1}{2} 2^{-k} \lmax \vd = \mu\vf - 2^{-(k+1)} \lmax \vd\;.
  \]
 \qed
\end{proof}

The following proposition guarantees a convergence order of the Newton sequence in terms of a cone vector.

\begin{proposition} \label{prop:blubblubblub}
 Let $\vd$ be a cone vector of a clean and feasible SPP $\vf$ and let
  $\lmax = \max_j \left\{ \frac{\mu\vf_j}{d_j} \right\}$ and
  $\lmin = \min_j \left\{ \frac{\mu\vf_j}{d_j} \right\}$.
 Let $k_{\vf,\vd} = \left\lceil \log \frac{\lmax}{\lmin} \right\rceil$.
 Then $\beta(k_{\vf,\vd} + i) \ge i$ for all $i\in\N$.
\end{proposition}

\begin{proof}
 For all $1 \le j \le n$ the following holds.
 \begin{align*}
  \bigl( \fix{\vf} - \ns{k_{\vf,\vd} + i} \bigr)_j
   & \le 2^{-(k_{\vf,\vd} + i)} \lmax d_j
        && \text{(Lemma~\ref{lem:blubblub})} \\
   & \le \frac{\lmin}{\lmax} 2^{-i} \lmax d_j
        && \text{(def.\ of $k_{\vf,\vd}$)} \\
   &  =  \lmin d_j \cdot 2^{-i} \\
   & \le \fix{\vf}_j \cdot 2^{-i}
        && \text{(def.\ of $\lmin$)}
 \end{align*}
 Hence, $\ns{k_{\vf,\vd} + i}$ has $i$ valid bits of~$\mu\vf$.
\qed
\end{proof}

\subsection{Convergence Speed Independent from Cone Vectors}

The convergence order provided by Proposition~\ref{prop:blubblubblub} depends on a cone vector~$\vd$.
While Proposition~\ref{prop:cone-vector-exists} guarantees the existence of a cone vector for scSPPs,
 it does not give any information on the magnitude of its components.
So we do not have any bound yet on the ``threshold'' $k_{\vf,\vd}$ from Proposition~\ref{prop:blubblubblub}.
The following theorem solves this problem.

\begin{theorem} \label{thm:estimate}
 Let $\vf$ be a quadratic, clean and feasible scSPP.
 Let $\cmin$ be the smallest nonzero coefficient of $\vf$
  and let $\mumin$ and $\mumax$ be the minimal and maximal component of $\fix{\vf}$, respectively.
 Let
  \[ k_\vf = \left\lceil \log \frac{\mumax}{\mumin \cdot \left(\cmin \cdot \min\{\mumin, 1\}\right)^n} \right\rceil\;.
  \]
 Then
  \[
   \beta( k_\vf + i) \ge i \text{ for all $i\in\N$.}
  \]
\end{theorem}
Before we prove Theorem~\ref{thm:estimate} we give an example.
\begin{example} \label{ex:thm-estimate}
 As an example of application of Theorem~\ref{thm:estimate} consider the scSPP equation of the back button process of
 Example~\ref{ex:back-button}.
 \[
  \begin{pmatrix}
   X_1 \\ X_2 \\ X_3
  \end{pmatrix}
  =
  \begin{pmatrix}
   0.4 X_2 X_1 + 0.6 \\
   0.3 X_1 X_2 + 0.4 X_3 X_2 + 0.3 \\
   0.3 X_1 X_3 + 0.7
  \end{pmatrix}
 \]


 We wish to know if there is a component $s \in \{1,2,3\}$ with $\mu\vf_s = 1$.
 Notice that $\vf(\vone) = \vone$, so $\mu\vf \le \vone$.
 Performing $14$ Newton steps (e.g.\ with Maple) yields an approximation $\ns{14}$ to~$\mu\vf$ with
 \[
  \begin{pmatrix}
   0.98 \\
   0.97 \\
   0.992
  \end{pmatrix}
  \le
  \ns{14}
  \le
  \begin{pmatrix}
   0.99 \\
   0.98 \\
   0.993
  \end{pmatrix} \:.
 \]
 We have $\cmin = 0.3$.
 In addition, since Newton's method converges to $\fix{\vf}$ from below, we know $\mumin \ge 0.97$.
 Moreover, $\mumax \le 1$, as $\vone = \vf(\vone)$ and so $\fix{\vf} \leq \vone$.
 Hence $\displaystyle k_\vf \le \left\lceil \log \frac{1}{0.97 \cdot (0.3 \cdot 0.97)^3} \right\rceil = 6$.
 Theorem~\ref{thm:estimate} then implies that $\ns{14}$ has 8 valid bits of $\fix{\vf}$.
 As $\fix{\vf} \leq \vone$, the absolute errors are bounded by the relative errors, and since $2^{-8} \le 0.004$ we know:
 \[
  \fix{\vf} \le
  \ns{14}
  +
  \begin{pmatrix}
   2^{-8}\\
   2^{-8}\\
   2^{-8}
  \end{pmatrix}
  \le
  \begin{pmatrix}
   0.994\\
   0.984\\
   0.997\\
  \end{pmatrix}
  \prec
  \begin{pmatrix}
   1\\
   1\\
   1\\
  \end{pmatrix}
 \]
 So Theorem~\ref{thm:estimate} yields a proof that $\mu\vf_s < 1$ for all three components~$s$.

 Notice also that the Newton sequence converges much faster than the Kleene sequence $(\ks{k})_{k\in\Nat}$.
 We have $\ks{14} \prec \bigl(0.89, 0.83, 0.96 \bigr)^\top$,
 so $\ks{14}$ has no more than $4$ valid bits in any component,
  whereas $\ns{14}$ has, in fact, more than $30$ valid bits in each component.
\qed
\end{example}

For the proof of Theorem~\ref{thm:estimate} we need the following lemma.
\begin{lemma} \label{lem:dmin-over-dmax}
 Let $\vd$ be a cone vector of a quadratic, clean and feasible scSPP~$\vf$.
 Let $\cmin$ be the smallest nonzero coefficient of~$\vf$ and $\mumin$ the minimal component of~$\mu\vf$.
 Let $\dmin$ and $\dmax$ be the smallest and the largest component of~$\vd$, respectively.
 Then
  \[
   \frac{\dmin}{\dmax} \ge \left( \cmin \cdot \min\{\mumin, 1\} \right)^n\;.
  \]
\end{lemma}
\begin{proof}
 In what follows we shorten $\mu\vf$ to~$\vmu$.
 Let w.l.o.g.\ $d_1 = \dmax$ and $d_n = \dmin$.
 We claim the existence of indices $s,t$ with $1 \le s,t \le n$
  such that $\vf'_{st}(\vmu) \ne 0$ and
  \begin{equation} \label{eq:proof-thm-proximity-1}
   \frac{\dmin}{\dmax} \ge \left( \frac{d_s}{d_t} \right)^n\;.
  \end{equation}
 To prove that such $s,t$ exist, we use the fact that $\vf$ is strongly connected,
  i.e., that there is a sequence $1 = r_1, r_2, \ldots, r_q = n$ with $q \le n$
   such that $\vf'_{r_{j+1}r_j}(\vX)$ is not constant zero.
 As $\vmu \succ \vzero$, we have $\vf'_{r_{j+1}r_j}(\vmu) \ne 0$.
 Furthermore
 \begin{align*}
   \frac{d_1}{d_n} & = \frac{d_{r_1}}{d_{r_2}} \cdots \frac{d_{r_{q-1}}}{d_{r_q}}
    \text{\ , and so} \\
   \log \frac{d_1}{d_n} & = \log \frac{d_{r_1}}{d_{r_2}} +
                                           \cdots + \log \frac{d_{r_{q-1}}}{d_{r_q}}\;.
 \end{align*}
 So there must exist a $j$ such that
  \begin{align*}
   \log \frac{d_1}{d_n} & \le (q-1) \log \frac{d_{r_j}}{d_{r_{j+1}}} \le n \log \frac{d_{r_j}}{d_{r_{j+1}}} \text{\quad, and so} \\
   \frac{d_n}{d_1}      & \ge \left( \frac{d_{r_{j+1}}}{d_{r_j}} \right)^n \;.
  \end{align*}
 Hence one can choose $s = r_{j+1}$ and $t = r_j$.

 As $\vd$ is a cone vector we have $\vf'(\vmu) \vd \le \vd$ and thus $\vf'_{st}(\vmu) d_t \le d_s$.
 Hence
 \begin{equation} \label{eq:proof-thm-proximity-2}
  \vf'_{st}(\vmu) \le \frac{d_s}{d_t}\;.
 \end{equation}
 On the other hand, since $\vf$ is quadratic, $\vf'$ is a linear mapping such that
 \begin{align*}
  \vf'_{st}(\vmu) & = 2 ( b_1 \cdot \mu_1 + \cdots + b_n \cdot \mu_n ) + \ell
 \end{align*}
 where $b_1, \ldots, b_n$ and $\ell$ are coefficients of quadratic, respectively linear, monomials of $\vf$.
 As $\vf'_{st}(\vmu) \ne 0$, at least one of these coefficients must be nonzero and so greater than or equal to $\cmin$.
 It follows $\vf'_{st}(\vmu) \ge \cmin \cdot \min\{\mumin, 1\}$.
 So we have
  \begin{align*}
   \left( \cmin \cdot \min\{\mumin, 1\} \right)^n
    & \le \left( \vf'_{st}(\vmu) \right)^n \\
    & \le \left( \frac{d_s}{d_t} \right)^n && \text{(by~\eqref{eq:proof-thm-proximity-2})} \\
    & \le \frac{\dmin}{\dmax} && \text{(by~\eqref{eq:proof-thm-proximity-1})}\;.
  \end{align*}
\qed
\end{proof}

Now we can prove Theorem~\ref{thm:estimate}.
\begin{proof}[of Theorem~\ref{thm:estimate}]
 By Proposition~\ref{prop:cone-vector-exists}, $\vf$ has a cone vector~$\vd$.
 Let
  $\dmax = \max_j \{d_j\}$ and
  $\dmin = \min_j \{d_j\}$ and
  $\lmax = \max_j \left\{ \frac{\mu\vf_j}{d_j} \right\}$ and
  $\lmin = \min_j \left\{ \frac{\mu\vf_j}{d_j} \right\}$.
 We have:
 \begin{align*}
  \frac{\lmax}{\lmin} & \le \frac{\mumax \cdot \dmax}{\mumin \cdot \dmin}
                        && \text{(as $\lmax \le \frac{\dmax}{\mumin}$ and $\lmin \ge \frac{\dmin}{\mumax}$)} \\
                      & \le \frac{\mumax}{\mumin \cdot \left(\cmin \cdot \min\{\mumin, 1\}\right)^n}
                        && \text{(Lemma~\ref{lem:dmin-over-dmax})\;.}
 \end{align*}
 So the statement follows with Proposition~\ref{prop:blubblubblub}.
\qed
\end{proof}

The following consequence of Theorem~\ref{thm:estimate} removes some of the parameters
 on which the $k_\vf$ from Theorem~\ref{thm:estimate} depends.
\begin{theorem} \label{thm:estimate-cor}
Let $\vf$ be a quadratic, clean and feasible scSPP,
let $\mumin$ and $\mumax$ be the minimal and maximal component of $\fix{\vf}$, respectively,
and let the coefficients of~$\vf$ be given as ratios of $m$-bit integers.
Then
  \[
   \beta( k_\vf + i) \ge i \text{ for all $i\in\N$}
  \]
 holds for any of the following choices of~$k_\vf$.
 \begin{itemize}
  \item[1.]
   $\displaystyle \lceil 4mn + 3n \max\{ 0, - \log \mumin \} \rceil$;
  \item[2.]
   $\displaystyle 4mn2^n$;
  \item[3.]
   $\displaystyle 7mn$ whenever $\vf(\vzero) \succ \vzero$;
  \item[4.]
   $\displaystyle 2mn + m$ whenever both $\vf(\vzero) \succ \vzero$ and $\mumax \le 1$.
 \end{itemize} 
\end{theorem}
Items 3.\ and 4.\ of Theorem~\ref{thm:estimate-cor} apply in particular to termination SPPs of strict pPDAs
 (\S~\ref{subsec:stochastic-models}), i.e., they satisfy $\vf(\vzero) \gg \vzero$ and $\mumax \le 1$.

To prove Theorem~\ref{thm:estimate-cor} we need some relations between the parameters of~$\vf$.
We collect them in the following lemma.
\begin{lemma} \label{lem:parameter-relations}
 Let $\vf$ be a quadratic, clean and feasible scSPP.
 With the terminology of Theorem~\ref{thm:estimate} and Theorem~\ref{thm:estimate-cor} the following relations hold.
 \begin{itemize}
  \item[1.]
   $\cmin \ge 2^{-m}$.
  \item[2.]
   If $\vf(\vzero) \succ \vzero$ then $\mumin \ge \cmin$.
  \item[3.]
   If $\cmin > 1$ then $\mumin > 1$.
  \item[4.]
   If $\cmin \le 1$ then $\mumin \ge \cmin^{2^n-1}$.
  \item[5.]
   If $\vf$ is strictly quadratic, i.e.\ nonlinear, then the following inequalities hold: $\cmin \le 1$
    and $\mumax \cdot \cmin^{3n-2} \cdot \min\{\mumin^{2n-2},1\} \le 1$.
 \end{itemize}
\end{lemma}

\begin{proof}
 We show the relations in turn.
 \begin{itemize}
  \item[1.]
   The smallest nonzero coefficient representable as a ratio of $m$-bit numbers is $\frac{1}{2^m}$.
  \item[2.]
   As $\vf(\vzero) \succ \vzero$, in all components~$i$ there is a nonzero coefficient $c_i$ such that $f_i(\vzero) = c_i$.
   We have $\mu\vf \ge \vf(\vzero)$, so $\mu\vf_i \ge f_i(\vzero) = c_i \ge \cmin > 0$ holds for all~$i$.
   Hence $\mumin > 0$.
  \item[3.]
   Let $\cmin > 1$.
   Recall the Kleene sequence $(\ks{k})_{k\in\N}$ with $\ks{k} = \vf^k(\vzero)$.
   We first show by induction on $k$ that for all $k\in\N$ and all components $i$ either $\ksc{k}_i = 0$ holds or $\ksc{k}_i > 1$.
   For the induction base we have $\ks{0} = \vzero$.
   Let $k \ge 0$.
   Then $\ksc{k+1}_i = f_i(\ks{k})$ is a sum of products of numbers which are either coefficients of~$\vf$
    (and hence by assumption greater than $1$) or which are equal to~$\ksc{k}_j$ for some~$j$.
   By induction, $\ksc{k}_j$ is either $0$ or greater than~$1$.
   So, $\ksc{k+1}_i$ must be $0$ or greater than~$1$.

   By Theorem~\ref{thm:kleene}, the Kleene sequence converges to~$\mu\vf$.
   As $\vf$ is clean, we have $\mu\vf \succ \vzero$, and so there is a $k\in\N$ such that $\ks{k} \succ \vone$.
   The statement follows with $\mu\vf \ge \ks{k}$.
  \item[4.]
   Let $\cmin \le 1$.
   We prove the following stronger statement by induction on~$k$:
    For every $k$ with $0 \le k \le n$ there is a set $S_k \subseteq \{1, \ldots, n\}$, $|S_k| = k$, such that
     $\mu\vf_s \ge \cmin^{2^k-1}$ holds for all $s \in S_k$.
   The induction base ($k = 0$) is trivial.
   Let $k \ge 0$.
   Consider the SPP $\widehat{\vf}(\vX_{\{1,\ldots,n\} \setminus S_k})$
    that is obtained from~$\vf(\vX)$ by removing the $S_k$-components from $\vf$ and
    replacing every $S_k$-variable in the polynomials by the corresponding component of~$\mu\vf$.
   Clearly, $\mu\widehat{\vf} = (\mu\vf)_{\{1,\ldots,n\} \setminus S_k}$.
   \newcommand{\hcmin}{\widehat{c}_{\mathit{min}}}
   By induction, the smallest nonzero coefficient $\hcmin$ of~$\widehat{\vf}$
    satisfies $\hcmin \ge \cmin (\cmin^{2^k-1})^2 = \cmin^{2^{k+1}-1}$.
   Pick a component $i$ with ${\widehat{f}}_i(\vzero) > 0$.
   Then $\mu\widehat{\vf}_i \ge \widehat{f}_i(\vzero) \ge \hcmin \ge \cmin^{2^{k+1}-1}$.
   So set $S_{k+1} := S_k \cup \{i\}$.
  \item[5.]
   Let w.l.o.g.\ $\mumax = \mu\vf_1$.
   The proof is based on the idea that $X_1$ indirectly depends quadratically on itself.
   More precisely, as $\vf$ is strongly connected and strictly quadratic,
    component~$1$ depends (indirectly) on some component, say~$i_r$, such that $f_{i_r}$ contains a degree-2-monomial.
   The variables in that monomial, in turn, depend on~$X_1$.
   This gives an inequality of the form $\mu\vf_1 \ge C \cdot {\mu\vf_1}^2$,
    implying $\mu\vf_1 \cdot C \le 1$.

   We give the details in the following.
   As $\vf$ is strongly connected and strictly quadratic there exists a sequence of variables
    $X_{i_1}, \ldots, X_{i_r}$
   and a sequence of monomials
    $m_{i_1}, \ldots, m_{i_r}$
    ($1 \le r \le n$)
    with the following properties:
   $$\begin{array}{ll}
    \text{ --\quad} X_{i_1}  = X_1, \\
    \text{ --\quad} m_{i_u} \text{ is a monomial appearing in $f_{i_u}$ }                    & (1 \le u \le r), \\
    \text{ --\quad} m_{i_u}  = c_{i_u} \cdot X_{i_{u+1}}                                       & (1 \le u \le r), \\
    \text{ --\quad} m_{i_r}  = c_{i_r} \cdot X_{j_1} \cdot X_{k_1} \text{ for some variables $X_{j_1}, X_{k_1}$.}\\
   \end{array}$$
   Notice that
    \begin{equation} \label{eq:mumax1}
     \begin{split}
       \mumax = \mu\vf_1 & \ge c_{i_1} \cdot \ldots \cdot c_{i_r} \cdot \mu\vf_{j_1} \cdot \mu\vf_{k_1} \\
                              & \ge \min(\cmin^{n}, 1) \cdot \mu\vf_{j_1} \cdot \mu\vf_{k_1} \:. \\
     \end{split}
    \end{equation}

   Again using that $\vf$ is strongly connected, there exists a sequence of variables
    $X_{j_1}, \ldots, X_{j_s}$
   and a sequence of monomials
    $m_{j_1}, \ldots, m_{j_{s-1}}$
    ($1 \le s \le n$)
    with the following properties:
   $$\begin{array}{ll}
    \text{ --\quad} X_{j_s}  = X_1, \\
    \text{ --\quad} m_{j_u} \text{ is a monomial appearing in $f_{j_u}$ } (1 \le u \le s-1), \\
    \text{ --\quad} m_{j_u}  = c_{j_u} \cdot X_{j_{u+1}}
         \text{ or $m_{j_u}   = c_{j_u} \cdot X_{j_{u+1}} \cdot X_{j'_{u+1}}$ } \\
         \text{\hspace{20mm} for some variable $X_{j'_{u+1}} \quad (1 \le u \le s-1)$.}
   \end{array}$$
   Notice that
    \begin{equation} \label{eq:mumax2}
     \begin{split}
       \mu\vf_{j_1} & \ge c_{j_1} \cdot \ldots \cdot c_{j_{s-1}} \cdot \min(\mumin^{s-1}, 1) \cdot \mu\vf_1 \\
                         & \ge \min(\cmin^{n-1}, 1) \cdot \min(\mumin^{n-1}, 1) \cdot \mu\vf_1 \:. \\
      \end{split}
    \end{equation}
   Similarly, there exists a sequence of variables $X_{k_1}, \ldots, X_{k_t}$ ($1 \le t \le n$) with $X_{k_t} = X_1$ showing
    \begin{equation} \label{eq:mumax3}
      \mu\vf_{k_1} \ge \min(\cmin^{n-1}, 1) \cdot \min(\mumin^{n-1}, 1) \cdot \mu\vf_1 \:.
    \end{equation}
   Combining (\ref{eq:mumax1}) with (\ref{eq:mumax2}) and (\ref{eq:mumax3}) yields
    \begin{equation*}
      \mumax \ge \min(\cmin^{3n-2}, 1) \cdot \min(\mumin^{2n-2}, 1) \cdot \mumax^2 \:,
    \end{equation*}
   or
    \begin{equation} \label{eq:mumax4}
      \mumax \cdot \min(\cmin^{3n-2}, 1) \cdot \min(\mumin^{2n-2}, 1) \le 1 \:.
    \end{equation}
   Now it suffices to show $\cmin \le 1$.
   Assume for a contradiction $\cmin > 1$.
   Then, by statement 3., $\mumin > 1$.
   Plugging this into (\ref{eq:mumax4}) yields $\mumax \le 1$.
   This implies $\mumax < \mumin$, contradicting the definition of $\mumax$ and $\mumin$.
 \end{itemize}
\qed
\end{proof}

Now we are ready to prove Theorem~\ref{thm:estimate-cor}.
\begin{proof}[of Theorem~\ref{thm:estimate-cor}]
 \begin{itemize}
  \item[1.]
   First we check the case where $\vf$ is linear, i.e., all polynomials $f_i$ have degree at most $1$.
   In this case, Newton's method reaches $\fix{\vf}$ after one iteration, so the statement holds.
   Consequently, we can assume in the following that $\vf$ is strictly quadratic,
    meaning that $\vf$ is quadratic and there is a polynomial in~$\vf$ of degree $2$.

   By Theorem~\ref{thm:estimate} it suffices to show
   \[ \log \frac{\mumax}{\mumin \cdot \cmin^n \cdot \min\{\mumin^n, 1\}}
      \le 4mn + 3n \max\{ 0, - \log \mumin \}\:.
   \]
   We have
   \begin{align*}
        & \ \log \frac{\mumax}{\mumin \cdot \cmin^n \cdot \min\{\mumin^n, 1\}} \\
    \le & \ \log \frac{1}{\cmin^{4n-2} \cdot \min\{\mumin^{3n-1},1\}}
             && \text{(Lemma~\ref{lem:parameter-relations}.5.)} \\
    \le & \ 4n \cdot \log \frac{1}{\cmin} - \log (\min\{\mumin^{3n-1}, 1\})
             && \text{(Lemma~\ref{lem:parameter-relations}.5.: $\cmin \le 1$)} \\
    \le & \ 4mn                           - \log (\min\{\mumin^{3n-1}, 1\})
             && \text{(Lemma~\ref{lem:parameter-relations}.1.)}\:.
   \end{align*}
   If $\mumin \ge 1$ we have $- \log (\min\{\mumin^{3n-1}, 1\}) \le 0$, so we are done in this case.
   If $\mumin \le 1$ we have
    $- \log (\min\{\mumin^{3n-1}, 1\}) = - (3n-1) \log \mumin \le 3n \cdot (-\log \mumin)$.
  \item[2.]
   By statement 1.\ of this theorem, it suffices to show that $4mn + 3n \max\{ 0, - \log \mumin \} \le 4mn2^n$.
   This inequality obviously holds if $\mumin \ge 1$.
   So let $\mumin \le 1$.
   Then, by Lemma~\ref{lem:parameter-relations}.3., $\cmin \le 1$.
   Hence, by Lemma~\ref{lem:parameter-relations} parts 4.\ and~1., $\mumin \ge \cmin^{2^n-1} \ge 2^{-m(2^n-1)}$.
   So we have an upper bound on $-\log \mumin$ with $-\log \mumin \le m(2^n-1)$ and get:
   \begin{align*}
    4mn + 3n \max\{ 0, - \log \mumin \}
    & \le 4mn + 3n m (2^n-1) \\
    & \le 4mn + 4n m (2^n-1) = 4mn2^n
   \end{align*}
  \item[3.]
   Let $\vf(\vzero) \succ \vzero$.
   By statement 1.\ of this theorem it suffices to show that $4mn + 3n \max\{ 0, - \log \mumin \} \le 7mn$ holds.
   By Lemma~\ref{lem:parameter-relations} parts 2.\ and~1., we have $\mumin \ge \cmin \ge 2^{-m}$, so $-\log \mumin \le m$.
   Hence, $4mn + 3n \max\{ 0, - \log \mumin \} \le 4mn + 3nm = 7mn$.
  \item[4.]
   Let $\vf(\vzero) \succ \vzero$ and $\mumax \le 1$.
   By Theorem~\ref{thm:estimate} it suffices to show that
    $\displaystyle\log \frac{\mumax}{\mumin \cdot \cmin^n \cdot \min\{\mumin^n, 1\}} \le 2mn + m$.
   We have:
   \begin{align*}
     & \log \frac{\mumax}{\mumin \cdot \cmin^n \cdot \min\{\mumin^n, 1\}} \\
     & \le -n\log \cmin - (n+1)\log \mumin && \text{(as $\mumin \le \mumax \le 1$)} \\
     & \le -(2n+1) \log \cmin              && \text{(Lemma~\ref{lem:parameter-relations}.2.)} \\
     & \le 2mn + m                         && \text{(Lemma~\ref{lem:parameter-relations}.1.)}
   \end{align*}
 \end{itemize}
\qed
\end{proof}

\subsection{Upper Bounds on the Least Fixed Point Via Newton Approximants} \label{sub:upper-bounds}
By Theorem~\ref{thm:well-defined} each Newton approximant $\ns{k}$ is a lower bound on~$\mu\vf$.
Theorem~\ref{thm:estimate} and Theorem~\ref{thm:estimate-cor} give us upper bounds on the error $\Ds{k} := \mu\vf - \ns{k}$.
Those bounds can directly transformed into upper bounds on~$\mu\vf$,
 as $\mu\vf = \ns{k} + \Ds{k}$, cf.\ Example~\ref{ex:thm-estimate}.

Theorem~\ref{thm:estimate} and Theorem~\ref{thm:estimate-cor} allow to compute bounds on~$\Ds{k}$
 even before the Newton iteration has been started.
However, this may be more than we actually need.
In practice, we may wish to use an iterative method that yields guaranteed lower {\em and upper} bounds on~$\mu\vf$
 that improve during the iteration.
The following theorem and its corollary can be used to this end.
\begin{theorem} \label{thm:proximity}
 Let $\vf$ be a quadratic, clean and feasible scSPP.
 Let $\vzero \le \vx \le \mu\vf$ and $\vx \le \vf(\vx)$ such that $\vf'(\vx)^*$ exists.
 Let $\cmin$ be the smallest nonzero coefficient of~$\vf$ and $\mumin$ the minimal component of~$\mu\vf$.
 Then
  \[
   \frac{\norm{\Ne(\vx) - \vx}_\infty}{\norm{\mu\vf - \Ne(\vx)}_\infty} \ge \left( \cmin \cdot \min\{\mumin, 1\} \right)^n\;.
  \]
\end{theorem}

We prove Theorem~\ref{thm:proximity} at the end of the section.
The theorem can be applied to the Newton approximants:
\begin{theorem} \label{thm:proximity-2}
 Let $\vf$ be a quadratic, clean and feasible scSPP.
 Let $\cmin$ be the smallest nonzero coefficient of~$\vf$ and $\mumin$ the minimal component of~$\mu\vf$.
 For all Newton approximants $\ns{k}$ with $\ns{k} \succ \vzero$, let $\nsmin$ be the smallest coefficient of~$\ns{k}$.
 Then
   \[ 
    \ns{k} \le \mu\vf \le \ns{k} + \Vector{\frac{\norm{\ns{k} - \ns{k-1}}_\infty}{\left( \cmin \cdot \min\{\nsmin, 1\} \right)^n}}
   \] 
 where $\Vector{s}$ denotes the vector $\vx$ with $x_j = s$ for all $1 \le j \le n$.
\end{theorem}
\begin{proof}[of Theorem~\ref{thm:proximity-2}]
 Theorem~\ref{thm:proximity} applies, due to Theorem~\ref{thm:well-defined}, to the Newton approximants with $\vx = \ns{k-1}$.
 So we get
  \begin{align*}
   \norm{\mu\vf - \ns{k}}_\infty & \le \frac{\norm{\ns{k} - \ns{k-1}}_\infty}{\left( \cmin \cdot \min\{\mumin, 1\} \right)^n} \\
                                 & \le \frac{\norm{\ns{k} - \ns{k-1}}_\infty}{\left( \cmin \cdot \min\{\nsmin, 1\} \right)^n}
                                     && \text{(as $\ns{k} \le \mu\vf$)}\;.
  \end{align*}
 Hence the statement follows from $\ns{k} \le \mu\vf$.
\qed
\end{proof}

\begin{example} \label{ex:thm-proximity-1}
 Consider again the equation $\vX = \vf(\vX)$ from Examples \ref{ex:back-button} and~\ref{ex:thm-estimate}:
 \[
  \begin{pmatrix}
   X_1 \\ X_2 \\ X_3
  \end{pmatrix}
  =
  \begin{pmatrix}
   0.4 X_2 X_1 + 0.6 \\
   0.3 X_1 X_2 + 0.4 X_3 X_2 + 0.3 \\
   0.3 X_1 X_3 + 0.7
  \end{pmatrix}
 \]
 Again we wish to verify that there is no component $s \in \{1,2,3\}$ with $\mu\vf_s = 1$.
 Performing $10$ Newton steps yields an approximation $\ns{10}$ to~$\mu\vf$ with
 \[
  \begin{pmatrix}
   0.9828 \\
   0.9738 \\
   0.9926
  \end{pmatrix}
  \prec
  \ns{10}
  \prec
  \begin{pmatrix}
   0.9829 \\
   0.9739 \\
   0.9927
  \end{pmatrix} \:.
 \]
 Further, it holds $\norm{\ns{10} - \ns{9}}_\infty \le 2 \cdot 10^{-6}$.
 So we have
  \[
   \frac{\norm{\ns{10} - \ns{9}}_\infty}{\left( \cmin \cdot \min\{\nsc{10}_{\mathit{min}}, 1\} \right)^3} \le
   \frac{2 \cdot 10^{-6}}{\left( 0.3 \cdot 0.97 \right)^3} \le 0.00009
  \]
 and hence by Theorem~\ref{thm:proximity-2}
 \[
  \ns{10} \le
  \fix{\vf} \le
  \ns{10}
  +
  \Vector{0.00009}
  \le
  \begin{pmatrix}
   0.983\\
   0.974\\
   0.993\\
  \end{pmatrix}
 \]
 In particular we know that $\mu\vf_s < 1$ for all three components~$s$.
\qed
\end{example}
\begin{example} \label{ex:thm-proximity-2}
 Consider again the SPP~$\vf$ from Example~\ref{ex:thm-proximity-1}.
 \newcommand{\us}[1]{\vu^{(#1)}}
 \newcommand{\plower}{p_{\mathit{lower}}}
 \newcommand{\pupper}{p_{\mathit{upper}}}
 Setting
  \[
   \us{k} := \ns{k} + \Vector{\frac{\norm{\ns{k} - \ns{k-1}}_\infty}{\left( 0.3 \cdot \nsmin \right)^3}}\;,
  \]
 Theorem~\ref{thm:proximity-2} guarantees
  \[
   \ns{k} \le \mu\vf \le \us{k}\;.
  \]
 Let us measure the tightness of the bounds $\ns{k}$ and $\us{k}$ on~$\mu\vf$ in the first component.
 Let
  \begin{align*}
   \plower(k) & := - \log_{2} (\mu\vf_1  - \nsc{k}_1) \qquad \text{and} \\
   \pupper(k) & := - \log_{2} (u^{(k)}_1 - \mu\vf_1) \;.
  \end{align*}
 Roughly speaking, $\nsc{k}_1$ and $u^{(k)}_1$ have $\plower(k)$ and $\pupper(k)$ valid bits of~$\mu\vf_1$, respectively.
 Figure~\ref{fig:proximity} shows $\plower(k)$ and $\pupper(k)$ for $k \in \{1,\ldots, 11\}$.

 It can be seen that the slope of~$\plower(k)$ is approximately $1$ for $k=2,\ldots,6$.
 This corresponds to the linear convergence of Newton's method according to Theorem~\ref{thm:estimate}.
 Since $\Id - \vf'(\mu\vf)$ is non-singular%
\footnote{In fact, the matrix is ``almost'' singular, with $\det(\Id - \vf'(\mu\vf)) \approx 0.006$.},
 Newton's method actually has, asymptotically, an exponential convergence order, cf.\ Theorem~\ref{thm:quadratic-convergence}.
 This behavior can be observed in Figure~\ref{fig:proximity} for $k \ge 7$.
 For $\pupper$, we roughly have (using $\ns{k} \approx \mu\vf$):
  \[ \pupper(k) \approx \plower(k-1) + \log \left( 0.3 \cdot \nsmin \right)^3 \approx
                        \plower(k-1) - 5\;.
  \]
 \begin{figure}[ht]
 \begin{center}
 {
  \psfrag{\26110}{$-10$}
  \psfrag{0}{$0$}
  \psfrag{1}{$1$}
  \psfrag{2}{$2$}
  \psfrag{3}{$3$}
  \psfrag{4}{$4$}
  \psfrag{5}{$5$}
  \psfrag{6}{$6$}
  \psfrag{7}{$7$}
  \psfrag{8}{$8$}
  \psfrag{9}{$9$}
  \psfrag{10}{$10$}
  \psfrag{11}{$11$}
  \psfrag{12}{$12$}
  \psfrag{20}{$20$}
  \psfrag{30}{$30$}
  \psfrag{40}{$40$}
  \psfrag{50}{$50$}
  \psfrag{60}{$60$}
  \psfrag{plower}{$\plower(k)$}
  \psfrag{pupper}{$\pupper(k)$}
  \psfrag{k}{$k$}
   \scalebox{1.0}{ \includegraphics{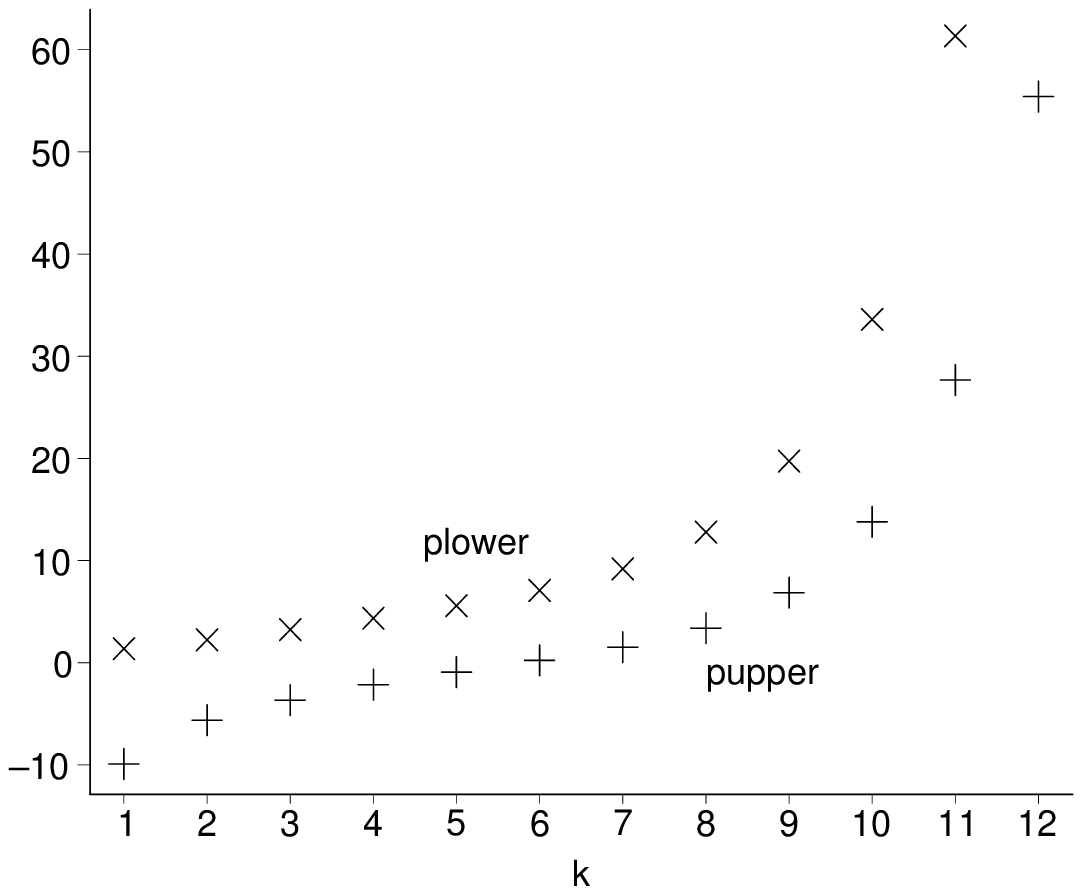}}
 }
 \end{center}
 \caption{Number of valid bits of the
          lower (shape: $\boldsymbol{\times}$) and upper (shape: $\boldsymbol{+}$) bounds on~$\mu\vf_1$,
          see Example~\ref{ex:thm-proximity-2}.
 }
 \label{fig:proximity}
\end{figure}
\qed
\end{example}

The proof of Theorem~\ref{thm:proximity} uses techniques similar to those of the proof of Theorem~\ref{thm:estimate},
 in particular Lemma~\ref{lem:dmin-over-dmax}.
\begin{proof}[of Theorem~\ref{thm:proximity}]
 By Proposition~\ref{prop:cone-vector-exists}, $\vf$ has a cone vector~$\vd$.
 Let $\dmin$ and $\dmax$ be the smallest and the largest component of~$\vd$, respectively.
 Let $\lmax := \max_j \{ \frac{\mu\vf_j - x_j}{d_j} \}$,
  and let w.l.o.g.\ $\lmax = \frac{\mu\vf_1 - x_1}{d_1}$.
 We have $\vx \ge \mu\vf - \lmax \vd$, so we can apply Lemma~\ref{lem:blub2} to obtain $\Ne(\vx) \ge \mu\vf - \frac{1}{2} \lmax \vd$.
 Thus
  \[
   \norm{\Ne(\vx) - \vx}_\infty \ge \left( \Ne(\vx) - \vx \right)_1
     \ge \mu\vf_1 - \frac{1}{2} \lmax d_1 - x_1 = \frac{1}{2} \lmax d_1 \ge \frac{1}{2} \lmax \dmin\;.
  \]
 On the other hand, with Lemma~\ref{lem:step-stays-below-mu} we have $\vzero \le \mu\vf - \Ne(\vx) \le \frac{1}{2} \lmax \vd$ and so
  $\norm{\mu\vf - \Ne(\vx)}_\infty \le \frac{1}{2} \lmax \dmax$.
 Combining those inequalities we obtain
  \[
   \frac{\norm{\Ne(\vx) - \vx}_\infty}{\norm{\mu\vf - \Ne(\vx)}_\infty} \ge \frac{\dmin}{\dmax}\;.
  \]
 Now the statement follows from Lemma~\ref{lem:dmin-over-dmax}.
\qed
\end{proof}

%% file: sec5-decomposed.tex
\section{General SPPs} \label{sec:decomposed}

In \S~\ref{sec:scSPPs} we considered {\em strongly connected} SPPs, see Definition~\ref{def:depend}.
However, it is not always guaranteed that the SPP~$\vf$ is strongly connected.
In this section we analyze the convergence speed of two variants of Newton's method that both compute approximations of~$\mu\vf$,
 where $\vf$ is a clean and feasible SPP that is not necessarily strongly connected (``general SPPs'').

The first one was suggested by Etessami and Yannakakis \cite{EYstacs05Extended} and is called {\em Decomposed Newton's Method (DNM)}.
It works by running Newton's method separately on each SCC, see \S~\ref{sub:DNM}.
The second one is the regular Newton's method from \S~\ref{sec:well-defined}.
We will analyze its convergence speed in \S~\ref{sub:conv-speed-general}.

The reason why we first analyze DNM is that our convergence speed results about Newton's method for general SPPs
 (Theorem~\ref{thm:conv-speed-general}) build on our results about DNM (Theorem~\ref{thm:dnm-error}).
From an efficiency point of view it actually may be advantageous to run Newton's method separately on each SCC.
For those reasons DNM deserves a separate treatment.

\subsection{Convergence Speed of the Decomposed Newton's Method (DNM)} \label{sub:DNM}

\newcommand{\SCC}{\mathcal{SCC}}

DNM, originally suggested in~\cite{EYstacs05Extended}, works as follows.
It starts by using Newton's method for each bottom SCC, say~$S$, of the SPP~$\vf$.
Then the corresponding variables $\vX_S$ are substituted for the obtained approximation for $\mu\vf_S$,
 and the corresponding equations $\vX_S = \vf_S(\vX)$ are removed.
The same procedure is then applied to the new bottom SCCs, until all SCCs have been processed.

Etessami and Yannakakis did not provide a particular criterion for the number of Newton iterations to be applied in each SCC.
Consequently, they did not analyze the convergence speed of DNM.
We will treat those issues in this section, thereby taking advantage of our previous analysis of scSPPs.

We fix a quadratic, clean and feasible SPP~$\vf$ for this section.
We assume that we have already computed the DAG (directed acyclic graph) of SCCs.
This can be done in linear time in the size of~$\vf$.
To each SCC~$S$ we can associate its {\em depth}~$t$: it is the longest path in the DAG of SCCs from $S$ to a top SCC.
Notice that $0 \le t \le n-1$.
We write $\SCC(t)$ for the set of SCCs of depth~$t$.
We define the height $h(\vf)$ as the largest depth of an SCC
 and the width $w(\vf) := \max_t |\SCC(t)|$ as the largest number of SCCs of the same depth.
Notice that $\vf$ has at most $(h(\vf) + 1) \cdot w(\vf)$ SCCs.
Further we define the component sets $[t] := \bigcup_{S \in \SCC(t)} S$ and $[\mathord{>}t] := \bigcup_{t' > t} [t']$ and similarly $[<t]$.

\begin{figure} [ht]
        \centering
        \fbox{\parbox{11cm}{ \flushleft \iftechrep{\vspace{-3mm}}{}
            \textbf{function} DNM $\left(\vf, i\right)$              \hfill /* \emph{The parameter $i$ controls the precision.} */\\
            \ind \textbf{for} $t$ \textbf{from} $h(\vf)$ \textbf{downto} $0$ \\
            \indd \textbf{forall} $S \in \SCC(t)$                     \hfill /* \emph{for all SCCs $S$ of depth $t$} */ \\
            \inddd $\rhos{i}_S$ := $\Ne_{\vf_S}^{i\cdot 2^t}(\vzero)$ \hfill /* \emph{perform $i \cdot 2^t$ Newton iterations}  */ \\
            \inddd $\vf_{[<t]}$ := $\vf_{[<t]}[S/\rhos{i}_S]$     \hfill /* \emph{apply $\rhos{i}_S$ in the upper SCCs} */ \\
            \textbf{return} $\rhos{i}$  \\ 
        }}
        \caption{Decomposed Newton's Method (DNM) for computing an approximation $\rhos{i}$ of~$\fix{\vf}$.}
        \label{fig:DNM}
\end{figure}

Figure~\ref{fig:DNM} shows our version of~DNM.
We suggest to run Newton's method in each SCC~$S$ for a number of steps that depends (exponentially) on the depth of~$S$
 and (linearly) on a parameter~$i$ that controls the precision.

\begin{proposition} \label{prop:DNM-how-many-iterations}
 The function \textup{DNM}$\left(\vf, i\right)$ of Figure~\ref{fig:DNM} runs at most
  \mbox{$i \cdot w(\vf) \cdot 2^{h(\vf) + 1} \le i \cdot n \cdot 2^n$} iterations of Newton's method.
\end{proposition}

\begin{proof}
 The number of iterations is $\sum_{t=0}^{h(\vf)} \abs{\SCC(t)} \cdot i \cdot 2^t$.
 This can be estimated as follows.
  \begin{align*}
    \sum_{t=0}^{h(\vf)} \abs{\SCC(t)} \cdot i \cdot 2^t & \le w(\vf) \cdot i \cdot \sum_{t=0}^{h(\vf)} 2^t \\
                                                        & \le w(\vf) \cdot i \cdot 2^{h(\vf) + 1} \\
                                                        & \le i \cdot n \cdot 2^n && \text{(as $w(\vf) \le n$ and $h(\vf) < n$)}
  \end{align*}
\qed
\end{proof}

The following theorem states that DNM has linear convergence order.
\begin{theorem}\label{thm:dnm-error}
 Let $\vf$ be a quadratic, clean and feasible SPP.
 Let $\rhos{i}$ denote the result of calling \textup{DNM}$(\vf, i)$ (see Figure~\ref{fig:DNM}).
 Let $\beta_\vrho$ denote the convergence order of~$(\rhos{i})_{i\in\N}$.
 Then there is a $k_\vf \in \Nat$ such that $\beta_\vrho(k_\vf + i) \ge i$ for all $i \in \N$.
\end{theorem}

Theorem~\ref{thm:dnm-error} can be interpreted as follows:
Increasing $i$ by one yields asymptotically at least one additional bit in each component
 and, by Proposition~\ref{prop:DNM-how-many-iterations}, costs at most $n \cdot 2^n$ additional Newton iterations.
Notice that for simplicity we do not take into account here that
 the cost of performing a Newton step on a single SCC is not uniform,
 but rather depends on the size of the SCC
 (e.g.\ cubically if Gaussian elimination is used for solving the linear systems).

For the proof of Theorem~\ref{thm:dnm-error},
 let $\Ds{i}$ denote the error when running DNM with parameter~$i$, i.e., $\Ds{i}$ := $\mu\vf - \rhos{i}$.
Observe that the error $\Ds{i}$ can be understood as the sum of two errors:
\[ \Ds{i} :=
   \mu\vf - \rhos{i} = (\vmu - \tvmu^{(i)}) + (\tvmu^{(i)} - \rhos{i}) \:,
\]
where $\tvmu_{[t]}^{(i)}$ := $\fix{\big(\vf_{[t]}[[\mathord{>}t] / \rhos{i}_{[\mathord{>}t]}]\big)}$,
 i.e., $\tvmu_{[t]}^{(i)}$ is the least fixed point of $\vf_{[t]}$
 after the approximations from the lower SCCs have been applied.
So, $\Ds{i}_{[t]}$ consists of the {\em propagation error} $(\mu\vf_{[t]} - \tvmu_{[t]}^{(i)})$
 (resulting from the error at lower SCCs)
 and the {\em approximation error} $(\tvmu_{[t]}^{(i)} - \rhos{i}_{[t]})$
 (resulting from the newly added error of Newton's method on level~$t$).

The following lemma gives a bound on the propagation error.
\newcommand{\stmtlempropagationerror}{
 There is a constant $C_\vf > 0$ such that
 \[ \norm{\mu\vf_{[t]} - \tvmu_{[t]}} \le C_\vf \cdot \sqrt{\norm{\mu\vf_{[\mathord{>}t]} - \vrho_{[\mathord{>}t]}}}
 \]
 holds for all $\vrho_{[\mathord{>}t]}$ with $\vzero \le \vrho_{[\mathord{>}t]} \le \mu\vf_{[\mathord{>}t]}$,
  where $\tvmu_{[t]} = \fix{\big(\vf_{[t]}[[\mathord{>}t] / \vrho_{[\mathord{>}t]}]\big)}$.
}
\begin{lemma}[Propagation error]\label{lem:propagation-error}
 \stmtlempropagationerror
\end{lemma}

Roughly speaking, Lemma~\ref{lem:propagation-error} states that if $\rhos{i}_{[\mathord{>}t]}$ has $k$ valid bits of $\mu\vf_{[\mathord{>}t]}$,
 then $\tvmu_{[t]}^{(i)}$ has at least about $k/2$ valid bits of $\mu\vf_{[t]}$.
In other words, (at most) one half of the valid bits are lost on each level of the DAG due to the propagation error.
The proof of Lemma~\ref{lem:propagation-error} is technically involved
 and, unfortunately, not constructive in that we know nothing about~$C_\vf$ except for its existence.
Therefore, the statements in this section are independent of a particular norm.
The proof of Lemma~\ref{lem:propagation-error} can be found in Appendix~\ref{app:proof-lem-propagation-error}.

The following lemma gives a bound on the error $\norm{\Ds{i}_{[t]}}$ on level $t$,
 taking both the propagation error and the approximation error into account.
\begin{lemma}\label{lem:dnm-error}
 There is a $C_\vf > 0$ such that
 $\displaystyle \norm{\Ds{i}_{[t]}} \le 2^{C_\vf - i \cdot 2^t}$ for all $i\in\N$.
\end{lemma}
\begin{proof}
 Let $\widetilde{\vf}_{[t]}^{(i)} := \vf_{[t]}[[>\!t] / \rhos{i}_{[\mathord{>}t]}]$.
 Observe that the coefficients of~$\widetilde{\vf}_{[t]}^{(i)}$ and thus its least fixed point~$\tvmu_{[t]}^{(i)}$
  are monotonically increasing with~$i$, because $\rhos{i}_{[\mathord{>}t]}$ is monotonically increasing as well.
 Consider an arbitrary depth~$t$ and choose real numbers $\cmin > 0$ and $\mumin > 0$ and an integer $i_0$ such that,
  for all $i \ge i_0$, $\cmin$ and $\mumin$ are lower bounds on the smallest nonzero coefficient of~$\widetilde{\vf}_{[t]}^{(i)}$
  and the smallest coefficient of~$\tvmu_{[t]}^{(i)}$, respectively.
 Let $\mumax$ be the largest component of~$\mu\vf_{[t]}$.
 Let $\widetilde{k} := \left\lceil n \cdot \log \frac{\mumax}{\cmin \cdot \mumin \cdot \min\{\mumin, 1\}} \right\rceil$.
 Then it follows from Theorem~\ref{thm:estimate} that performing $\widetilde{k} + j$ Newton iterations
  ($j \ge 0$) on depth~$t$ yields $j$ valid bits of~$\tvmu_{[t]}^{(i)}$ for any $i \ge i_0$.
 In particular, $\widetilde{k} + i \cdot 2^t$ Newton iterations give $i \cdot 2^t$ valid bits of~$\tvmu_{[t]}^{(i)}$
  for any $i \ge i_0$.
 So there exists a constant~$c_1 > 0$ such that, for all $i \ge i_0$,
  \begin{equation}
   \norm{\tvmu_{[t]}^{(i)} - \rhos{i}_{[t]}} \le 2^{c_1 - i \cdot 2^t} \:, \label{eq:approximation-error}
  \end{equation}
 because DNM (see Figure~\ref{fig:DNM}) performs $i \cdot 2^t$ iterations to compute~$\rhos{i}_S$ where $S$ is an SCC of depth~$t$.
 Choose $c_1$ large enough such that Equation~\eqref{eq:approximation-error} holds for all $i \ge 0$ and all depths~$t$.

 Now we can prove the theorem by induction on $t$.
 In the base case ($t = h(\vf)$) there is no propagation error,
  so the claim of the lemma follows from~\eqref{eq:approximation-error}.
 Let $t < h(\vf)$. Then
  \begin{align*}
   \norm{\Ds{i}_{[t]}} &  =  \norm{\mu\vf_{[t]} - \tvmu_{[t]}^{(i)} + \tvmu_{[t]}^{(i)} - \rhos{i}_{[t]}} \\
        & \le \norm{\mu\vf_{[t]} - \tvmu_{[t]}^{(i)}}  + \norm{\tvmu_{[t]}^{(i)} - \rhos{i}_{[t]}} \\
        & \le \norm{\mu\vf_{[t]} - \tvmu_{[t]}^{(i)}}  + 2^{c_1 - i \cdot 2^t}
                                 &&  \text{ (by \eqref{eq:approximation-error})} \\
        & \le c_2 \cdot \sqrt{\norm{\Ds{i}_{[\mathord{>}t]}}}       + 2^{c_1 - i \cdot 2^t}
                                 &&  \text{ (Lemma~\ref{lem:propagation-error})} \\
        & \le c_2 \cdot \sqrt{2^{c_3 - i \cdot 2^{t+1}}}+ 2^{c_1 - i \cdot 2^t}
                                 &&  \text{ (induction hypothesis)} \\
        & \le 2^{c_4 - i \cdot 2^t}
  \end{align*}
  for some constants $c_2, c_3, c_4 > 0$.
\qed
\end{proof}

Now Theorem~\ref{thm:dnm-error} follows easily.
\begin{proof}[of Theorem~\ref{thm:dnm-error}]
 From Lemma~\ref{lem:dnm-error} we deduce that for each component $j \in [t]$ there is a $c_j$ such that
  \[ (\fix{\vf}_j - \rhosc{i}_j) / \fix{\vf}_j \le 2^{c_j - i \cdot 2^t} \le 2^{c_j - i} \: .
  \]
 Let $k_\vf \ge c_j$ for all $1 \le j \le n$.
 Then
  \[ (\fix{\vf}_j - \rhosc{i + k_\vf}_j) / \fix{\vf}_j \le 2^{c_j - (i + k_\vf)} \le 2^{-i} \:.
  \]
\qed
\end{proof}

Notice that, unfortunately, we cannot give a bound on~$k_\vf$,
 mainly because Lemma~\ref{lem:propagation-error} does not provide a bound on~$C_\vf$.

\subsection{Convergence Speed of Newton's Method} \label{sub:conv-speed-general}

We use Theorem~\ref{thm:dnm-error} to prove the following theorem for the regular
 (i.e.\ not decomposed) Newton sequence~$(\ns{i})_{i\in\N}$.

\begin{theorem} \label{thm:conv-speed-general}
Let $\vf$ be
a quadratic, clean and feasible SPP. There is a threshold $k_\vf \in \Nat$ such that
$\beta(k_\vf + i \cdot n \cdot 2^n) \ge \beta(k_\vf + i \cdot (h(\vf) + 1) \cdot 2^{h(\vf)}) \ge i$ for all $i\in\N$.
\end{theorem}

In the rest of the section we prove this theorem by a sequence of lemmata.
The following lemma states that a Newton step is not faster on an SCC, if the values of the lower SCCs are fixed.
\begin{lemma}\label{lem:do-not-fix-values}
 Let $\vf$ be a clean and feasible SPP.
 Let $\vzero \le \vx \le \vf(\vx) \le \fix{\vf}$ such that $\vf'(\vx)^*$ exists.
 Let $S$ be an SCC of~$\vf$ and let $L$ denote the set of components that are not in~$S$, but on which a variable in~$S$ depends.
 Then $(\Ne_\vf(\vx))_S \ge \Ne_{\vf_S[L / \vx_L]}(\vx_S)$.
\end{lemma}
\begin{proof}
  \begin{align*}
    (\Ne_\vf(\vx))_S  &  =  \bigl( \vf'(\vx)^* (\vf(\vx) - \vx) \bigr)_S \\
                      &  =  \vf'(\vx)^*_{SS} (\vf(\vx) - \vx)_S + \vf'(\vx)^*_{SL} (\vf(\vx) - \vx)_L \\
                      & \ge \vf'(\vx)^*_{SS} (\vf(\vx) - \vx)_S \\
                      &  = \bigl( (\vf_S[L / \vx_L])'(\vx_S) \bigr)^* (\vf_S [L / \vx_L] (\vx_S) - \vx_S) \\
                      &  = \Ne_{\vf_S[L / \vx_L]}(\vx_S)
  \end{align*}
\qed
\end{proof}

Recall Lemma~\ref{lem:newton-mon} which states that the Newton operator~$\Ne$ is monotone.
This fact and Lemma~\ref{lem:do-not-fix-values}
 can be combined to the following lemma stating that $i \cdot (h(\vf) + 1)$ iterations of the regular Newton's method
 ``dominate'' a decomposed Newton's method that performs $i$ Newton steps in each SCC.

\begin{lemma}\label{lem:undecomposed-at-least-as-good}
 Let $\tns{i}$ denote the result of a decomposed Newton's method which performs $i$ iterations
  of Newton's method in each SCC.
 Let $\ns{i}$ denote the result of $i$ iterations of the regular Newton's method.
 Then $\ns{i \cdot (h(\vf) + 1)} \ge \tns{i}$.
\end{lemma}
\begin{proof}
 Let $h = h(\vf)$.
 Let $[t]$ and $[\mathord{>}t]$ again denote the set of components of depth $t$ and~$>t$, respectively.
 We show by induction on the depth~$t$:
  \[ \ns{i \cdot (h + 1 - t)}_{[t]} \ge \tns{i}_{[t]}
  \]
 The induction base ($t = h$) is clear, because for bottom SCCs the two methods are identical.
 Let now $t < h$.
 Then
  \begin{align*}
   \ns{i \cdot (h + 1 - t)}_{[t]}
    &  =  \Ne_\vf^i(\ns{i \cdot (h - t)})_{[t]} \\
    & \ge \Ne_{\vf_{[t]} [[\mathord{>}t] / \ns{i \cdot (h - t)}_{[\mathord{>}t]}]}^i ( \ns{i \cdot (h - t)}_{[t]})
              && \text{(Lemma~\ref{lem:do-not-fix-values})} \\
    & \ge \Ne_{\vf_{[t]} [[\mathord{>}t] / \tns{i}_{[\mathord{>}t]}]}^i ( \ns{i \cdot (h - t)}_{[t]})
              && \text{(induction hypothesis)} \\
    & \ge \Ne_{\vf_{[t]} [[\mathord{>}t] / \tns{i}_{[\mathord{>}t]}]}^i ( \vzero_{[t]} )
              && \text{(Lemma~\ref{lem:newton-mon})} \\
    &  =  \tns{i}_{[t]}
              && \text{(definition of~$\tns{i}$)}
  \end{align*}
 Now, the lemma itself follows by using Lemma~\ref{lem:newton-mon} once more.
\qed
\end{proof}

As a side note, observe that above proof of Lemma~\ref{lem:undecomposed-at-least-as-good} implicitly benefits from
 the fact that SCCs of the same depth are independent.
So, SCCs with the same depth are handled in parallel by the regular Newton's method.
Therefore, $w(\vf)$, the width of $\vf$, is irrelevant here (cf.\ Proposition~\ref{prop:DNM-how-many-iterations}).

Now we can prove Theorem~\ref{thm:conv-speed-general}.

\begin{proof}[of Theorem~\ref{thm:conv-speed-general}]
 Let $k_2$ be the $k_\vf$ of Theorem~\ref{thm:dnm-error},
  and let  \mbox{$k_1 = k_2 \cdot (h(\vf)+1) \cdot 2^{h(\vf)}$}.
 Then we have
  \begin{align*}
   \ns{k_1 + i \cdot (h(\vf) + 1) \cdot 2^{h(f)}}
    &  =  \ns{(k_2 + i) \cdot (h(\vf)+1) \cdot 2^{h(\vf)}} \\
    & \ge \tns{(k_2 + i) \cdot 2^{h(\vf)}} && \text{(Lemma~\ref{lem:undecomposed-at-least-as-good})} \\
    & \ge \rhos{k_2 + i}\;,
  \end{align*}
 where the last step follows from the fact that
  \textup{DNM}$(\vf, k_2 + i)$ runs at most $(k_2 + i) \cdot 2^{h(\vf)}$ iterations in every SCC.
 By Theorem~\ref{thm:dnm-error}, $\rhos{k_2 + i}$ and hence $\ns{k_1 + i \cdot (h(\vf) + 1) \cdot 2^{h(f)}}$
  have $i$ valid bits of $\fix{\vf}$.
 Therefore, Theorem~\ref{thm:conv-speed-general} holds with $k_\vf = k_1$.
\qed
\end{proof}

%% file: sec6-upper-bounds.tex
\section{Upper Bounds on the Convergence} \label{sec:upper-bounds}

In this section we show that the lower bounds on the convergence order of Newton's method
 that we obtained in the previous section are essentially tight,
 meaning that an exponential (in~$n$) number of iterations may be needed per bit.

More precisely, we expose a family $\left(\vf^{(n)}\right)_{n \ge 1}$ of SPPs with $n$ variables,
 such that more than $k \cdot 2^{n-1}$ iterations are needed for $k$ valid bits.
Consider the following system.
\begin{equation}\label{eq:upper-bound}
 \vX = \vf^{(n)}(\vX) =
   \begin{pmatrix}
     \frac{1}{2} + \frac{1}{2}X_1^2\\
     \frac{1}{4}X_1^2 + \frac{1}{2}X_1X_2 + \frac{1}{4}X_2^2\\
     \vdots\\
     \frac{1}{4}X_{n-1}^2 + \frac{1}{2}X_{n-1}X_n + \frac{1}{4}X_n^2
   \end{pmatrix}
 \end{equation}
The only solution of~\eqref{eq:upper-bound} is $\mu\vf^{(n)} = (1,\hdots,1)^\top$.
Notice that each component of~$\vf^{(n)}$ is an SCC.
We prove the following theorem.
\begin{theorem} \label{thm:upper-bound}
 The convergence order of Newton's method applied to the SPP~$\vf^{(n)}$ from~\eqref{eq:upper-bound} (with $n \ge 2$) satisfies
 \[
  \beta(k \cdot 2^{n-1}) < k \text{ for all $k \in \{1,2,\ldots\}$.}
 \]
 In particular, $\beta(2^{n-1}) = 0$.
\end{theorem}
\begin{proof}
 We write $\vf := \vf^{(n)}$ for simplicity.
 Let
  \[\Ds{i} := \mu\vf - \ns{i} = (1, \ldots, 1)^\top - \ns{i}\;.
  \]
 Notice that $(\nsc{i}_1)_{i\in\Nat} = (0,\frac{1}{2},\frac{3}{4},\frac{7}{8},\ldots)$
  which is the same sequence as obtained by applying Newton's method
  to the 1-dimensional system $X_1 =\frac{1}{2} + \frac{1}{2}X_1^2$.
 So we have $\Delta^{(i)}_1 = 2^{-i}$, i.e., after $i$ iterations we have exactly $i$ valid bits in the first component.

 We know from Theorem~\ref{thm:well-defined} that for all~$j$ with $1 \le j \le n-1$ we have
  $\nsc{i}_{j+1} \le f_{j+1}(\ns{i}) = \frac{1}{4}(\nsc{i}_j)^2 + \frac{1}{2}\nsc{i}_{j}\nsc{i}_{j+1} + \frac{1}{4}(\nsc{i}_{j+1})^2$
  and $\nsc{i}_{j+1} \le 1$.
 It follows that $\nsc{i}_{j+1}$ is at most the least solution of
  $X_{j+1} = \frac{1}{4}(\nsc{i}_{j})^2 + \frac{1}{2}\nsc{i}_{j}X_{j+1} + \frac{1}{4}(X_{j+1})^2$,
  and so $\Delta^{(i)}_{j+1} \ge 2 \sqrt{\Delta^{(i)}_j} - \Delta^{(i)}_j > \sqrt{\Delta^{(i)}_j}$.

 By induction it follows that $\Delta^{(i)}_{j+1} > (\Delta^{(i)}_1)^{2^{-j}}$.
 In particular,
  \[
   \Delta^{(k \cdot 2^{n-1})}_n > \left(\Delta^{(k \cdot 2^{n-1})}_1\right)^{2^{-(n-1)}}
     = 2^{-k\cdot 2^{n-1} \cdot 2^{-(n-1)}} = 2^{-k}.
  \]
 Hence, after $k \cdot 2^{n-1}$ iterations we have fewer than $k$ valid bits.
\qed
\end{proof}

Notice that the proof exploits that an error in the first component gets ``amplified'' along the DAG of SCCs.
One can also show along those lines that computing $\mu\vf$ is an {\em ill-conditioned} problem:
Consider the SPP $\vg^{(n,\varepsilon)}$ obtained from~$\vf^{(n)}$ by replacing the first component by $1 - \varepsilon$
 where $0 \le \varepsilon < 1$.
If $\varepsilon = 0$ then $(\mu\vg^{(n,\varepsilon)})_n = 1$,
 whereas if $\varepsilon = \frac{1}{2^{2^{n-1}}}$ then $(\mu\vg^{(n,\varepsilon)})_n < \frac{1}{2}$.
In other words, to get $1$ bit of precision of~$\mu\vg$ one needs exponentially in~$n$ many bits in~$\vg$.
Note that this observation is independent from any particular method to compute or approximate the least fixed point.

%% file: sec-geometry.tex
\newcommand{\qiat}[1]{q_i(#1)}
\newcommand{\fiat}[1]{f_i(#1)}
\newcommand{\hiat}[1]{h_i(#1)}
\newcommand{\nqiat}[1]{q'_i(#1)}
\newcommand{\nfiat}[1]{f'_i(#1)}
\newcommand{\nhiat}[1]{h'_i(#1)}
\newcommand{\nqoat}[1]{q'_1(#1)}
\newcommand{\nqnat}[1]{q'_n(#1)}
\newcommand{\nfoat}[1]{f'_1(#1)}
\newcommand{\nfnat}[1]{f'_n(#1)}
\newcommand{\Jfat}[1]{\vf'(#1)}
\newcommand{\Jqat}[1]{\vq'(#1)}
\newcommand{\Jpiat}[1]{\vp'_i(#1)}
\newcommand{\nqix}{\nqiat{\vx}}
\newcommand{\nqox}{\nqoat{\vx}}
\newcommand{\nqnx}{\nqnat{\vx}}
\newcommand{\nfix}{\nfiat{\vx}}
\newcommand{\Jqx}{\Jqat{\vx}}
\newcommand{\Jfx}{\Jfat{\vx}}
\newcommand{\Jpix}{\Jpiat{\vx}}

\section{Geometrical Aspects of SPPs}\label{sec:geo}
As shown in \S~\ref{sub:reduction-quadratic-case} we can
assume that $\vf$ consists of quadratic polynomials. For quadratic
polynomials the locus of zeros is also called a {\em quadric
surface}, or more commonly {\em quadric}. Quadrics are one of the
most fundamental class of hypersurfaces. It is therefore natural
to study the quadrics induced by a quadratic SPP $\vf$, and how
the Newton sequence is connected to these surfaces.

Let us write $\vq$ for $\vf-\vX$.
Every component $q_i$ of $\vq$ is also a quadratic polynomial
each defining a quadric denoted by
\[
  Q_i := \{ \vx \in \R^n \mid q_i(\vx) = f_i(\vx) - x_i = 0 \}.
\]
Finding $\mu\vf$ thus corresponds
to finding the least non-negative point of intersection of these $n$ quadrics $Q_i$.
\begin{example}\label{geo:ex1}
Consider the SPP $\vf$ given by
\[
  \vf(X,Y) = \begin{pmatrix} \frac{1}{2}X^2 + \frac{1}{4}Y^2 + \frac{1}{4} \\ \frac{1}{4}X + \frac{1}{4}XY + \frac{1}{4}Y^2 + \frac{1}{4} \end{pmatrix}
\]
leading to
\[
    q_1(X,Y) = \frac{1}{2}X^2 + \frac{1}{4}Y^2 + \frac{1}{4} - X \text{ and } q_2(X,Y) =  \frac{1}{4}X + \frac{1}{4}XY + \frac{1}{4}Y^2 + \frac{1}{4} - Y.
\]
Using standard techniques from linear algebra one can show that $q_1$ defines an ellipse while $q_2$
describes a parabola (see Figure~\ref{fig:geo-ex1}).
\qed
\end{example}
\begin{figure}[ht]
\begin{center}
\begin{tabular}{cc}
\scalebox{0.3}{
  {
  \psfrag{X}{$X$}
  \psfrag{Y}{$Y$}
  \psfrag{muf}{$\mu\vf$}
  \includegraphics{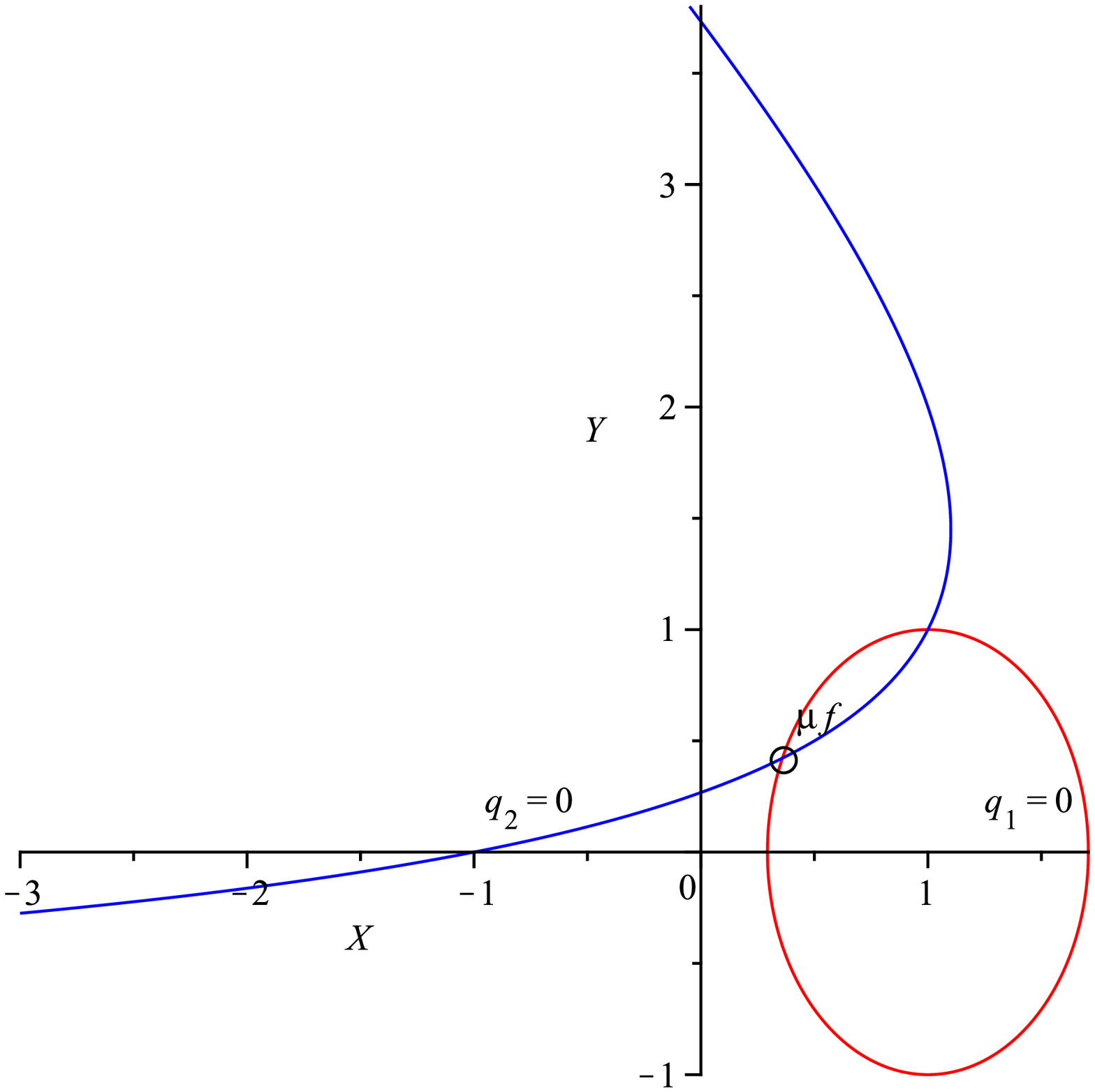}
  }
}
&
\scalebox{0.3}{
  {
  \psfrag{X}{$Y$}
  \psfrag{Y}{$Y$}
  \psfrag{muf}{$\mu\vf$}
  \includegraphics{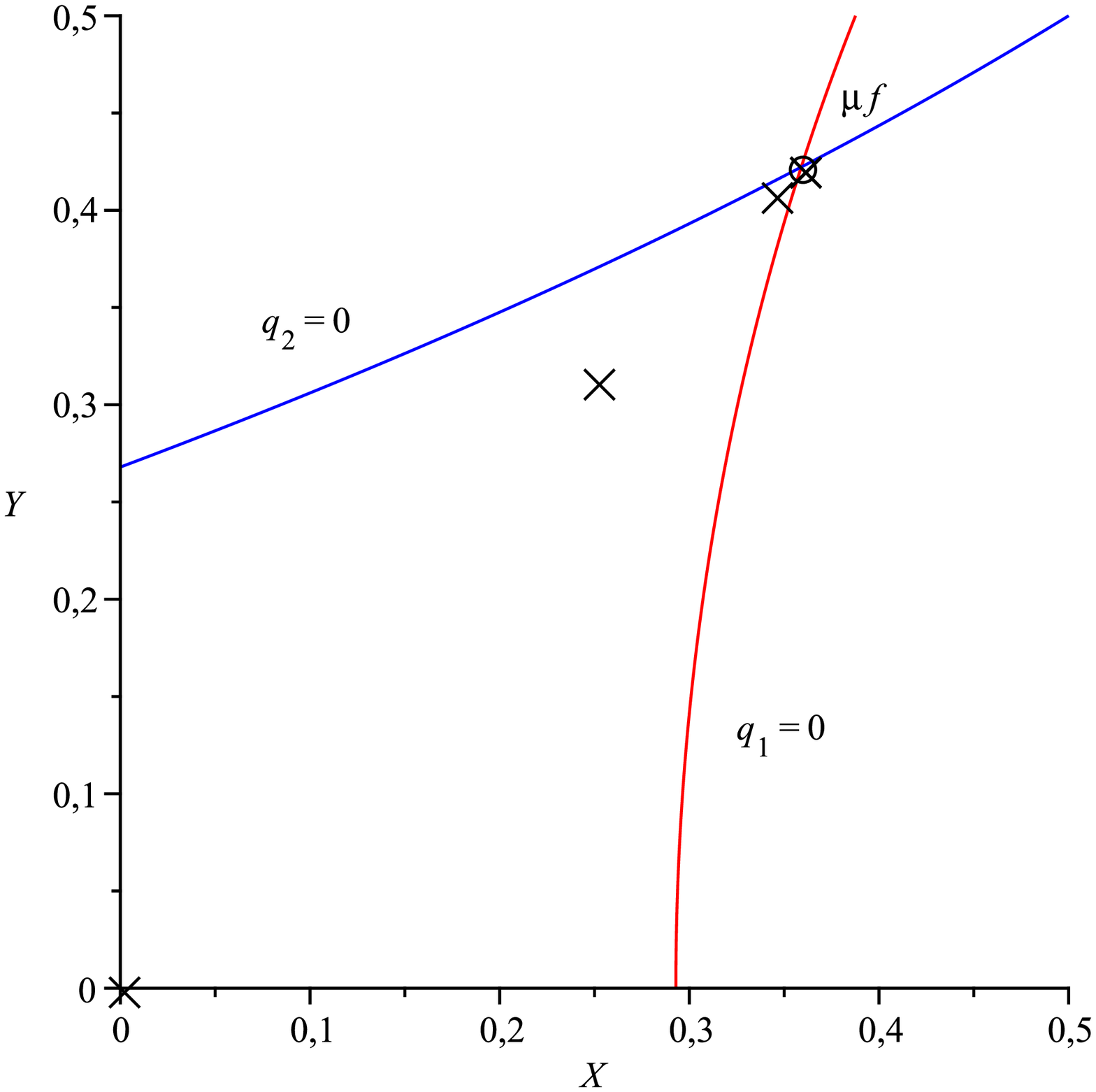}
  }
}\\
(a) & (b)\\
\end{tabular}
\end{center}
\caption{(a) The quadrics induced by the SPP from
Example~\ref{geo:ex1} with ``$q_1=0$'' an ellipse, and ``$q_2=0$'' a
parabola. (b) Close-up view of the region important for
determining $\mu\vf$. The crosses show the Newton approximants of
$\mu\vf$.} \label{fig:geo-ex1}
\end{figure}

Figure~\ref{fig:geo-ex1} shows the two quadrics induced by the SPP
$\vf$ discussed in the example above. In
Figure~\ref{fig:geo-ex1}~(a) one can recognize one of
the two quadrics as an ellipse while the other one is a parabola. In
this example the Newton approximants (depicted as crosses) stay
within the region enclosed by the coordinate axes and the two
quadrics as shown in Figure~\ref{fig:geo-ex1}~(b).

In this section we want to show that the above picture in
principle is the same for all clean and feasible scSPPs.
That is, we show that the Newton (and
Kleene) approximants always stay in the region enclosed by the
coordinate axes and the quadrics. We characterize this region and
study some of the properties of the quadrics restricted to this
region. This eventually leads to a generaliztion of Newton's method (Theorem~\ref{thm:tang-approx}).
We close the section by showing that this new method converges at least as fast as Newton's method.
%
\ifthenelse{\equal{\geoproofs}{false}}{All missing proofs can
be found in the appendix.}{}

Let us start with the properties of the quadrics $Q_i$.
We restrict our attention to the region $[\vzero,\mu\vf)$. For this we
set
\[
  M_i := Q_i \cap [\vzero,\mu\vf) = \{ \vx \in [\vzero,\mu\vf) \mid q_i(\vx) = 0 \}.
\]
We start by showing that for every $\vx \in M_i$ the gradient $\nqix$ in $\vx$ at $M_i$
does not vanish. As $\nqix$ is perpendicular to the tangent plane
in $\vx$ at $M_i$, this means that the normal of the tangent plane is determined by $\nqix$
(up to orientation). See Figure~\ref{fig:normals} for an example.
This will later allow us to apply the implicit function theorem.
\begin{figure}[ht]
\begin{center}
\scalebox{0.3}{\includegraphics{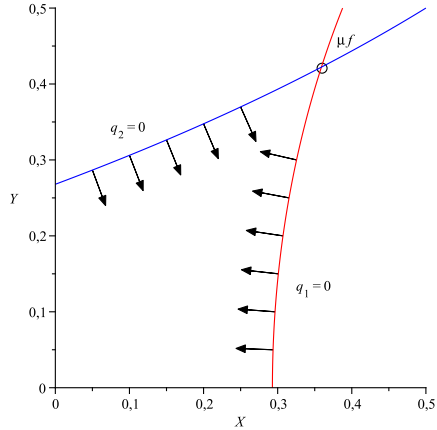}}
\end{center}
\caption{The normals (scaled down) of the quadrics from Example~\ref{geo:ex1}.}
\label{fig:normals}
\end{figure}

\begin{lemma}\label{lem:normal}
For every quadric $q_i$ induced by a clean and feasible scSPP~$\vf$ we have
\[
    \nqix = \left(\pdat{q_i}{X_1}{\vx},\pdat{q_i}{X_2}{\vx},\ldots,\pdat{q_i}{X_n}{\vx}\right) \neq \vzero \text{ and } \pdat{q_i}{X_i}{\vx} < 0 \quad \forall \vx \in [\vzero,\mu\vf).
\]
\end{lemma}
\ifthenelse{\equal{\geoproofs}{true}}{%
\begin{proof}
As shown by Etessami and Yannakakis in \cite{EYstacs05Extended} under the above preconditions
it holds for all $\vx\in [\vzero,\mu\vf)$ that
$\bigl(\text{Id} - \Jfx\bigl)$ is invertible with
\[
  \bigl(\text{Id} - \Jfx\bigl)^{-1} = \Jfx^\ast.
\]
Thus, we have
\[
    \Jqx^{-1} = \bigl(\Jfx - \text{Id}\bigr)^{-1} = - \bigl( \Jfx^\ast \bigr),
\]
implying that $\nqix \neq \vzero$ for all $\vx\in [\vzero,\mu\vf)$ as $\Jqx$ has to have full rank $n$
in order for $\Jqx^{-1}$ to exist.
Furthermore, it follows that all entries of $\Jqx^{-1}$ are non-positive as $\Jfx^\ast$ is non-negative.
Now, as $q_i(\vX) = f_i(\vX)- X_i$ and $f_i(\vX)$ is a polynomial with non-negative coefficients, it holds that
\[
 \nqix \cdot \ve_j = \pdat{q_i}{X_j}{\vx} = \pdat{f_i}{X_j}{\vx} \ge 0
\]
for all $j \neq i$ and $\vx\ge \vzero$.
With every entry of $\Jqx^{-1}$ non-positive, and
\[
\nqix \cdot \Jqx^{-1} = \ve_i^\top,
\]
we conclude $\pdat{q_i}{X_i}{\vx} < 0$.
\qed
\end{proof}
}{%
}

In the following, for $i\in\{1,\ldots,n\}$ we write $\vx_{-i}$ for the vector
$(x_1,\ldots,x_{i-1},x_{i+1},\ldots,x_n)$ and define $(\vx_{-i},x_i)$ to also
denote the original vector $\vx$.

We next show that there exists a complete parametrization of ``the lower part'' of $M_i$.
With ``lower part'' we refer to the set
\[
S_i := \{ \vx \in M_i \mid \forall \vy \in M_i : ( \vx_{-i} = \vy_{-i} ) \Rightarrow x_i \le y_i \}\,,
\]
i.e., the points $\vx \in M_i$ such that there is no point $\vy$ with the same non-$i$-components but smaller $i$-component.
Taking a look at Figure~\ref{fig:geo-ex1}, the surfaces $S_1$ and $S_2$ are those
parts of $M_1$, resp.~$M_2$, which delimit that part of $\Rp^2$ shown in Figure~\ref{fig:geo-ex1}~(b).

If $\vx \in S_i$ then $x_i$ is the least non-negative root of the (at most) quadratic
polynomial $q_i(X_i, \vx_{-i})$.
As we will see, these roots can also be represented by the following functions:
\begin{definition}\label{def:h}
For a clean and feasible scSPP $\vf$ we
define for all $k\in\N$ the polynomial $h_{i}^{(k)}$ by
\[
h_i^{(0)}(\vX_{-i}) := f_i[i/0](\vX_{-i}), \quad h_i^{(k+1)}(\vX_{-i}) := f_i[i/h_i^{(k)}(\vX_{-i})](\vX_{-i})
\]
The function $h_i(\vX)$ is then defined pointwise by
\[
    h_i(\vx_{-i}) := \lim_{k\to\infty} h_i^{(k)}(\vx_{-i})
\]
for all $\vx_{-i}\in[\vzero,\mu\vf_{-i}]$.
\end{definition}

We show in the appendix (see Proposition~\ref{prop:h}) that the function $h_i$
is well-defined and exists.
We therefore can parameterize the surface $S_i$ w.r.t.\ the remaining variables $\vX_{-i}$,
i.e., $h_i$ is the ``height'' of the surface $S_i$ above the ``ground'' $X_i = 0$.
\ifthenelse{\equal{\geoproofs}{true}}{%
\begin{prf}
Let $\vzero \le \vx \le \vy \le \mu\vf$.
Using the monotonicity of $f_i$ over $\Rp^n$ we proceed by induction on $k$.
\begin{itemize}
\item[(a)]
  For $k=0$ we have
  \[
    0 \le h_i^{(0)}(\vx_{-i}) = f_i(0,\vx_{-i}) \le f_i( \mu\vf) = \mu\vf_i.
  \]
  We then get
  \[
    0 \le h_i^{(k+1)}(\vx_{-i}) = f_i( h_i^{(k)}(\vx_{-i}), \vx_{-i} ) \le f_i( \mu\vf ) = \mu\vf_i.
  \]
\item[(b)]
  For $k=0$ we have
  \[
    h_i^{(0)}(\vx_{-i}) = f_i( 0, \vx_{-i} ) \le f_i( h_i^{(0)}(\vx_{-i}),\vx_{-i} ) = h_i^{(1)}(\vx_{-i}).
  \]
  Thus
  \[
   h_i^{(k+1)}(\vx_{-i}) = f_i( h_i^{(k)}(\vx_{-i}), \vx_{-i} ) \le f_i( h_i^{(k+1)}(\vx_{-i}),\vx_{-i} ) = h_i^{(k+2)}(\vx_{-i})
  \]
  follows.
\item[(c)]
  As $\vx \le \vy$, we have for $k=0$
  \[
    h_i^{(0)}(\vx_{-i}) = f_i(0,\vx_{-i}) \le f_i(0,\vy_{-i}) = h_i^{(0)}(\vy_{-i}).
  \]
  Hence, we get
  \[
    h_i^{(k+1)}(\vx_{-i}) = f_i( h_i^{(k)}(\vx_{-i}),\vx_{-i}) \le f_i( h_i^{(k)}(\vy_{-i}),\vy_{-i}) = h_i^{(k+1)}(\vy_{-i}).
  \]
\item[(d)]
  As the sequence $(h_i^{(k)}(\vx_{-i}))_{k\in\N}$ is monotonically increasing and bounded from above
  by $\mu\vf_i$, the sequence converges.
  Thus, for every $\vx$ the value
  \[
    h_i(\vx_{-i}) = \lim_{k\to\infty} h_i^{(k)}(\vx_{-i})
  \]
  is well-defined, i.e., $h_i$ is a map from $[\vzero,\mu\vf_{-i}]$ to $[0,\mu f_i]$.

  If $f_i$ depends on at least one other variable except $X_i$,
  then $h_i$ is a non-constant power series in this variable with non-negative coefficients.
  For $\vx_{-i} \in [\vzero,\mu\vf_{-i})$ we thus always have
  \[
    h_i( \vx_{-i} ) < h_i( \mu\vf_{-i} )  = \mu\vf_i
  \]
  as $\vx_{-i} \prec \mu\vf_{-i}$.
\item[(e)]
  This follows immediately from (b).
\item[(f)]
  As $f_i$ is continuous, we have
 \[ f_i( h_i(\vx_{-i}), \vx_{-i} ) = f_i( \lim_{k\to\infty} h_i^{(k)}(\vx_{-i}), \vx_{-i}) = \lim_{k\to\infty} h_i^{(k+1)}(\vx_{-i}) = h_i(\vx_{-i}),
 \]
 where the last equality holds because of (b).
\item[(g)]
  Using induction similar to (a) replacing $\mu\vf$ by $\vx$,
  one gets $h_i^{(k)}(\vx_{-i}) \le x_i$ for all $k\in\N$ as $f_i(\vx_{-i}) = x_i$.
  Thus, $h_i(\vx_{-i}) \le x_i$ follows similarly to (d).
\item[(h)]
  By definition, we have $\mu \vf = \lim_{k\to\infty} \vf^k(\vzero)$.
  For $k=0$, we have
  \[ (\vf^{0}(\vzero))_i = 0 \le f_i(0,\mu\vf_{-i}) = h_i^{(0)}(\mu\vf_{-i}). \]
  We thus get by induction
  \[
    (\vf^{(k+1)}(\vzero))_i = f_i( \vf^{k}(\vzero) ) \le f_i( h_i^{(k)}(\mu\vf_{-i}), \mu\vf_{-i}) = h_i^{(k+1)}(\mu\vf_{-i}).
  \]
  Thus, we may conclude $\mu\vf_i \le h_i(\mu\vf_{-i})$.
  As $\mu\vf_i = f_i(\mu\vf)$, we get by virtue of (g) that $h_i(\mu\vf_{-i}) \le \mu\vf_i$, too.
\end{itemize}
\end{prf}
}{%
}

By the preceding proposition the map
\[
 \vp_i : [\vzero,\mu\vf_{-i}) \to [\vzero,\mu\vf] : \vx_{-i} \mapsto ( x_1, \ldots, x_{i-1}, h_i(\vx_{-i}), x_{i+1}, \ldots, x_n )
\]
gives us a pointwise parametrization of $S_i$.
We want to show that $\vp_i$ is continuously differentiable. For this it suffices to show
that $h_i$ is continuously differentiable which follows easily from the implicit function theorem (see e.g.~\cite{OrtegaRheinboldt:book}).

%

\begin{lemma}\label{lem:h-diff}
$h_i$ is continuously differentiable with
\[
    \pdat{h_i}{X_j}{\vx_{-i}} = \frac{\pdat{f_i}{X_j}{\vx}}{-\pdat{q_i}{X_i}{\vx}}
    = \frac{\pdat{q_i}{X_j}{\vx}}{-\pdat{q_i}{X_i}{\vx}} \text{ for } \vx \in S_i \text{ and } j \neq i.
\]
In particular, $\pd{h_i}{X_j}$ is monotonically increasing with $\vx$.
\end{lemma}
\ifthenelse{\equal{\geoproofs}{true}}{%
\begin{proof}
By Lemma~\ref{lem:normal} the implicit function theorem is applicable for every $\vx\in S_i$.
We therefore find for every $\vx\in S_i$ a local parametrization
$h_{\vx} : U \mapsto V$ with $h_{\vx}(\vx_{-i}) = x_i$.
Thus $h_{\vx}(\vx_{-i})$ is the least non-negative solution of $q_i(X_i,\vx_{-i}) = 0$.
By continuity of $q_i$ it is now easily shown that for all $\vy_{-i} \in U$ it has to hold
that $h_{\vx}(\vy_{-i})$ is also the least non-negative solution of $q_i(X_i,\vy_{-i})=0$ (see below).
By uniqueness we therefore have $h_{\vx} = h_i$ and that $h_i$ is continuously differentiable
for all $\vx_{-i} \in [\vzero,\mu\vf_{-i})$.

For every $\vx_{-i}\in [\vzero,\mu\vf_{-i})$ we can solve the (at most) quadratic equation
$q_i(X_i,\vx_{-i}) = 0$. We already know that $h_i(\vx_{-i})$ is the least non-negative solution of
this equation. So, if there exists another solution, it has to be real, too.

Assume first that this equation has two distinct solutions for some fixed $\vx_{-i}\in [\vzero,\mu\vf_{-i})$.
Solving $q_i(X_i,\vx_{-i})=0$ thus leads to an expression of the form
\[
  \frac{-b(\vx_{-i}) \pm \sqrt{b(\vx_{-i})^2 - 4a\cdot c(\vx_{-i})}}{2a}
\]
for the solutions where $b,c$ are (at most) quadratic polynomials in $\vX_{-i}$, $c$ having non-negative coefficients,
and $a$ is a positive constant (leading coefficient of $X_i^2$ in $q_i(\vX)$).
As $b$ and $c$ are continuous, the discriminant $b(\cdot)^2 - 4a\cdot c(\cdot)$ stays positive
for some open ball around $\vx_{-i}$ included inside of $U$ (it is positive in $\vx_{-i}$ as we assume
that we have two distinct solutions).
By making $U$ smaller, we may assume that $U$ is this open ball.
One of the two solutions has than to be the least non-negative solution.
As $h_{\vx}$ is the least non-negative solution for $\vx_{-i}$, and $h_{\vx}$ is continuous,
this also has to hold for some open ball centered at $\vx_{-i}$, w.l.o.g.\ $U$ is this ball.
So, $h_{\vx}$ and $h_i$ coincide on $U$.

We turn to the case that $q_i(X_i,\vx_{-i})=0$ has only a single solution, i.e.\ $h_i(\vx_{-i})$.
Note that $q_i(\vX)$ is linear in $X_i$ if, and only if, $q_i(X_i,\vx_{-i})$ is linear in $X_i$.
Obviously, if $q_i$ linear in $X_i$, then $h_i$ and $h_{\vx}$ coincide on $U$.
Thus, consider the case that $q_i(\vX)$ is quadratic in $X_i$, but $q_i(X_i,\vx_{-i})$ has only
a single solution.
This means that $\vx_{-i}$ is a root of the discriminant, i.e.\ $b(\vx_{-i}) - 4ac(\vx_{-i})=0$.
As $h_i(\vy_{-i})$ is a solution of $q_i(X_i,\vy_{-i})=0$ for all $\vy_{-i}\in U$,
the discriminant is non-negative on $U$.
If it equal to zero on $U$, then we again have that $h_i$ is equal to $h_{\vx}$ on $U$.
Therefore assume that is positive in some point of $U$.
As the discriminant is continuous, the solutions change continuously with $\vx_{-i}$.
But this implies that for some $\vy_{-i}\in U$ there are at least two $y_i,y_i^\ast\in V$
such that $(\vy_{-i},y_i)$ and $(\vy_{-i},y_i^\ast)$ are both located on the quadric $q_i(\vX)=0$.
But this contradicts the uniqueness of $h_{\vx}$ guaranteed by the implicit function theorem.

\michael{(Sollte auch mit dem Satz ueber implizite Funktionen gehen.)}
Assume now that $\vx \in S_i$. We then have
\[
    q_i( \vx ) = q_i(\vx_{-i}, h_i(\vx_{-i})) = 0,
\]
or equivalently
\[
  f_i(\vx_{-i},h_i(\vx_{-i})) = h_i(\vx_{-i}).
\]
Calculating the gradient of both in $\vx$ yields
\[
    \nfix \cdot \Jpiat{\vx_{-i}} = \nhiat{\vx_{-i}}.
\]
For the Jacobian of $\vp_i$ we obtain
\[
 \Jpiat{\vx_{-i}} = \begin{pmatrix} \ve_1^\top \\ \vdots \\ \ve_{i-1}^{\top} \\ \nhiat{\vx_{-i}} \\ \ve_{i+1}^\top \\ \vdots \\ \ve_{n}^{\top} \end{pmatrix}.
\]
This leads to
\[
  \pdat{f_i}{X_j}{\vx} + \pdat{f_i}{X_i}{\vx} \cdot \pdat{h_i}{X_j}{\vx_{-i}} = \pdat{h_i}{X_j}{\vx_{-i}}
\]
which solved for $\pd{h_i}{X_j}$ yields
\[
    \pdat{h_i}{X_j}{\vx_{-i}} = \frac{\pdat{f_i}{X_j}{\vx}}{-\pdat{q_i}{X_i}{\vx}}.
\]
As $\pdat{q_i}{X_i}{\vx} < 0$ and both $\pd{f_i}{X_j}$ and $\pd{q_i}{X_i}$ monotonically increase with $\vx$,
it follows that $\pd{h_i}{X_j}$ also monotonically increases with $\vx$. Finally, for $j\neq i$ we have
that $\pd{q_i}{X_j} = \pd{f_i}{X_j}$ as $q_i = f_i - X_i$.
\qed
\end{proof}
}{%
}

\begin{corollary}
The map
\[
 \vp_i: [\vzero,\mu\vf_{-i}) \to [\vzero,\mu\vf] : \vx_{-i} \mapsto ( x_1, \ldots, x_{i-1}, h_i(\vx_{-i}), x_{i+1}, \ldots, x_n )
\]
is continuously differentiable and a local parametrization of the manifold $S_i$.
%
%
\end{corollary}

\begin{example}
For the SPP $\vf$ defined in Example~\ref{geo:ex1} we can simply solve
$q_1(X,Y)$ for $X$ leading to
\[
    h_1(Y) = 1 - \sqrt{\frac{1}{2}(1 - Y^2)}.
\]
The important point is that by the previous result we know that this function has to be defined on $[0,\mu\vf_2]$,
and differentiable on $[0,\mu\vf_2)$.
Similarly, we get
\[
 h_2(X) = 2-\frac{1}{2}X - \frac{1}{2} \sqrt{X^2 - 12X + 12}.
\]
%
\end{example}

Figure~\ref{fig:geo-ex1}~(b) conveys the impression that the surfaces $S_i$ are convex
w.r.t.\ the parameterizations $\vp_i$.
As we have seen, the functions $h_i$ are monotonically increasing.
Thus, in the case of two dimensions the functions $h_i$ even have to be
strictly monotonically increasing (as $\vf$ is strongly-connected),
so that the surfaces $S_i$ are indeed convex.
(Recall that a surface $S$ is convex in a point $\vx\in S$ if
$S$ is located completely on one side of the tangent plane at $S$ in $\vx$.)
But in the case of more than two variables this no longer needs to hold.

\begin{example}
The equation
\[
  Z = \frac{1}{8}X^2 + \frac{3}{4}XY + \frac{1}{8}Y^2 + \frac{1}{4}
\]
is an admissible part of any SPP. It defines the hyperbolic
paraboloid depicted in Figure~\ref{fig:hyppar} which is clearly not
convex.
\begin{figure}[ht]
\begin{center}
\begin{tabular}{ccc}
  \scalebox{0.2}{\includegraphics{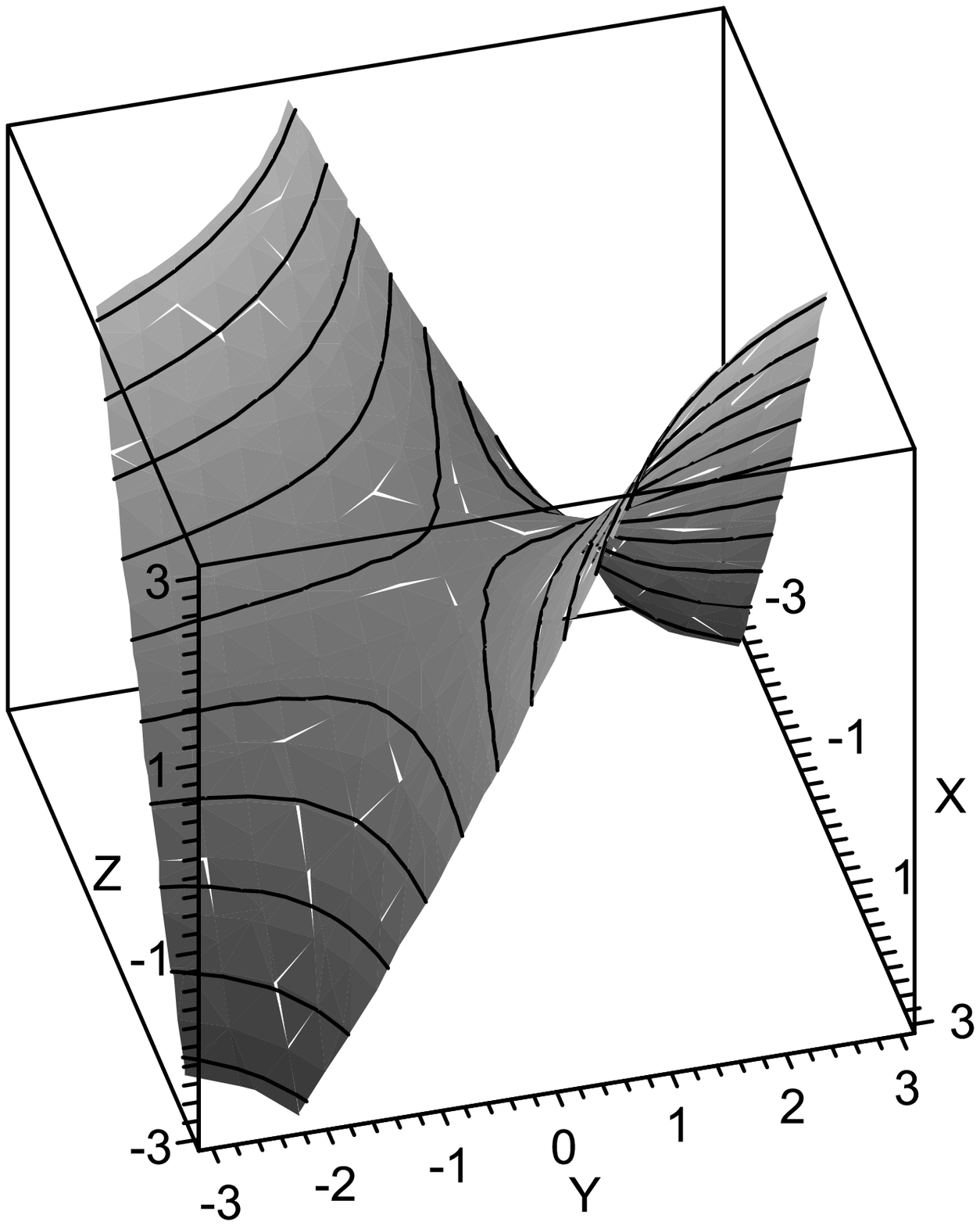}} & \scalebox{0.2}{\includegraphics{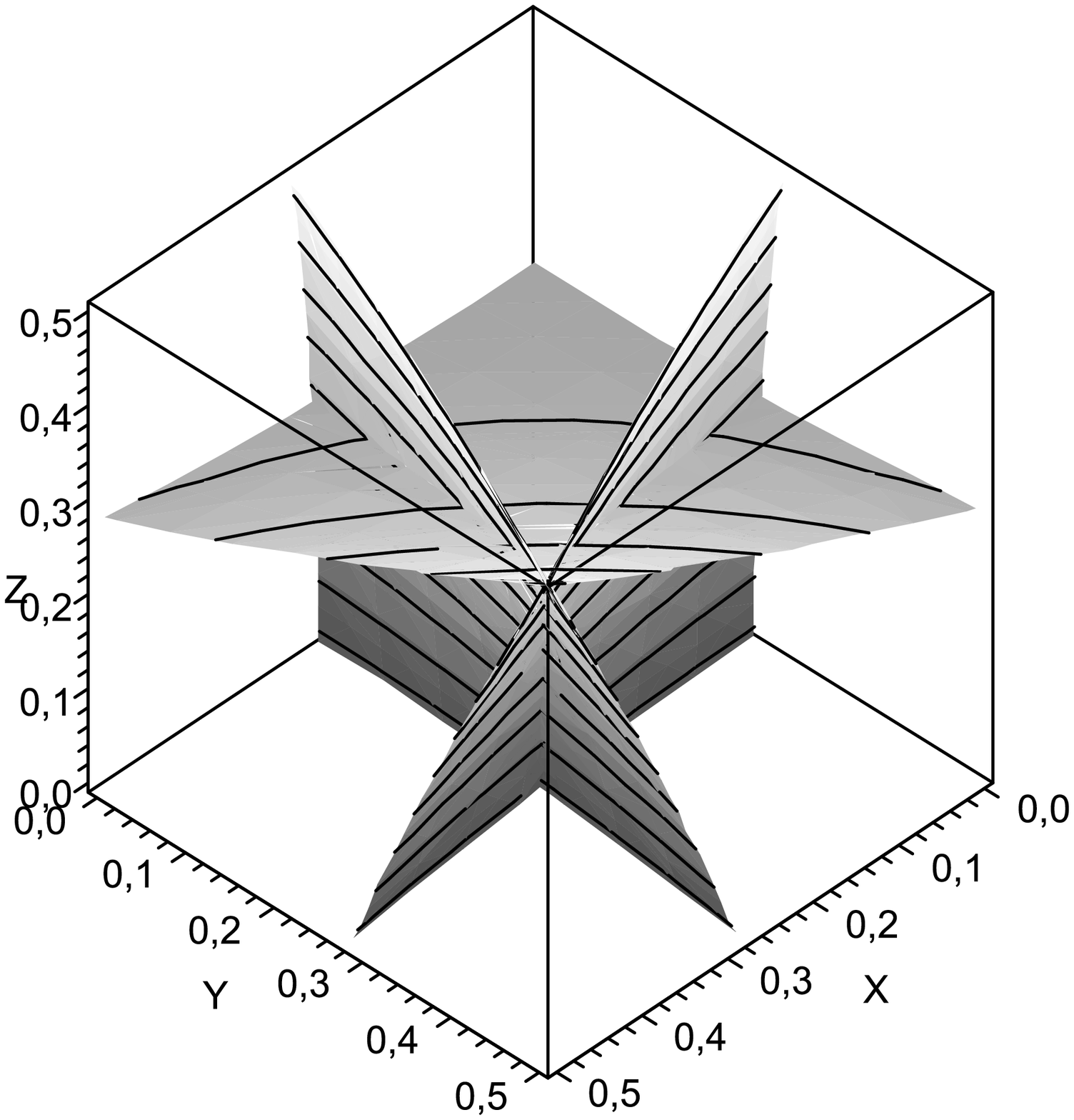}} & \scalebox{0.2}{\includegraphics{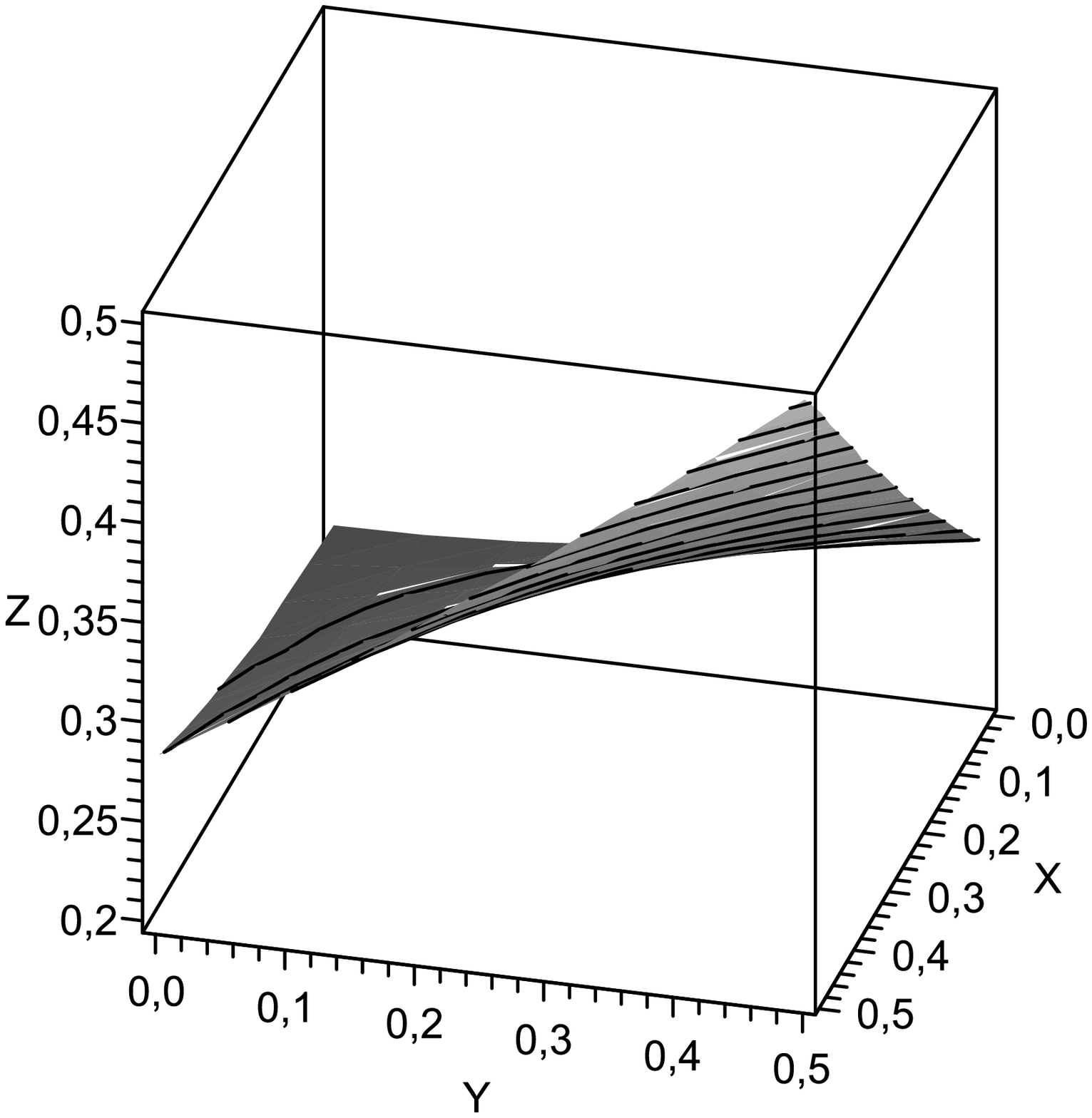}}\\
  (a) & (b) & (c)\\
\end{tabular}
\end{center}
\caption{(a) The hyperbolic paraboloid defined by $Z =
\frac{1}{8}X^2 + \frac{3}{4}XY + \frac{1}{8}Y^2+\frac{1}{4}$ for
$X,Y,Z\in[-10,10]$. (b) A visualization of an SPP consisting of
three copies of the quadric of (a) with $\mu\vf = (\frac{1}{2},\frac{1}{2},\frac{1}{2})$ the
upper apex. (c) One of the three quadrics of (b) over
$[\vzero,\mu\vf]$. Clearly, even limited to this range the surface
is not convex.} \label{fig:hyppar}
\end{figure}
\end{example}

Still, as shown in Lemma~\ref{lem:taylor} it holds for all $\vzero\le \vx \le \vy$ that
\[
    \vx + \Jfx \cdot \vy \le \vf(\vx+\vy).
\]
It now follows (see the following lemma) that the surfaces $S_i$ have the property that for every $\vx\in [\vzero,\mu\vf)$ the
``relevant'' part of $S_i$ for determining $\mu\vf$, i.e.\ $S_i\cap[\vx,\mu\vf]$, is located on
the same side of the tangent plane at $S_i$ in $\vx$ (see Figure~\ref{fig:convex}).
\begin{figure}[ht]
\begin{center}
\scalebox{0.3}{\includegraphics{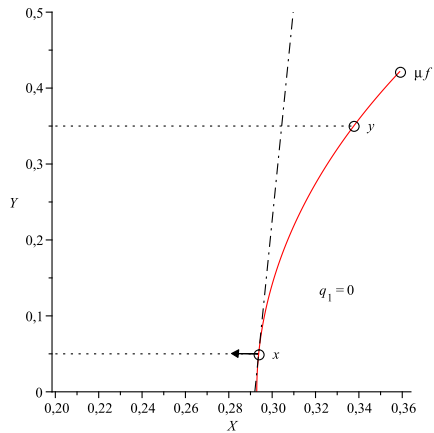}}
\end{center}
\caption{The graphic shows the quadric defined by $q_1=0$ with the tangent and normal in $\vx$ at $S_1$. Every point $\vy$ of $S_1$
above $\vx$ is located on the same side of the the tangent. More precisely, we have $\nabla q_1|_{\vx} \cdot (\vy-\vx)\le 0$.}
\label{fig:convex}
\end{figure}

\begin{lemma}\label{lem:upconv}
For all $\vx\in S_i$ we have
\[
    \forall \vy \in S_i \cap [\vx,\mu\vf] : \nqix \cdot (\vy-\vx) \le 0.
\]
In particular
\[
\forall \vy \in S_i \cap [\vx,\mu\vf] : y_i \ge x_i + \sum_{j\neq i} \pdat{h_i}{X_j}{\vx_{-i}} \cdot (y_j - x_j).
\]
\end{lemma}
\ifthenelse{\equal{\geoproofs}{true}}{%
\begin{prf}
Let $\vx\in S_i$, i.e., $f_i(\vx) = x_i$.
We want to show that
\[
    \nqix \cdot( \vy - \vx ) \le 0
\]
for all $\vy\in S_i \cap [\vx,\mu\vf)$.
As $f_i$ is quadratic in $\vX$, we may write
\[
\begin{array}{lcl}
    0 & = & q_i(\vy) \\
      & = & - y_i + f_i(\vy)\\
      & = & - y_i + \underbrace{f_i(\vx)}_{= x_i} + \nfix \cdot (\vy - \vx) + \underbrace{(\vy-\vx)^\top \cdot A \cdot (\vy-\vx)}_{\ge 0} \\
      & \ge & -y_i + x_i + \nfix \cdot (\vy-\vx)\\
      & =   & \nfix\cdot (\vy-\vx) - \ve_i^\top \cdot(\vy-\vx)\\
      & = & \nqix\cdot(\vy-\vx)
\end{array}
\]
where $A$ is a symmetric square-matrix with non-negative components such that the quadric terms of $f_i$ are given by
$\vX^\top A \vX$.

The second claim is easily obtained by solving this inequality for $y_i$ and recalling that by Lemma~\ref{lem:h-diff}
we have $\pdat{h_i}{X_j}{\vx_{-i}} = \frac{\pdat{q_i}{X_j}{\vp_i(\vx_{-i})}}{-\pdat{q_i}{X_i}{\vp_i(\vx_{-i})}}$
and $\pdat{q_i}{X_i}{\vp_i(\vx_{-i})} < 0$.
\end{prf}
}{%
}

Consider now the set
\[
R := \bigcap_{i=1}^n \{ \vx\in [\vzero,\mu\vf) \mid x_i \le h_i(\vx_{-i}) \},
\]
i.e., the region of $[\vzero,\mu\vf)$ delimited by the coordinate axes and the surfaces $S_i$.
Note that the gradient $\nqix$ for $\vx\in S_i$ points from $S_i$ into $R$ (see Figure~\ref{fig:normals}).
\begin{proposition}\label{prop:R}
It holds
\[
    \vx \in R \Leftrightarrow \vx\in[\vzero,\mu\vf) \wedge \vq(\vx) \ge \vzero.
\]
\end{proposition}
\ifthenelse{\equal{\geoproofs}{true}}{%
\begin{prf}
Let $\vx \in R$ and $i\in\{1,\ldots,n\}$.
Consider the function
\[
    g(t) := q_i(\vp_{i}(\vx_{-i}) + t \ve_i).
\]
As $q_i$ is a quadratic polynomial in $\vX$ there exists a symmetric square-matrix $A$ with non-negative entries, a vector $\vb$,
and a constant $c$ such that
\[
    q_i(\vX) = \vX^\top A \vX + \vb^\top \vX + c.
\]
It then follows that
\[
    q_i(\vX+\vY) = q_i(\vX) + \nqiat{\vX} \vY + \vY^\top A \vY.
\]
With $q_i(\vp_{i}(\vx_{-i}))=0$ this implies
\[
 g(t) = \nqiat{\vp_{i}(\vx_{-i})} t\ve_i + t^2 \underbrace{\ve_i^\top A \ve_i}_{:=a \ge 0} = t\cdot \left( \pdat{q_i}{X_i}{\vp_{i}(\vx_{-i})} + a\cdot t\right).
\]
As $\vp_i(\vx_{-i}) \prec \mu\vf$ ($\vf$ is strongly connected and $\vx\in [\vzero,\mu\vf)$), we know
that $\pdat{q_i}{X_i}{\vp_{i}(\vx_{-i})} < 0$. Thus, $g(t)$ has at most two zeros, one at $0$, the other
for some $t^\ast \ge 0$.

For the direction $(\Rightarrow)$ we only have to show that $x_i \le h_i(\vx_{-i})$ implies that $q_i(\vx) \ge 0$.
This now easily follows as $x_i \le h_i(\vx_{-i})$ implies that there is a $t' \le 0$ with $p_{i}(\vx_{-i}) + t\ve_i = \vx$.
But for this $t' \le 0$ we have $q_i(\vx) = g(t') \ge 0$.

Consider therefore the other direction $(\Leftarrow)$, that is $\vx\in [\vzero,\mu\vf)$ with $\vq(\vx) \ge \vzero$.
Assume that $\vx\not\in R$, i.e., for at least one $i$ we have $x_i > h_i(\vx_{-i})$.
As $q_i(\vx)\ge 0$ there has to be a $t'' > 0$ with $p_{i}(\vx_{-i}) + t''\ve_i = \vx$ and $g(t'') \ge 0$.
This implies that $a>0$ has to hold as otherwise $g(t)$ would be linear in $t$ and negative for $t>0$.
But then the second root $t^\ast$ of $g(t)$ has to be positive.
Set $\vx^\ast = p_i(\vx_{-i}) + t^\ast \ve_i$ with $q_i(\vx^\ast) = 0$, too.

A calculation similar to the one from above leads to
\[
  g(t+t^\ast) = q_i(\vx^\ast + t\ve_i) = t\cdot \left( \pdat{q_i}{X_i}{\vx^\ast} + a\cdot t\right).
\]
It follows that $\pdat{q_i}{X_i}{\vx^\ast}$ has to be greater than zero for $-t^\ast$ to be a root (as $a>0$).
But we have shown that $\pdat{q_i}{X_i}{\vx} < 0$ for all $\vx\in[\vzero,\mu\vf)$.
%
\end{prf}
}{%
}

From this last result it now easily follows that $R$ is indeed the region of $[\vzero,\mu\vf)$
where all Newton and Kleene steps are located in.
\begin{theorem}\label{cor:R-newton-kleene}
Let $\vf$ be a clean and feasible scSPP.
All Newton and Kleene steps starting from $\vzero$ lie within $R$, i.e.\
\[
 \ns{i}, \ks{i} \in R \quad (\forall i\in\N).
\]
\end{theorem}
\begin{prf}
For an scSPP we have
$\ks{i},\ns{i}\in [\vzero,\mu\vf)$ for all $i$. Further, $\ks{i}
\le \ks{i+1} = \vf(\ks{i})$ and $\ns{i} \le \vf(\ns{i})$ holds for
all $i$, too.
\end{prf}

In the rest of this section we will use the results regarding $R$ and the surfaces $S_i$
for interpreting Newton's method geometrically and for obtaining a generalization of Newton's method.

The preceding results suggest another way of determining $\mu\vf$ (see Figure~\ref{fig:tangents}):
Let $\vx$ be some point inside of $R$.
We may move from $\vx$ onto one of the surface $S_i$ by going upward along the line $\vx + t\cdot \ve_i$
which gives us the point $\vp_i(\vx_{-i}) = (\vx_{-i}, h_i(\vx_{-i}))$.
As $\vx \in R$, we have $\vx, \vp_i(\vx_{-i}) \le \mu\vf$.
Consider now the tangent plane
$$T_i|_{\vx} = \left\{ \vy\in\R^n \mid \nqiat{\vp_i(\vx_{-i})} \cdot \bigl( \vy - \vp_i(\vx_{-i})\bigr) = 0 \right\}$$
 at $S_i$ in $\vp_i(\vx_{-i})$.
 \begin{figure}[ht]
\begin{center}
  \scalebox{0.3}{\includegraphics{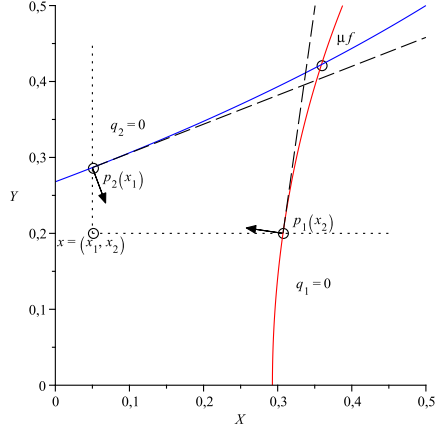}}
\end{center}
\caption{Given a point $\vx$ inside of $R$ the intersection of the tangents at the quadrics in the points
$p_1(x_2)$, resp. $p_2(x_1)$ is also located inside of $R$, yielding a better approximation of $\mu\vf$.}
\label{fig:tangents}
\end{figure}
Recall that by Lemma~\ref{lem:upconv} we have
\[
\forall \vy \in S_i \cap [\vp_{i}(\vx_{-i}),\mu\vf) : \nqiat{\vp_{i}(\vx_{-i})} \cdot (\vy-\vp_{i}(\vx_{-i})) \le 0,
\]
i.e., the part of $S_i$ relevant for determining $\mu\vf$ is located completely below (w.r.t. $\nqiat(\vp_i(\vx_{-i}))$)
 this tangent plane.
By continuity this also has to hold for $\vy=\mu\vf$.
Hence, when taking the intersection of all the tangent planes $T_1$ to $T_n$ this gives us
again a point $\Ta(\vx)$ inside of $R$.
That this point $\Ta(\vx)$ exists and is uniquely determined is shown in the following lemma.

\begin{lemma}\label{lem:tangents-invertable}
Let $\vf$ be a clean and feasible scSPP.
Let $\vx^{(1)},\ldots,\vx^{(n)}\in [\vzero,\mu\vf)$.
Then the matrix
\[
    \begin{pmatrix}
      \nqoat{\vx^{(1)}}\\
      \vdots\\
      \nqnat{\vx^{(n)}}
      \end{pmatrix}
\]
is regular, i.e., the vectors $\{ \nqiat{\vx^{(i)}} \mid i = 1,\ldots,n \}$ are linearly independent.
\end{lemma}
\ifthenelse{\equal{\geoproofs}{true}}{%
\begin{prf}
%
%
%
Define $\vx\in [\vzero,\mu\vf)$ by setting
\[
        x_i := \max\{ x^{(j)}_i \mid j = 1,\ldots,n \}.
\]
We then have $\vx^{(i)} \le \vx$ for all $i$, and $\vx\prec \mu\vf$.
As mentioned above, we therefore have that $\Jqx$ is regular with
\[
    \Jqx^{-1} = - \sum_{k\in\N} \Jfx^k.
\]
%
As $\vx^{(i)} \le \vx$ it follows that
\[
    \begin{pmatrix}
    \nfoat{\vx^{(1)}} \\
    \vdots\\
    \nfnat{\vx^{(n)}}
    \end{pmatrix}
    \le
    \Jfx.
\]
Hence, we also have
\[
  \sum_{k=0}^l \begin{pmatrix}
    \nfoat{\vx^{(1)}} \\
    \vdots\\
    \nfnat{\vx^{(n)}}
    \end{pmatrix}^l \le \sum_{k=0}^l \Jfx
\]
implying that
\[
  \begin{pmatrix}
    \nfoat{\vx^{(1)}} \\
    \vdots\\
    \nfnat{\vx^{(n)}}
    \end{pmatrix}^{\ast}
    \text{ and, thus, }
    \begin{pmatrix}
    \nqoat{\vx^{(1)}} \\
    \vdots\\
    \nqnat{\vx^{(n)}}
    \end{pmatrix}^{-1}
    \text{ exist.}
\]
So, the vectors $\{ \nqoat{\vx^{(1)}}, \ldots, \nqnat{\vx^{(n)}} \}$ have to be linearly independent.
\end{prf}
}{%
}

By this lemma the normals at the quadrics in the points $\vp_i(\vx_{-i})$ for $\vx\in [\vzero,\mu\vf)$
are linearly independent. Thus, there exists a unique point of intersection of tangent planes at the quadrics
in these points.
\ifthenelse{\equal{1}{1}}{
Of course, in general the values $h_i(\vx_{-i})$ can be irrational. The following definition takes this
in account by only requiring that underapproximations $\eta_i$ of $h_i(\vx_{-i})$ are known.
\begin{definition}\label{def:tang-approx}
Let $\vx\in R$.
For $i=1,\ldots,n$ fix some $\eta_i \in [x_i,h_i(\vx_{-i})]$, and set $\veta = (\eta_1,\ldots,\eta_n)$.
We then let $\Ta_{\veta}(\vx)$ denote the solution of
\[
 \nqiat{(\vx_{-i},\eta_i)}(\vX - (\vx_{-i},\eta_i)) = - q_i((\vx_{-i},\eta_i)) \quad (i=1,\ldots,n).
\]
We drop the subscript and simply write $\Ta$ in the case of $\eta_i=h_{i}(\vx_{-i})$ for $i=1,\ldots,n$.
\end{definition}

Note that the operator $\Ta_{\vx}$ is the Newton operator $\Ne$.

\begin{theorem}{\label{thm:tang-approx}}
Let $\vf$ be a clean and feasible scSPP.
Let $\vx\in R$.
For $i=1,\ldots,n$ fix some $\eta_i \in [x_i,h_i(\vx_{-i})]$, and set $\veta = (\eta_1,\ldots,\eta_n)$.
We then have
\[
\vx \le \Ne(\vx) \le \Ta_{\veta}(\vx) \le \Ta(\vx) \le \mu\vf
\]
Further, the operator $\Ta$ is monotone on $R$, i.e., for any $\vy\in R$ with $\vx\le \vy$ it holds that $\Ta(\vx)\le \Ta(\vy)$.
\end{theorem}

By Theorem~\ref{thm:tang-approx}, replacing the Newton operator $\Ne$ by~$\Ta$ gives a variant of Newton's method which converges at least as fast.
}{
\begin{definition}\label{def:tangent-method}
For $\vx \in R$ define $\Ta(\vx)$ to be the unique solution of the equation system
\[
    \begin{array}{rcl}
    \nqoat{\vp_1(\vx_{-1})}\cdot \vX & = & \nqoat{\vp_1(\vx_{-1})}\cdot \vp_1(\vx_{-1})\\
                                                                          & \vdots &\\
  \nqnat{\vp_n(\vx_{-n})}\cdot\vX & = & \nqnat{\vp_n(\vx_{-n})}\cdot \vp_n(\vx_{-n})\\
    \end{array}.
\]
\end{definition}
\begin{theorem}\label{thm:tangent-method}
Let $\vf$ be a clean and feasible scSPP.
We have
\[
 \vx \le \Ta(\vx) \text{ and } \Ne(\vx) \le \Ta(\vx) \le \mu\vf.
\]
Further, for $\vy\in R$ with $\vx \le \vy$
\[
  \Ta(\vx) \le \Ta(\vy).
\]
\end{theorem}
\ifthenelse{\equal{\geoproofs}{true}}{%
\begin{proof}
As $\vx \le \Ne(\vx)$ holds, we only need to show that $\Ne(\vx) \le \Ta(\vx) \le \mu\vf$. We refer the reader to the proof of Theorem~\ref{thm:tang-approx} for
this.

We turn to the monotonicity of $\Ta$.
Let $\vy \in R$ with $\vx \le \vy$.
Assume first that $\vx$ and $\vy$ are located on the surface $S_i$, i.e.\
\[
 h_i(\vx_{-i}) = x_i \text{ and } h_i(\vy_{-i}) = y_i.
\]
The tangent $T_i|_{\vx}$ at $S_i$ in $\vx$ is spanned by the partial derivatives of $\vp_i$
in $\vx$. The part $T_i|_{\vx}\cap [\vx,\mu\vf]$ relevant for $\Ta(\vx)$ can therefore be parameterized by
\[
    \vx + \sum_{j\neq i} \pdat{\vp_{i}}{X_j}{\vx} \cdot ( u_j - x_j ) \text{ with } \vu_{-i} \in [\vx_{-i},\mu\vf_{-i}].
\]
Similarly for $T_i|_{\vy}$.

In particular, for $\vu_{-i}\in [\vy_{-i},\mu\vf_{-i}]$ both points on the tangents defined by $\vu_{-i}$
differ only in the $i$-th coordinate being (the remaining coordinates are simply $\vu_{-i}$)
\[
    t_{\vy} = y_i + \sum_{j\neq i} \pdat{h_i}{X_j}{\vy} \cdot ( u_j - y_j) \text{, resp. }
    t_{\vx} = x_i + \sum_{j\neq i} \pdat{h_i}{X_j}{\vx} \cdot ( u_j - x_j).
\]
By Lemma~\ref{lem:upconv} we have
\[
    y_i \ge x_i + \sum_{j\neq i} \pdat{h_i}{X_j}{\vx} \cdot ( y_j - x_j).
\]
From Lemma~\ref{lem:h-diff} it follows that $\pdat{h_i}{X_j}{\vy} \ge \pdat{h_i}{X_j}{\vx}$.
Thus $t_{\vy} \ge t_{\vx}$ immediately follows.

Now for $\vx,\vy\in R$ with $\vx\le \vy$ we can apply this result to the tangents at $S_i$ in $\vp_i(\vx_{-i})$,
resp. $\vp_i(\vy_{-i})$, and $\Ta(\vx) \le \Ta(\vy)$ follows.
\qed
\end{proof}
}{%
}
By Theorem~\ref{thm:tangent-method}, replacing the Newton operator $\Ne$ by~$\Ta$ gives a variant of Newton's method which converges at least as fast.
}

We do not know whether this variant is substantially faster.
See Figure~\ref{fig:geo-newton} for a geometrical interpretation of both methods.
\begin{figure}[htp]
\begin{center}
\begin{tabular}{cc}
  \scalebox{0.3}{\includegraphics{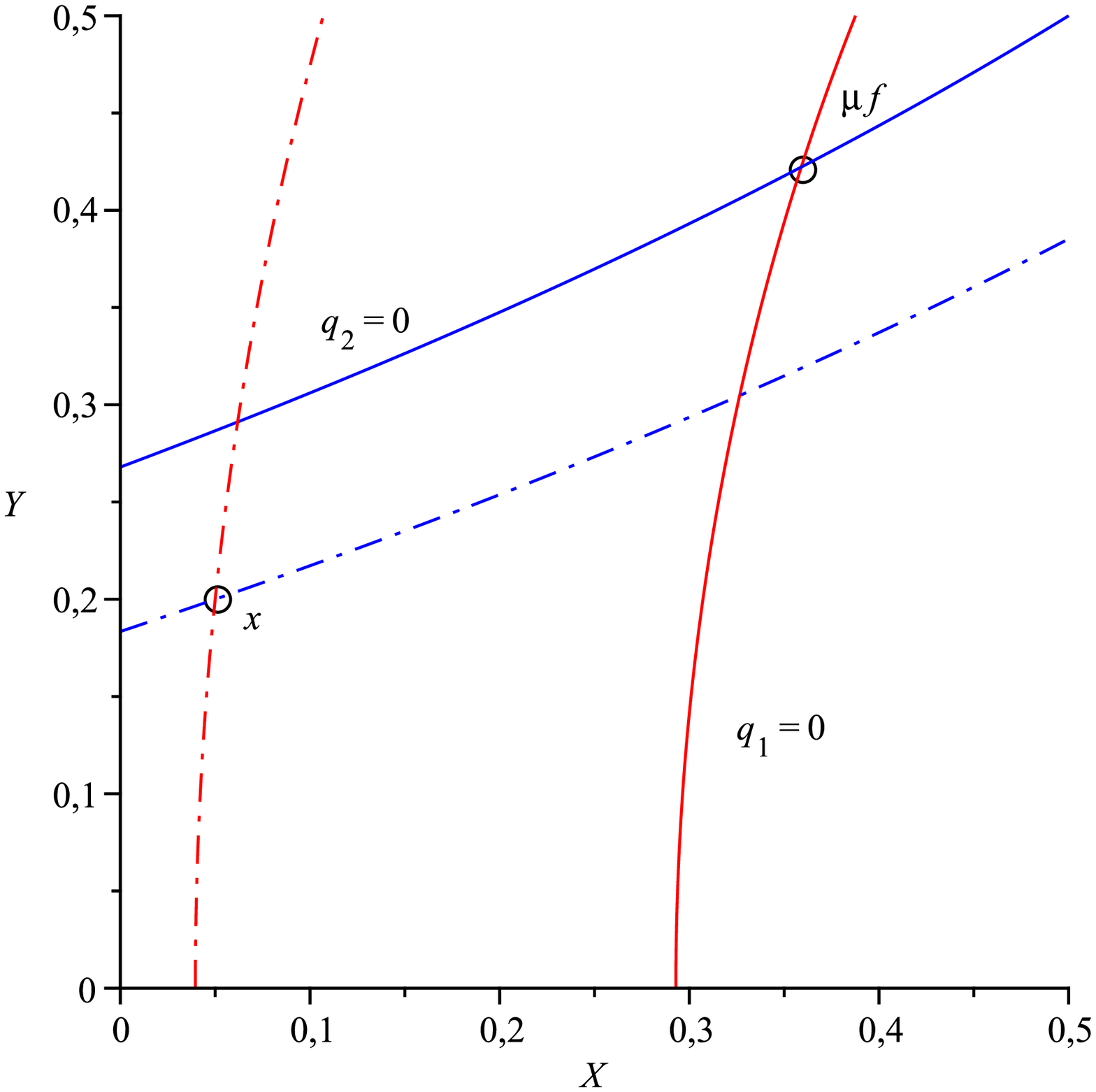}} & \scalebox{0.3}{\includegraphics{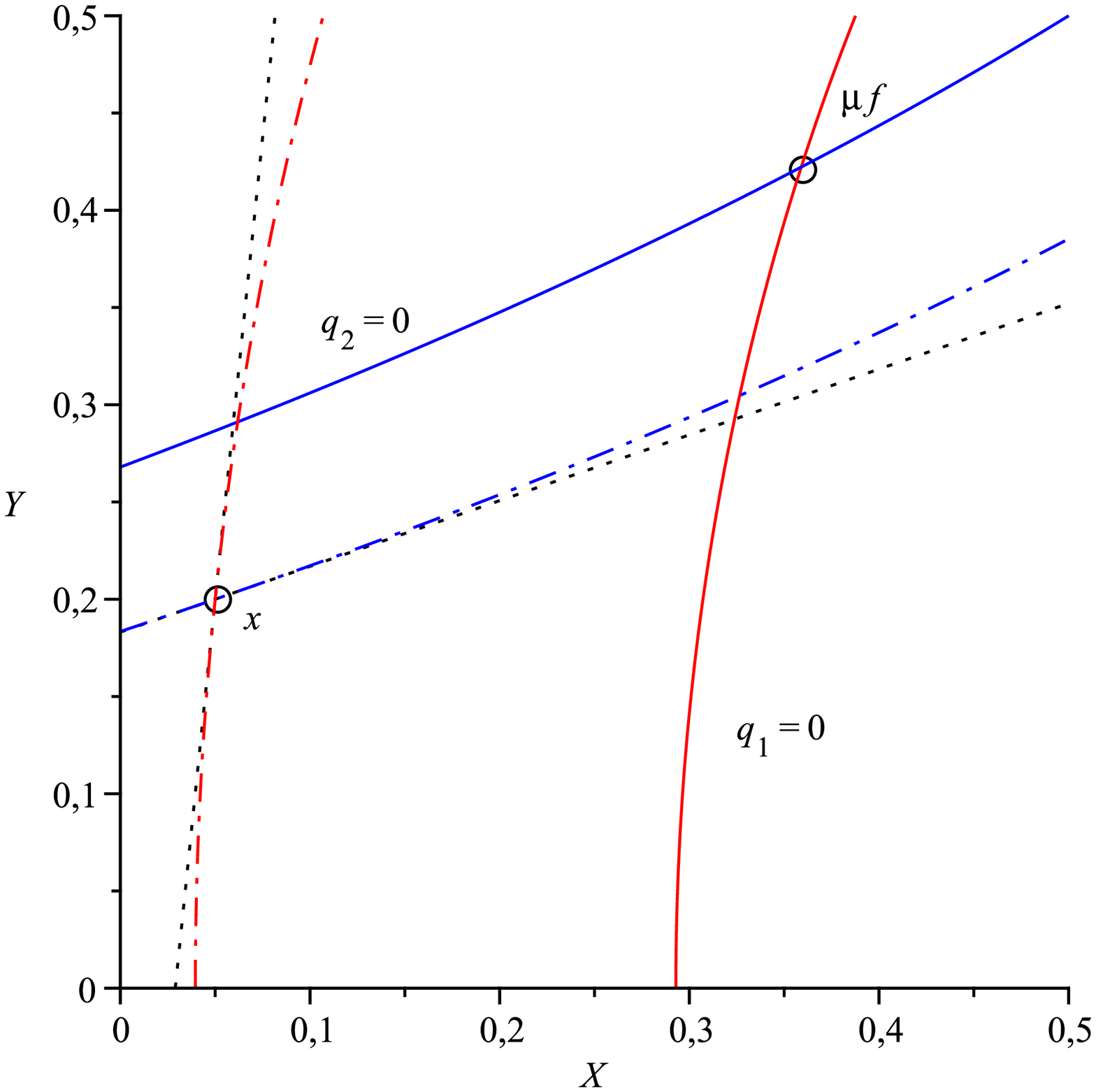}}\\
  (a) & (b)\\
  \scalebox{0.3}{\includegraphics{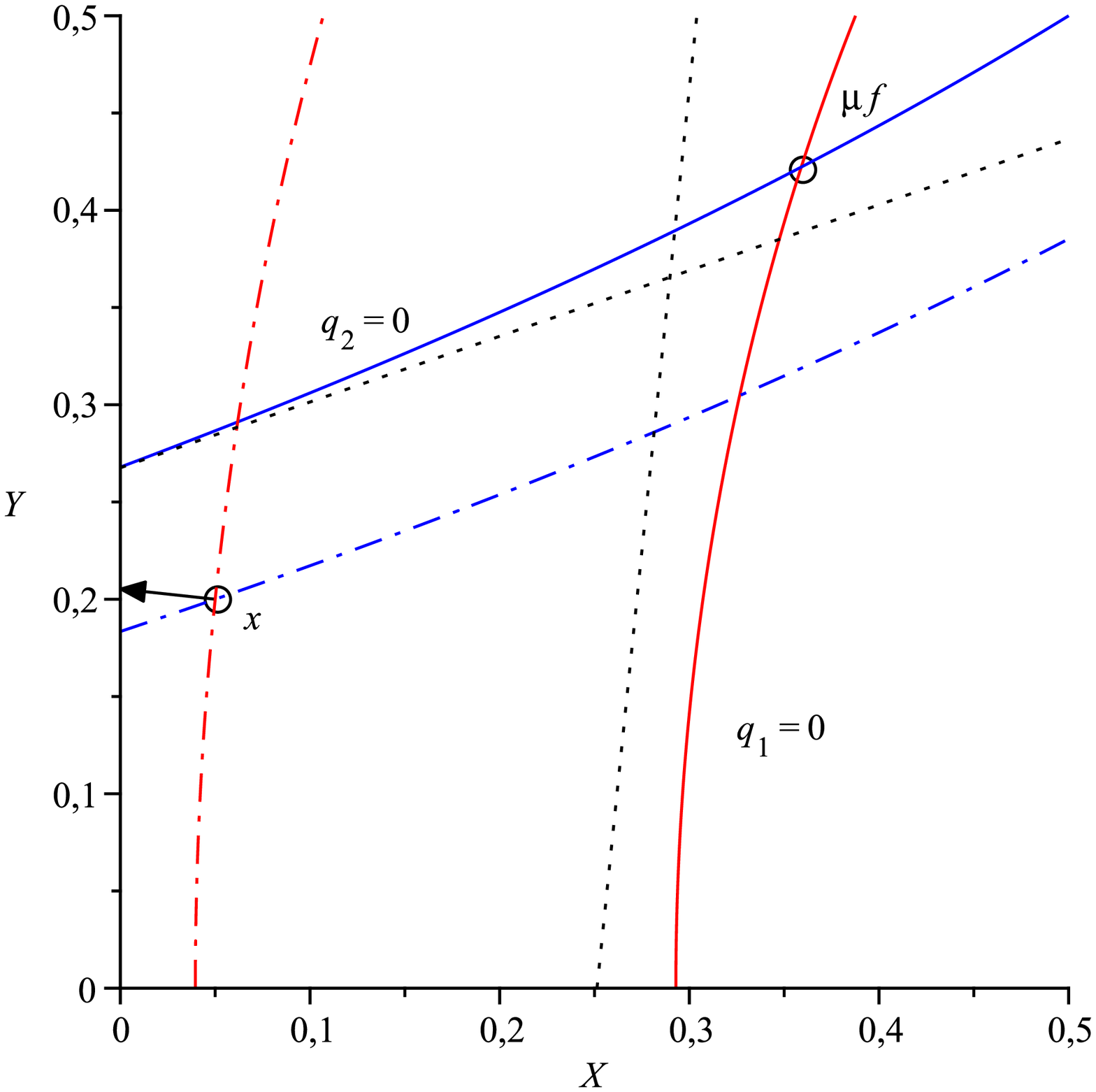}} & \scalebox{0.3}{\includegraphics{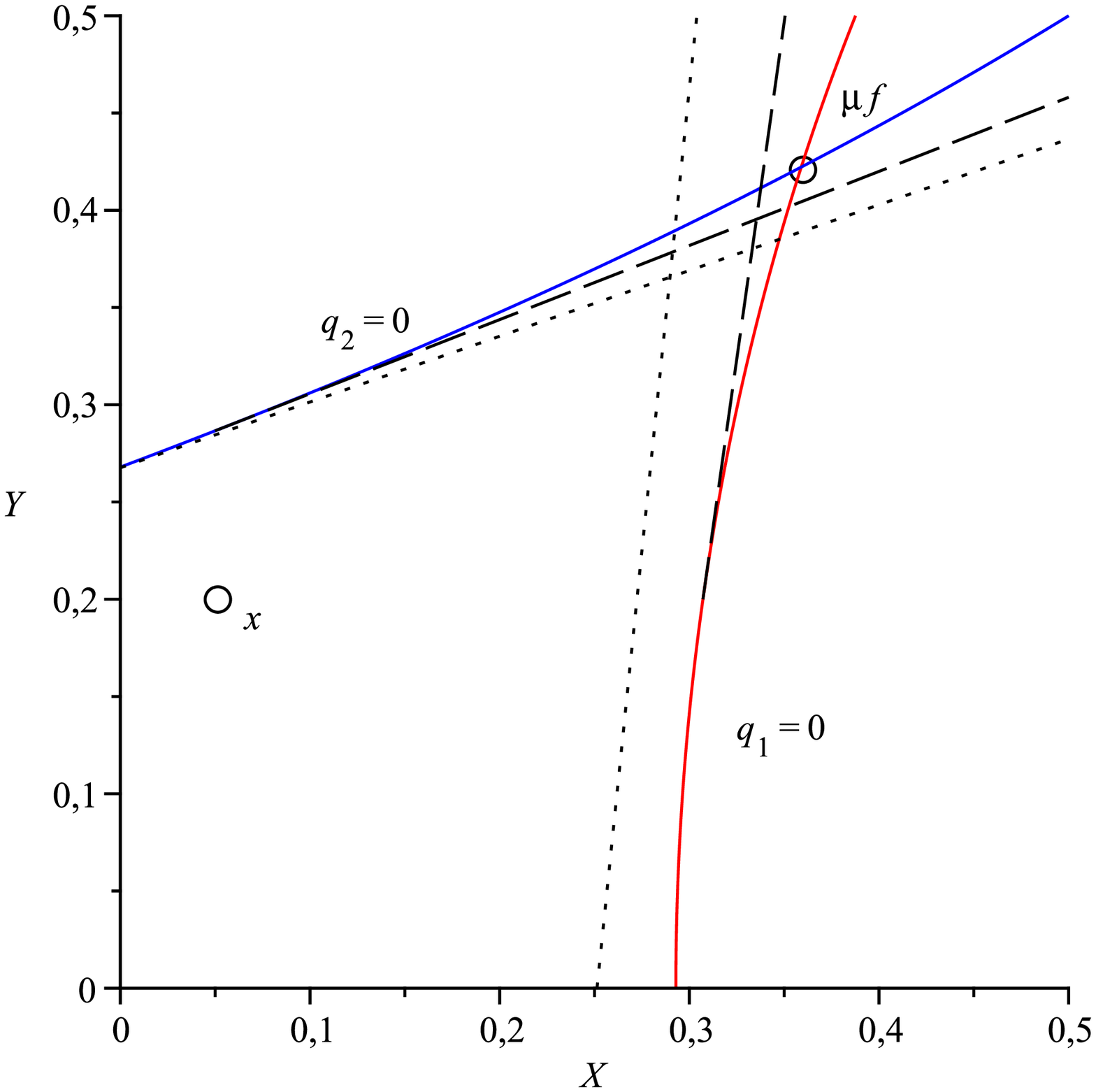}}\\
  (c) & (d)
\end{tabular}
\end{center}
\caption{Geometrical interpretation of Newton's method: (a) Given a point $\vx\in R$ Newton's method first
considers the ``enlarged'' quadrics defined by $q_i(\vX)=q_i(\vx)$ (drawn dashed and dotted) which contain the
current approximation $\vx$.
(b) Then the tangents in $\vx$ at these
enlarged quadrics are computed (drawn dotted), i.e., $\nqix \cdot(\vX-\vx) =0$.
(c) Finally, these tangents are corrected by moving them towards the actual
quadrics, i.e.\ $\nqix \cdot(\vX-\vx) = -q_i(\vx)$. The intersection of these corrected tangents gives the next Newton approximation. (d) A comparison between $\Ne(\vx)$ and $\Ta(\vx)$: $\Ne(\vx)$, resp.\ $\Ta(\vx)$ is given by the
intersection of the dotted, resp.\ dashed lines. Clearly, we have $\Ne(\vx) \le \Ta(\vx)$.}\label{fig:geo-newton}
\end{figure}
\ifthenelse{\equal{0}{1}}{
Let us compare this to the Newton approximation. Given some $\vx\in R$ the next Newton approximation $\Ne(\vx)$ is defined to be
\[
\begin{array}{lcl}
\Ne(\vx)
& = & \vx + \Jfx^\ast \cdot \left( \vf(\vx) - \vx \right)\\[2mm]
& = & \vx + \left( \text{Id} - \Jfx \right)^{-1} \cdot \vq(\vx) \\[2mm]
& = & \vx - \Jqx^{-1} \cdot \vq(\vx) \\[2mm]
& = & \Jqx^{-1} \cdot \left( \Jqx \cdot \vx - \vq(\vx) \right).
\end{array}
\]
This means the new Newton approximant is the unique solution of the linear equation system
\[
\begin{array}{rcl}
 \nqox \cdot \vX & = & \nqox \cdot \vx - q_1(\vx)\\
                             &\vdots&\\
 \nqnx \cdot \vX & = & \nqnx \cdot \vx - q_n(\vx).
\end{array}
\]
As $q_i(\vp_i(\vx_{-i})) = 0$, the two systems result from each
other by interchanging $\vx$ and $\vp_i(\vx_{-i})$ in the $i$th
equation. This gives us the following interpretation of Newton's
method for SPP (see Figure~\ref{fig:geo-newton}):
Newton's method considers in every step the quadrics defined by $q_i(\vX) = q_i(\vx)$,
i.e., it ``inflates'' the quadrics $Q_i$ so that $\vx$ is located on them.
For every $i$ it then takes the tangent at the ``inflated'' quadric $q_i(\vX) = q_i(\vx)$
and moves these tangents towards the original quadric $q_i(\vX) = 0$ by
adjusting the offset of these tangents by $-q_i(\vx)$.
When comparing both methods one has to observe that Newton only requires to solve a linear system in each iteration,
whereas the tangent method also requires the computation of $\vp_{i}(\vx_{-i})$, i.e.,
it needs to solve $n$ quadratic equations in a single variable in order to obtain the values
$h_i(\vx_{-i})$. In particular, the values $h_i(\vx_{-i})$ do not need to be rational, whereas
all Newton steps $\ns{i}$ are rational (assuming that $\vf$ has only rational coefficients).

As a heuristic one can combine both methods:
\begin{itemize}
\item In a first step one approximates $h_i(\vx_{-i})$ using
Newton's method, i.e., as $f_i(\vx_{-i},X_i) = X_i$ is a univariate
SPP, Newton converges monotonically from below, and at least
linearly to $h_i(\vx_{-i})$. Denote by $\eta_i$ this approximation
of $h_i(\vx_{-i})$ with $x_i \le \eta_i \le h_i(\vx_{-i})$. With
the help of $\eta_i$ we obtain an approximation of
$\vp_{i}(\vx_{-i})$ by setting $\vpi_i := (\vx_{-i}, \eta_i)$.
\item As in Newton's method, one then approximates the tangent at
$S_i$ in $\vp_{i}(\vx_{-i})$ using the plane defined by
\[
    \nqiat{\vpi_i}(\vX - \vpi_i) = - q_i(\vpi_i).
\]
The new approximation of $\mu\vf$ is then given by the intersection of these planes.
\end{itemize}
\begin{definition}\label{def:tang-approx}
Let $\vx\in R$.
For $i=1,\ldots,n$ fix some $\eta_i \in [x_i,h_i(\vx_{-i})]$, and set $\veta = (\eta_1,\ldots,\eta_n)$.
We then let $\Ta_{\veta}(\vx)$ denote the solution of
\[
 \nqiat{(\vx_{-i},\eta_i)}(\vX - (\vx_{-i},\eta_i)) = - q_i((\vx_{-i},\eta_i)) \quad (i=1,\ldots,n).
\]
\end{definition}
Note that we have $\Ne(\vX) = \Ta_{\vx}(\vX)$, and $\Ta(\vX)= \Ta_{(h_1(\vx_{-1}),\ldots,h_n(\vx_{-n}))}(\vX)$.
\begin{theorem}\label{thm:tang-approx}
Let $\vx\in R$.
For $i=1,\ldots,n$ fix some $\eta_i \in [x_i,h_i(\vx_{-i})]$, and set $\veta = (\eta_1,\ldots,\eta_n)$.
We then have
\[
\vx \le \Ne(\vx) \le \Ta_{\veta}(\vx) \le \Ta(\vx) \le \mu\vf
\]
\end{theorem}
\ifthenelse{\equal{\geoproofs}{true}}{%
\begin{prf}
Set
\[ \vpi_i := (\vx_{-i},\eta_i) \text{ and } \vh := (h_1(\vx_{-1}),\ldots, h_n(\vx_{-n})). \]
We first show that $\vx \le \Ta_{\veta}(\vx)$:
\[
\begin{array}{lcl}
\Ta_{\veta}(\vx)
& = & \bigl( \nqiat{\vpi_i} \bigr))^{-1}_{i=1,\ldots,n} \cdot \bigl( \nqiat{\vpi_i} \cdot \vpi_i - q_i(\vpi_i) \bigr)_{i=1,\ldots,n}\\[2mm]
& = & \bigl( \nfiat{\vpi_i} \bigr)^\ast_{i=1,\ldots,n} \cdot \bigl( - \nqiat{\vpi_i} \cdot \vpi_i + q_i(\vpi_i) \bigr)_{i=1,\ldots,n}\\[2mm]
& = & \bigl( \nfiat{\vpi_i} \bigr)^\ast_{i=1,\ldots,n} \cdot \bigl( - \nqiat{\vpi_i} \cdot \left(\vx + (\eta_i-x_i) \cdot \ve_i \right) + q_i(\vpi_i) \bigr)_{i=1,\ldots,n}\\[2mm]
& = & \underbrace{\bigl( \nfiat{\vpi_i} \bigr)^\ast_{i=1,\ldots,n}}_{\ge 0 \text{ in every comp.}}
      \cdot \bigl( - \nqiat{\vpi_i} \cdot \vx - \underbrace{\pdat{q_i}{X_i}{\vpi_i}}_{<0} \cdot  \underbrace{(\eta_i-x_i)}_{\ge 0} + \underbrace{q_i(\vpi_i)}_{\ge 0} \bigr)_{i=1,\ldots,n}\\[2mm]
& \ge & \vx.
\end{array}
\]
$\Ta_{\veta}(\vx)$ is by definition the (unique) solution of the equation system defined by
\[
 \nqiat{\vpi_{i}} ( \vX - \vpi_{i} ) = -q_i(\vpi_i) \quad (i=1,\ldots,n).
\]
As $\Ta_{\veta}(\vx) \ge \vx$ we can also consider this system with the origin of the coordinate system
moved into $\vx$, i.e.\
\[
    \nqiat{\vpi_{i}} ( \vX + \vx - \vpi_{i}) = -q_i(\vpi_i) \quad (i=1,\ldots,n).
\]
We show that this system is equivalent to an SPP. For this, we
solve these equations for $X_i$:
\[
 \begin{array}{cl}
                 & \nqiat{\vpi_{i}} ( \vX + \vx - \vpi_{i}) = -q_i(\vpi_i)\\[2mm]
 \Leftrightarrow & \nqiat{\vpi_i} \vX = -q_i(\vpi_i) + \nqiat{\vpi_i} \underbrace{( \vpi_i - \vx )}_{=(\eta_i-x_i)\cdot \ve_i}\\[2mm]
 \Leftrightarrow & X_i = \sum_{j\neq i} \frac{\pdat{q_i}{X_j}{\vpi_i}}{-\pdat{q_i}{X_i}{\vpi_i}} \cdot X_j + \frac{q_i(\vpi_i)}{-\pdat{q_i}{X_i}{\vpi_i}} + (\eta_i-x_i).
 \end{array}
\]
Again, we have $\pdat{q_i}{X_i}{\vpi_i} < 0 \le
\pdat{q_i}{X_j}{\vpi_i}$ as $\vpi_i\in R$, and $\nqiat{\vpi_i}$
monotonically increases with $\eta_i$. Hence, the above linear
equation for $X_i$ is indeed a polynomial with non-negative
coefficients. Denote by $\vf_{\veta}$ the SPP defined by these
linear equations. We then have $\mu \vf_{\veta} =
\Ta_{\veta}(\vx)-\vx$ as the above equation system has
$\Ta_{\veta}(\vx)-\vx \ge \vzero$ as its unique solution. Further,
we know that the Kleene sequence
$\bigl(\vf_{\veta}^k(\vzero)\bigr)_{k\in\N}$ converges to
$\mu\vf_{\veta}$. We show that all coefficients of $\vf_{\veta}$
increase with $\veta \to \vh$. This is straight-forward for
\[
\frac{\pdat{q_i}{X_j}{\vpi_i}}{-\pdat{q_i}{X_i}{\vpi_i}}
\]
as $\pdat{q_i}{X_i}{\vpi_i} < 0 \le \pdat{q_i}{X_j}{\vpi_i}$, and all these terms increase with $\eta_i \to h_i(\vx_{-i})$.
Consider therefore
\[
  0\ge \frac{q_i(\vpi_i)}{-\pdat{q_i}{X_i}{\vpi_i}} + (\eta_i-x_i) = \frac{q_i(\vpi_i) - \pdat{q_i}{X_i}{\vpi_i} (\eta_i - x_i)}{-\pdat{q_i}{X_i}{\vpi_i}}.
\]
We show that this term increases with $\eta_i$. Set $\delta_i := \eta_i - x_i$.
We can find a non-negative, symmetric square-matrix $A$, a vector $\vb$, and constant $c$ such that
\[
 q_i(\vX) = \vX^\top A \vX + \vb^\top \vX + c \text{ and } \nqiat{\vX} = 2 \vX^\top A + \vb^\top.
\]
As $\vpi_i = \vx + \delta_i \ve_i$ we have
\[
    q_i(\vpi_i) = q_i(\vx + \delta_i \ve_i ) = q_i(\vx) + \pdat{q_i}{X_i}{\vx} \delta_i + \delta_i^2 A_{ii},
\]
and
\[
 \pdat{q_i}{X_i}{\vpi_i} \cdot \delta_i = \nqiat{\vx + \delta_i \ve_i} \delta_i\ve_i = \pdat{q_i}{X_i}{\vx} \delta_i + 2
 \delta_i^2 A_{ii}.
\]
This leads to
\[
\frac{q_i(\vpi_i) - \pdat{q_i}{X_i}{\vpi_i} \delta_i}{-\pdat{q_i}{X_i}{\vpi_i}} = \frac{ q_i(\vx) - \delta_i^2 A_{ii}}{-\pdat{q_i}{X_i}{\vx} - 2 \delta_i A_{ii}}.
\]
Deriving this w.r.t.~$\delta_i$ yields:
\[
\begin{array}{cl}
  & \frac{-2 A_{ii} \delta_i}{-\pdat{q_i}{X_i}{\vx} + 2 A_{ii} \delta_i} - \frac{q_i(\vx) - A_{ii} \delta_i^2}{( -\pdat{q_i}{X_i}{\vx} - 2 A_{ii} \delta_i)^2} (-2 A_{ii} )\\[2mm]
= & \frac{2 A_{ii} \pdat{q_i}{X_i}{\vx} \delta_i + 4 A_{ii}^2 \delta_i^2 + 2A_{ii} q_i(\vx) - 2 A_{ii}^2 \delta_i^2}{(-\pdat{q_i}{X_i}{\vx} - 2 A_{ii} \delta_i)^2}\\[2mm]
= & 2 A_{ii} \frac{ A_{ii} \delta_i^2 + \pdat{q_i}{X_i}{\vx} \delta_i + q_i(\vx)}{(-\pdat{q_i}{X_i}{\vx} - 2 A_{ii} \delta_i)^2}\\[2mm]
= & 2 A_{ii} \frac{ q_i(\vpi_i)}{(-\pdat{q_i}{X_i}{\vpi_i})^2}.
\end{array}
\]
As $q_i(\vpi_i) \ge 0$ and $A_{ii} \ge 0$, it follows that
\[
  \frac{q_i(\vpi_i)}{-\pdat{q_i}{X_i}{\vpi_i}} + (\eta_i-x_i)
\]
increases with $\eta_i \to h_i(\vx_{-i})$.
Thus, all coefficients of $\vf_{\veta}$ increase with $\eta_i \to h_i(\vx_{-i})$,
and so for any $\veta'\in [\veta, \vh]$ it follows that
\[
 \vf_{\veta}(\vy) \le \vf_{\veta'}(\vy) \text{ for all } \vy \ge \vzero,
\]
and
\[
 \Ta_{\veta}(\vx)-\vx = \mu\vf_{\veta} \le \mu\vf_{\veta'} = \Ta_{\veta'}(\vx) - \vx.
\]
As $\Ne(\vX) = \Ta_{\vx}(\vX)$ and $\Ta(\vX) = \Ta_{\vh}(\vX)$ we may therefore conclude that
\[
    \Ne(\vx) \le \Ta_{\veta}(\vx) \le \Ta_{\veta'}(\vx) \le \Ta(\vx).
\]
It remains to show that $\Ta(\vx) \le \mu\vf$. This is equivalent to showing
that $\mu \vf_{\vh} \le \mu \vf -\vx$. For $\vf_{\vh}(\vX)$ we have by definition and Lemma~\ref{lem:h-diff}
\[
  \bigl(\vf_{\vh}(\vX)\bigr)_i = \sum_{j\neq i} \frac{\pdat{q_i}{X_j}{\vp_{i}(\vx_{-i})}}{-\pdat{q_i}{X_i}{\vp_{i}(\vx_{-i})}} X_j + (h_i(\vx_{-i})-x_i)
  = \sum_{j\neq i} \pdat{h_i}{X_j}{\vx_{-i}} X_j + (h_i(\vx_{-i})-x_i).
\]
By virtue of Lemma~\ref{lem:upconv} it follows that $\mu\vf$ is above all the tangents, i.e.\
\[
 \vf_{\vh}(\mu\vf - \vx ) \le \mu\vf - \vx.
\]
By monotonicity of $\vf_{\vh}$ we also have
\[
    \vf_{\vh}(\vzero) \le \vf_{\vh}(\mu\vf-\vx).
\]
A straight-forward induction therefore shows that
\[
 \vf^k_{\vh}(\vzero) \le \mu\vf - \vx \quad (\forall k\in \N),
\]
and, thus,
\[
    \Ta(\vx) - \vx = \mu\vf_{\vh} \le \mu\vf -\vx.
\]
\end{prf}
}{%
}
}{}

\ifthenelse{\equal{0}{1}}{%
This geometrical interpretation of Newton's method, resp. its
generalization ($\Ta_{\veta}$) which we have just presented,
readily lends itself to a further generalization to the setting of
min-SPPs. In a min-SPP for every variable $X_i$ we have several
polynomials $f_i^{(1)},\ldots,f_i^{(k_i)}$ all of them having
non-negative coefficients, and $X_i$ is required to be the minimum
of these, i.e.\
\[
  X_i = \min( f_i^{(1)}(\vX), \ldots, f_i^{(k_i)}(\vX) ).
\]
The right-hand side of such a min-SPP defines again a monotone
operator such that Kleene's theorem holds again, and the existence
of the least fixed point is always guaranteed over $\Rp\cup
\{\infty\}$.

For generalizing the operator $\Ta_{\veta}$ to min-SPPs, let us
first reinterpret the tangent method ($\Ta$). Given an
approximation $\vx \in R$ we may also obtain $\Ta(\vx)$ as
follows:
\begin{itemize}
\item Approximate the region $[\vx,\mu\vf] \cap R$ by the convex set defined by
\[
\begin{array}{rcl}
 \nqoat{\vp_1(\vx_{-1})} \cdot \vX & \ge & \nqoat{\vp_1(\vx_{-1})} \cdot \vp_1(\vx_{-1})\\
                             &\vdots&\\
 \nqnat{\vp_n(\vx_{-n})} \cdot \vX & \ge & \nqnat{\vp_n(\vx_{-n})} \cdot \vp_n(\vx_{-n})\\
 \vX & \ge & \vx.
\end{array}
\]
As every tangent safely under-approximates its quadric, the convex set is completely contained in $R$,
i.e., this convex set under-approximates $R$.
See Figure~\ref{fig:tangents} for an example.
\item
  As all tangents are monotonically increasing and intersect in a uniquely determined point
  (see the previous results),
  $\Ta(\vx)$ is then that point from this convex set which
  maximizes the sum of coefficients, i.e.\
  \[
    \sum_{i=1}^{n} X_i \to \max.
  \]
\end{itemize}
To summarize, we can also obtain $\Ta(\vx)$ as the unique solution of the linear program just described.
It is now intuitively clear how the operator $\Ta_{\veta}$ should
be generalized to the setting of min-SPP:
\begin{itemize}
\item
  Let $\vx$ be an approximation of $\mu\vf$ with $f_i^{(j)}(\vx) \ge x_i$ for all
  $1\le i \le n$ and $1\le j \le k_i$. Set $q_i^{(j)} = f_i^{(j)} - X_i$.
\item
  Fix for all $(i,j)$ with $1\le i \le n$, and $1\le j \le k_i$ some $\eta_{i}^{(j)} \in [x_i, h_{i}(\vx_{-i})]$,
  and set $\vpi_i^{(j)} = (\vx_{-i},\eta_i^{(j)})$.
\item
  Take as new approximation of $\mu\vf$ the solution of the linear program defined by (if it exists, otherwise the
  min-SPP is not feasible):
  \[
    \begin{array}{rcl}
      (q_i^{(j)})'( \vpi_i^{(j)} ) \vX & \ge & (q_i^{(j)})'( \vpi_i^{(j)} ) \cdot \vpi_i^{(j)} - q_i^{(j)}(\vpi_i^{(j)})\\
      \vX & \ge & \vx\\
      \sum_{i=1}^{n} X_i & \rightarrow & \max.
    \end{array}
  \]
\end{itemize}

\begin{example}\label{ex:min-SPP}
Consider the min-SPP given by
\[
\begin{array}{rcl}
    X_1 & =& \min( \frac{1}{2}X^2 + \frac{1}{4}Y^2 + \frac{1}{4}, 2Y^2+\frac{1}{4} ),\\
    X_2 & =& \min( \frac{1}{4}X+\frac{1}{4}Y^2+\frac{1}{4}XY, 4 X^2 + 2XY +\frac{1}{8} ).
\end{array}
\]
From this min-SPP we obtain the quadrics defined by
\[
\begin{array}{rcl}
  q_1 & = & \frac{1}{2}X^2 + \frac{1}{4}Y^2 + \frac{1}{4}-X,\\
  q_2 & = & \frac{1}{4}X+\frac{1}{4}Y^2+\frac{1}{4}XY-Y,\\
  q_3 & = & 2Y^2+\frac{1}{4}-X,\\
  q_4 & = & 4 X^2 + 2XY +\frac{1}{8}-Y.
\end{array}
\]
and depicted in Fig~\ref{fig:min-SPP}~(a). The least non-negative
solution $\mu\vf$ of this min-SPP is again the unique non-negative
point contained in the region delimited by the quadrics maximizing
the $1$-norm. Given some approximation $\vx$ we can safely
underapproximate this region by either taking the tangents at
these quadrics (moving from $\vx$ onto the quadrics, see
Figure~\ref{fig:min-SPP}~(b)), or by using the Newton-approximation
of these tangents  (Figure~\ref{fig:min-SPP}~(c)). As new
approximation of $\mu\vf$ we then take the point from the
approximated (convex) region maximizing the $1$-norm.
\end{example}

\begin{figure}[ht]
\begin{center}
\begin{tabular}{ccc}
  \scalebox{0.2}{\includegraphics{sec-geo-min-SPP-ex.eps}} & \scalebox{0.2}{\includegraphics{sec-geo-min-SPP-ex-tangents.eps}} & \scalebox{0.2}{\includegraphics{sec-geo-min-SPP-ex-tangents-Newton.eps}}\\
  (a) & (b) & (c)\\
\end{tabular}
\end{center}
\caption{Approximating the feasible region of a min-SPP: (a) The
quadrics $q_1,q_2,q_3,q_4$ from Example~\ref{ex:min-SPP} and the
region by them ($\{ \vx \in [\vzero,\vinfty] \mid q_i(\vx) \ge 0
\}$). Approximate the region delimited by the quadrics and the
current approximation by (b) the correct tangents ($\vpi_i^{(j)} =
\vp_i^{(j)}(\vx_{-i})$), resp. (c) the approximated tangents used
by Newton's method ($\vpi_i^{(j)} = \vx$). The new approximation
of $\mu\vf$ is encircled in (b), resp. (c), and is obtained by
maximizing the function $X+Y$ on the respective regions.}
\label{fig:min-SPP}
\end{figure}

Let us consider the constrains obtained by the Newton-approximation of the tangents (i.e.\ $\vpi_i^{(j)} = \vx$) at a quadric $q(\vX) = f(\vX) - X_i$ in a bit more detail:
\[
  q'(\vx) \vX \ge  q'(\vx) \vx - q(\vx).
\]
We may write this also as
\[
f'(\vx)\vX  - X_i \ge f'(\vx)\vx - x_i - f(\vx) + x_i \Leftrightarrow X_i \le  f(\vx) + f'(\vx) (\vX-\vx).
\]
In the case of the Newton approximation of the tangents, the
proposed method therefore corresponds to simply replacing each
polynomial on the right-hand side of a min-SPP by its
linearization at the current approximation $\vx$. In
\cite{EGKS08:icalp} it is shown that this approach indeed always
works, and, further, how it can be extended to also handle the
case of min-max-SPPs.

The formal proof of convergence for arbitrary $\eta_i^{(j)} \in [x_i,h_i(\vx_{-i})]$ is future work.
\michael{
VERMUTUNG: wenn $\eta_i$ immer durch genau $k_i$ Newton-/Kleene-Schritte bestimmt wird, dann ist der Operator
$\Ta_{\veta}(\cdot)$ auch wieder monoton.

Hoffnung: $\Ta(\cdot)$ konvergiert sofort linear, $\Ta_{\veta}$ approximiert dieses Verhalten, d.h.
der Threshold nimmt ab, je besser $\veta$ die Punkte auf den Quadriken approximiert.
}
}{}

%% file: sec-conclusions.tex
\section{Conclusions}
\label{sec:conclusions}

We have studied the convergence order and convergence rate of Newton's method
for fixed-point equations of systems of positive polynomials (SPP equations). These equations
appear naturally in the analysis of several stochastic computational
models that have been intensely studied in recent years, and they also play
a central r\^ole in the theory of stochastic branching processes.

The restriction to positive coefficients leads to strong results. For arbitrary
polynomial equations Newton's method
may not converge or converge only locally, i.e., when started at a point sufficiently
close to the solution. We have extended a result by Etessami and Yannakakis
\cite{EYstacs05Extended}, and shown that for SPP equations the method always converges starting at~$\vzero$.
Moreover, we have proved that
the method has at least linear convergence
order, and have determined the asymptotic convergence rate. To the best of our
knowledge, this is the first time that
a lower bound on the convergence order is proved for a significant class
of equations with a trivial membership test.\footnote{Notice the contrast
with the classical result stating that if
$(\Id -\vf'(\mu \vf))$ is non-singular, then Newton's method has exponential
convergence order; here the membership test is highly
non-trivial, and, for what we know, as hard as computing $\mu \vf$ itself.}
Finally, in the case of strongly connected SPPs we have also obtained upper bounds on
the threshold, i.e., the number of iterations necessary to reach the ``steady state'' in which
valid bits are computed at the asymptotic rate. These results
lead to practical tests for checking whether the
least fixed point of a strongly connected SPP exceeds a given bound.

It is worth mentioning that in a recent paper we study the behavior of Newton's
method when arithmetic operations only have a fixed
accuracy~\cite{EGK10:stacs}. We develop an algorithm for a relevant class of SPPs
that computes iterations of Newton's method increasing the accuracy on demand. A simple test
applied after each iteration decides if the round-off errors 
have become too large, in which case the accuracy is increased.

There are still at least two important open questions. The first one is,
can one provide a bound on the threshold valid for arbitrary SPPs, and not only for strongly
connected ones? Since SPPs cannot be solved exactly in general, we cannot
first compute the exact solution for the bottom SCCs, insert it in the SCCs above them, and iterate.
We can only compute an approximation, and we are not currently able to bound the propagation of the error.
For the second question, say that Newton's method
is {\em polynomial} for a class of SPP equations if
there is a polynomial $p(x,y,z)$ such that
for every $k \geq 0$ and for every system in the class with $n$ equations
and coefficients of size $m$,
the $p(n,m,k)$-th Newton approximant $\ns{p(n,m,k)}$ has $k$ valid bits.
We have proved in Theorem \ref{thm:estimate-cor} that Newton's method
is polynomial for strongly connected SPPs $\vf$ satisfying
$\vf(\vzero) \succ \vzero$; for this class
one can take $p(n,m,k) = 7mn+k$. We have also exhibited in
\S~\ref{sec:upper-bounds} a class for which computing the first
bit of the least solution takes $2^n$ iterations.
The members of this class, however, are not strongly connected, and
this is the fact we have exploited to construct them. So the following question
remains open: Is Newton's method polynomial for strongly connected SPPs?

%% file: acknowledgments.tex
\mbox{} \par \vspace{6mm} {\bf Acknowledgments.}
We thank Kousha Etessami for several 
illuminating discussions,
and two anonymous referees for helpful suggestions.

%% file: app5-decomposed.tex
\section{Proof of Lemma~\ref{lem:propagation-error}} \label{app:proof-lem-propagation-error}

The proof of Lemma~\ref{lem:propagation-error} is by a sequence of lemmata.
The proof of Lemma~\ref{lem:miss-lower-bound-quadratic} and, consequently, the proof of Lemma~\ref{lem:propagation-error} are non-constructive
 in the sense that we cannot give a particular~$C_\vf$.
Therefore, we often use the equivalence of norms, disregard the constants that link them, and state the results in terms of an arbitrary norm.

\iftechrep{}{
First we prove the following lemma on cone vectors.
\lemvdvzeroisenough
}

The following two Lemmata \ref{lem:miss-lower-bound-quadratic} and~\ref{lem:miss-lower-bound}
 provide a lower bound on $\norm{\vf(\vx) - \vx}$ for an ``almost-fixed-point''~$\vx$.

\begin{lemma} \label{lem:miss-lower-bound-quadratic}
 Let $\vf$ be a quadratic, clean and feasible SPP without linear terms,
  i.e., $\vf(\vX) = B(\vX,\vX) + \vc$ where $B$ is a bilinear map, and $\vc$ is a constant vector.
 Let $\vf(\vX)$ be non-constant in every component.
 Let $R \dotcup S = \{1,\ldots,n\}$ with $S \ne \emptyset$.
 Let every component depend on every $S$-component and not on any $R$-component.
 Then there is a constant~$C_\vf > 0$ such that
 \[
  \norm{\vf(\mu\vf - \vdelta) - (\mu\vf - \vdelta)} \ge C_\vf \cdot \norm{\vdelta}^2
 \]
 for all $\vdelta$ with $\vzero \le \vdelta \le \mu\vf$.
\end{lemma}

\begin{proof}
 With the given component dependencies we can write $\vf(\vX)$ as follows:
 \[
  \vf(\vX) =
  \begin{pmatrix}
   \vf_R(\vX) \\
   \vf_S(\vX)
  \end{pmatrix}
  =
  \begin{pmatrix}
   B_R(\vX_S,\vX_S) + \vc_R \\
   B_S(\vX_S,\vX_S) + \vc_S
  \end{pmatrix}
 \]
 A straightforward calculation shows
  \[
   \ve(\vdelta) := \vf(\mu\vf - \vdelta) - (\mu\vf - \vdelta) = (\Id - \vf'(\mu\vf)) \vdelta + B(\vdelta,\vdelta)\;.
  \]
 Furthermore, $\pd{\vf}{\vX_R}$ is constant zero in all entries, so
  \begin{align*}
   \ve_R(\vdelta) & = \vdelta_R - \pd{\vf_R}{\vX_S}(\mu\vf) \cdot \vdelta_S + B_R(\vdelta_S,\vdelta_S)  \text{\qquad and} \\
   \ve_S(\vdelta) & = \vdelta_S - \pd{\vf_S}{\vX_S}(\mu\vf) \cdot \vdelta_S + B_S(\vdelta_S,\vdelta_S) \;.
  \end{align*}
 Notice that for every real number $r > 0$ we have
  \[
   \min_{\vzero \le \vdelta \le \mu\vf, \norm{\vdelta} \ge r} \frac{\norm{\ve(\vdelta)}}{\norm{\vdelta}^2} > 0 \;,
  \]
  because otherwise $\mu\vf - \vdelta < \mu\vf$ would be a fixed point of~$\vf$.
 We have to show:
  \[
   \inf_{\vzero \le \vdelta \le \mu\vf, \norm{\vdelta} > 0} \frac{\norm{\ve(\vdelta)}}{\norm{\vdelta}^2} > 0
  \]
 Assume, for a contradiction, that this infimum equals zero.
\newcommand{\rs}[1]{r^{(#1)}}
 Then there exists a sequence~$(\ds{i})_{i\in\N}$ with $\vzero \le \ds{i} \le \mu\vf, \norm{\ds{i}} > 0$ such that
  $\lim_{i \to \infty} \norm{\ds{i}} = 0$ and $\lim_{i \to \infty} \frac{\norm{\ve(\ds{i})}}{\norm{\ds{i}}^2} = 0$.
 Define $\rs{i} := \norm{\ds{i}}$ and $\des{i} := \frac{\ds{i}}{\norm{\ds{i}}}$.
 Notice that $\des{i} \in \{ \vd \in \Rp^n \mid \norm{\vd} = 1 \} =: D$ where $D$ is compact.
 So some subsequence of~$(\des{i})_{i\in\N}$, say w.l.o.g.\ the sequence $(\des{i})_{i\in\N}$ itself,
  converges to some vector~$\vd^* \in D$.
 By our assumption we have
  \begin{equation} \label{eq:proof-miss-lower-bound-quadratic}
   \norm{\ve(\ds{i})}/\norm{\ds{i}}^2 = \norm{ \frac{1}{\rs{i}} (\Id - \vf'(\mu\vf)) \des{i} + B(\des{i},\des{i})} \longrightarrow 0 \;.
  \end{equation}
 As $B(\des{i},\des{i})$ is bounded, $\frac{1}{\rs{i}} (\Id - \vf'(\mu\vf)) \des{i}$ must be bounded, too.
 Since $\rs{i}$ converges to 0, $\norm{(\Id - \vf'(\mu\vf)) \des{i}}$ must converge to~$0$, so
  \[ (\Id - \vf'(\mu\vf))\vd^* = \vzero \;.
  \]
 In particular,
  $\left( (\Id - \vf'(\mu\vf)) \vd^* \right)_R = \vd^*_R - \pd{\vf_R}{\vX_S}(\mu\vf) \cdot \vd^*_S = \vzero$.
 So we have $\vd^*_S > \vzero$, because $\vd^*_S = \vzero$ would imply $\vd^*_R = \vzero$
  which would contradict $\vd^* > \vzero$.

 In the remainder of the proof we focus on $\vf_S$.
 Define the scSPP $\vg(\vX_S) := \vf_S(\vX)$.
 Notice that $\mu\vg = \mu\vf_S$.
 We can apply Lemma~\ref{lem:vd>vzero-is-enough} to $\vg$ and~$\vd^*_S$ and obtain $\vd^*_S \succ \vzero$.
 As $\vf_S(\vX)$ is non-constant we get $B_S(\vd^*_S,\vd^*_S) \succ \vzero$.
 By~\eqref{eq:proof-miss-lower-bound-quadratic},
  $\frac{1}{\rs{i}} (\Id - \vg'(\mu\vg)) \des{i}_S$ converges to~$-B_S(\vd^*_S,\vd^*_S) \prec \vzero$.
 So there is a $j\in\N$ such that $(\Id - \vg'(\mu\vg)) \des{j}_S \prec \vzero$.
 \newcommand{\tvdelta}{\widetilde{\vdelta}}
 Let $\tvdelta := r \des{j}$ for some small enough $r > 0$ such that $\vzero < \tvdelta_S \le \mu\vg$ and
  \begin{align*}
   \ve_S(\tvdelta) & = (\Id - \vg'(\mu\vg)) \tvdelta_S + B_S(\tvdelta_S,\tvdelta_S) \\
                   & = r (\Id - \vg'(\mu\vg)) \des{j}_S + r^2 B_S(\des{j}_S, \des{j}_S) \prec \vzero \;.
  \end{align*}
 So we have $\vg(\mu\vg - \tvdelta_S) \prec \mu\vg - \tvdelta_S$.
 However, $\mu\vg$ is the least point $\vx$ with $\vg(\vx) \le \vx$.
 Thus we get the desired contradiction.
\qed
\end{proof}

\begin{lemma} \label{lem:miss-lower-bound}
 Let $\vf$ be a quadratic, clean and feasible scSPP.
 Then there is a constant~$C_\vf > 0$ such that
 \[
  \norm{\vf(\mu\vf - \vdelta) - (\mu\vf - \vdelta)} \ge C_\vf \cdot \norm{\vdelta}^2
 \]
 for all $\vdelta$ with $\vzero \le \vdelta \le \mu\vf$.
\end{lemma}

\begin{proof}
 Write $\vf(\vX) = B(\vX,\vX) + L \vX + \vc$ for a bilinear map~$B$, a matrix~$L$ and a constant vector~$\vc$.
 \newcommand{\tvf}{\widetilde{\vf}}
 By Theorem~\ref{thm:well-defined}.2.\ the matrix $L^* = (\Id - L)^{-1} = (\Id - \vf'(\vzero))^{-1}$ exists.
 Define the SPP~$\tvf(\vX) := L^* B(\vX,\vX) + L^* \vc$.
 A straightforward calculation shows that the sets of fixed points of $\vf$ and $\tvf$ coincide and that
  \[
   \vf(\mu\vf - \vdelta) - (\mu\vf - \vdelta) = (\Id - L) \left( \tvf(\mu\vf - \vdelta) - (\mu\vf - \vdelta) \right) \;.
  \]
 Further, if $\sigma_n(\Id - L)$ denotes the smallest singular value of~$\Id - L$, we have by basic facts about singular values
  (see~\cite{HornJ91:book}, Chapter~3) that
 \[
  \norm{(\Id - L) \left( \tvf(\mu\vf - \vdelta) - (\mu\vf - \vdelta) \right)}_2 \ge
  \sigma_n(\Id - L) \norm{ \tvf(\mu\vf - \vdelta) - (\mu\vf - \vdelta) }_2\;.
 \]
 Note that $\sigma_n(\Id - L) > 0$ because $\Id - L$ is invertible.
 So it suffices to show that
  \[ 
   \norm{ \tvf(\mu\vf - \vdelta) - (\mu\vf - \vdelta) } \ge C_\vf \cdot \norm{\vdelta}^2 \;.
  \] 
 If $\vf(\vX)$ is linear (i.e.\ $B(\vX,\vX) \equiv \vzero$) then $\tvf(\vX)$ is constant
  and we have $\norm{ \tvf(\mu\vf - \vdelta) - (\mu\vf - \vdelta) } = \norm{\vdelta}$, so we are done in that case.
 Hence we can assume that some component of~$B(\vX,\vX)$ is not the zero polynomial.
 It remains to argue that $\tvf$ satisfies the preconditions of Lemma~\ref{lem:miss-lower-bound-quadratic}.
 By definition, $\tvf$ does not have linear terms.
 Define
  \[
   S := \{ i \mid 1 \le i \le n, \ X_i \text{ is contained in a component of } B(\vX,\vX) \} \;.
  \]
 Notice that $S$ is non-empty.
 Let $i_0, i_1, \ldots, i_m, i_{m+1}$ ($m \ge 0$) be any sequence such that, in~$\vf$, for all $j$ with $0 \le j < m$
  the component~$i_j$ depends directly on~$i_{j+1}$ via a linear term and $i_m$ depends directly on~$i_{m+1}$ via a quadratic term.
 Then $i_0$ depends directly on~$i_{m+1}$ via a quadratic term in~$L^mB(\vX,\vX)$ and hence also in~$\tvf$.
 So all components are non-constant and depend (directly or indirectly) on every $S$-component.
 Furthermore, no component depends on a component that is not in~$S$, because $L^*B(\vX,\vX)$ contains only $S$-components.
 Thus, Lemma~\ref{lem:miss-lower-bound-quadratic} can be applied, and the statement follows.
\qed
\end{proof}

The following lemma gives a bound on the propagation error for the case that $\vf$ has a single top SCC.
\newcommand{\mus}{\widetilde{\vmu}_S}
\begin{lemma} \label{lem:prop-error-top}
 Let $\vf$ be a quadratic, clean and feasible SPP.
 Let $S \subseteq \{1,\ldots,n\}$ be the single top SCC of~$\vf$.
 Let $L := \{1,\ldots,n\} \setminus S$.
 Then there is a constant~$C_\vf \ge 0$ such that
 \[
  \norm{\mu\vf_S - \mus} \le C_\vf \cdot \sqrt{\norm{\mu\vf_L - \vx_L}}
 \]
 for all $\vx_L$ with $\vzero \le \vx_L \le \mu\vf_L$ where $\mus := \mu\left( \vf_S[\vX_L / \vx_L] \right)$.
\end{lemma}
\begin{proof}
 We write $\vf_S(\vX) = \vf_S(\vX_S,\vX_L)$ in the following.

 If $S$ is a trivial SCC then $\mu\vf_S = \vf_S(\vzero,\mu\vf_L)$ and $\mus = \vf_S(\vzero,\vx_L)$.
 In this case we have with Taylor's theorem (cf.~Lemma~\ref{lem:taylor})
  \begin{align*}
   \norm{\mu\vf_S - \mus}
     & =   \norm{\vf_S(\vzero,\mu\vf_L) - \vf_S(\vzero,\vx_L)} \\
     & \le \norm{\pd{\vf_S}{\vX}(\vzero,\mu\vf_L) \cdot (\mu\vf_L - \vx_L)} \\
     & \le \norm{\pd{\vf_S}{\vX}(\vzero,\mu\vf_L)} \cdot \norm{\mu\vf_L - \vx_L} \\
     & =   \norm{\pd{\vf_S}{\vX}(\vzero,\mu\vf_L)} \cdot \sqrt{\norm{\mu\vf_L - \vx_L}} \cdot \sqrt{\norm{\mu\vf_L - \vx_L}} \\
     & \le \norm{\pd{\vf_S}{\vX}(\vzero,\mu\vf_L)} \cdot \sqrt{\norm{\mu\vf_L}} \cdot \sqrt{\norm{\mu\vf_L - \vx_L}}
  \end{align*}
  and the statement follows by setting $C_\vf := \norm{\pd{\vf_S}{\vX}(\vzero,\mu\vf_L)} \cdot \sqrt{\norm{\mu\vf_L}}$.

 Hence, in the following we can assume that $S$ is a non-trivial SCC.
 Set $\vg(\vX_S) := \vf_S(\vX_S, \mu\vf_L)$.
 Notice that $\vg$ is an scSPP with $\mu\vg = \mu\vf_S$.
 By applying Lemma~\ref{lem:miss-lower-bound} to $\vg$ and setting $c := 1 / \sqrt{C_\vg}$
  (the $C_\vg$ from Lemma~\ref{lem:miss-lower-bound}) we get
 \begin{align*}
   \norm{\mu\vf_S - \mus}
     & \le c \cdot \sqrt{\norm{ \vg(\mu\vg - (\mu\vf_S - \mus)) - (\mu\vg - (\mu\vf_S - \mus)) }} \\
     & =   c \cdot \sqrt{\norm{ \vf_S(\mus,\mu\vf_L) - \mus }} \\
     & =   c \cdot \sqrt{\norm{ \vf_S(\mus,\mu\vf_L) - \vf_S(\mus,\vx_L) }}
  \intertext{and with Taylor's theorem (cf.~Lemma~\ref{lem:taylor})}
     & \le c \cdot \sqrt{\norm{ \pd{\vf_S}{\vX_L}(\mus,\mu\vf_L) (\mu\vf_L - \vx_L) }} \\
     & \le c \cdot \sqrt{\norm{ \pd{\vf_S}{\vX_L}(\mu\vf_S,\mu\vf_L) (\mu\vf_L - \vx_L) }} \\
     & \le c \cdot \sqrt{\norm{ \pd{\vf_S}{\vX_L}(\mu\vf_S,\mu\vf_L) }} \cdot \sqrt{\norm{ \mu\vf_L - \vx_L }}\,.
 \end{align*}
 So the statement follows by setting $C_\vf := c \cdot \sqrt{\norm{ \pd{\vf_S}{\vX_L}(\mu\vf_S,\mu\vf_L) }}$.
\qed
\end{proof}

Now we can extend Lemma~\ref{lem:prop-error-top} to Lemma~\ref{lem:propagation-error}, restated here.

\vspace{3mm}
\begin{qlemma}{\ref{lem:propagation-error}}
 \stmtlempropagationerror
\end{qlemma}

\begin{proof}
 Observe that $\mu\vf_{[t]}$, $\tvmu_{[t]}$, $\mu\vf_{[\mathord{>}t]}$ and $\vrho_{[\mathord{>}t]}$ do not depend on the components of depth $<t$.
 So we can assume w.l.o.g.\ that $t=0$.
 Let $\SCC(0) = \{S_1, \ldots, S_k\}$.

 For any $S_i$ from $\SCC(0)$, let $\vf^{(i)}$ be obtained from~$\vf$ by removing all top SCCs except for~$S_i$.
 Lemma~\ref{lem:miss-lower-bound} applied to~$\vf^{(i)}$ guarantees a $C^{(i)}$ such that
  \[ \norm{\mu\vf_{S_i} - \tvmu_{S_i}} \le C^{(i)} \cdot \sqrt{\norm{\mu\vf_{[>0]} - \vrho_{[>0]}}}
  \]
  holds for all $\vrho_{[>0]}$ with $\vzero \le \vrho_{[>0]} \le \mu\vf_{[>0]}$.
 Using the equivalence of norms let w.l.o.g.\ the norm $\norm{\cdot}$ be the maximum-norm~$\norm{\cdot}_\infty$.
 Let $C_\vf := \max_{1 \le i \le k} C^{(i)}$.
 Then we have
  \[ \norm{\mu\vf_{[0]} - \tvmu_{[0]}}
       =   \max_{1 \le i \le k} \norm{\mu\vf_{S_i} - \tvmu_{S_i}}
       \le C_\vf \cdot \sqrt{\norm{\mu\vf_{[>0]} - \vrho_{[>0]}}}
  \]
  for all $\vrho_{[>0]}$ with $\vzero \le \vrho_{[>0]} \le \mu\vf_{[>0]}$.
\qed
\end{proof}

%% file: app-geo.tex
\section{Proofs of \S~\ref{sec:geo}}
\subsection{Proof of Lemma~\ref{lem:normal}}\hspace{1mm} \\
\begin{qlemma}{\ref{lem:normal}}
For every quadric $q_i$ induced by a clean and feasible scSPP $\vf$ we have
\[
    \nqix = \left(\pdat{q_i}{X_1}{\vx},\pdat{q_i}{X_2}{\vx},\ldots,\pdat{q_i}{X_n}{\vx}\right) \neq \vzero \text{ and } \pdat{q_i}{X_i}{\vx} < 0 \quad \forall \vx \in [\vzero,\mu\vf).
\]
\end{qlemma}
\begin{proof}
As shown by Etessami and Yannakakis in \cite{EYstacs05Extended} under the above preconditions
it holds for all $\vx\in [\vzero,\mu\vf)$ that
$\bigl(\text{Id} - \Jfx\bigl)$ is invertible with
\[
  \bigl(\text{Id} - \Jfx\bigl)^{-1} = \Jfx^\ast.
\]
Thus, we have
\[
    \Jqx^{-1} = \bigl(\Jfx - \text{Id}\bigr)^{-1} = - \bigl( \Jfx^\ast \bigr),
\]
implying that $\nqix \neq \vzero$ for all $\vx\in [\vzero,\mu\vf)$ as $\Jqx$ has to have full rank $n$
in order for $\Jqx^{-1}$ to exist.
Furthermore, it follows that all entries of $\Jqx^{-1}$ are non-positive as $\Jfx^\ast$ is non-negative.
Now, as $q_i(\vX) = f_i(\vX)- X_i$ and $f_i(\vX)$ is a polynomial with non-negative coefficients, it holds that
\[
 \nqix \cdot \ve_j = \pdat{q_i}{X_j}{\vx} = \pdat{f_i}{X_j}{\vx} \ge 0
\]
for all $j \neq i$ and $\vx\ge \vzero$.
With every entry of $\Jqx^{-1}$ non-positive, and
\[
\nqix \cdot \Jqx^{-1} = \ve_i^\top,
\]
we conclude $\pdat{q_i}{X_i}{\vx} < 0$.
\qed
\end{proof}

\subsection{Proof of Lemma~\ref{lem:h-diff}}\hspace{1mm}\\

We first summarize some properties of the functions $h_i$:
\begin{proposition}\label{prop:h}
Let $\vf$ be a clean and feasible scSPP.
Let $\vx,\vy \in [\vzero, \mu\vf]$ with $\vx \le \vy$.
\begin{itemize}
\item[(a)] $0 \le h_i^{(k)}(\vx_{-i}) \le \mu \vf_i$.
\item[(b)] $h_i^{(k)}(\vx_{-i}) \le h_i^{(k+1)}(\vx_{-i})$ for all $k\in\N$.
\item[(c)] $h_i^{(k)}(\vx_{-i}) \le h_i^{(k)}(\vy_{-i})$ for all $k\in\N$.
\item[(d)] $h_i(\vx_{-i}) \le \mu\vf_i$, and $h_i$ is
           a map from $[\vzero, \mu\vf_{-i}]$ to $[0,\mu\vf_i]$.

           If $f_i$ depends on at least one other variable except $X_i$, we also have
           $h_i( [\vzero, \mu\vf_{-i} ) ) \subseteq [\vzero,\mu\vf_i)$.
\item[(e)] $h_i(\vx_{-i}) \le h_i(\vy_{-i})$.
\item[(f)] $f_i( \vx_{-i}, h_i(\vx_{-i}) ) = h_i(\vx_{-i})$.
\item[(g)] For $x_i = f_i(\vx)$ we have $h_i(\vx_{-i}) \le x_i$.
\item[(h)] $h_i(\mu\vf_{-i}) = \mu\vf_i$.
\end{itemize}
\end{proposition}
%
\begin{prf}
Let $\vzero \le \vx \le \vy \le \mu\vf$.
Using the monotonicity of $f_i$ over $\Rp^n$ we proceed by induction on $k$.
\begin{itemize}
\item[(a)]
  For $k=0$ we have
  \[
    0 \le h_i^{(0)}(\vx_{-i}) = f_i(0,\vx_{-i}) \le f_i( \mu\vf) = \mu\vf_i.
  \]
  We then get
  \[
    0 \le h_i^{(k+1)}(\vx_{-i}) = f_i( h_i^{(k)}(\vx_{-i}), \vx_{-i} ) \le f_i( \mu\vf ) = \mu\vf_i.
  \]
\item[(b)]
  For $k=0$ we have
  \[
    h_i^{(0)}(\vx_{-i}) = f_i( 0, \vx_{-i} ) \le f_i( h_i^{(0)}(\vx_{-i}),\vx_{-i} ) = h_i^{(1)}(\vx_{-i}).
  \]
  Thus
  \[
   h_i^{(k+1)}(\vx_{-i}) = f_i( h_i^{(k)}(\vx_{-i}), \vx_{-i} ) \le f_i( h_i^{(k+1)}(\vx_{-i}),\vx_{-i} ) = h_i^{(k+2)}(\vx_{-i})
  \]
  follows.
\item[(c)]
  As $\vx \le \vy$, we have for $k=0$
  \[
    h_i^{(0)}(\vx_{-i}) = f_i(0,\vx_{-i}) \le f_i(0,\vy_{-i}) = h_i^{(0)}(\vy_{-i}).
  \]
  Hence, we get
  \[
    h_i^{(k+1)}(\vx_{-i}) = f_i( h_i^{(k)}(\vx_{-i}),\vx_{-i}) \le f_i( h_i^{(k)}(\vy_{-i}),\vy_{-i}) = h_i^{(k+1)}(\vy_{-i}).
  \]
\item[(d)]
  As the sequence $(h_i^{(k)}(\vx_{-i}))_{k\in\N}$ is monotonically increasing and bounded from above
  by $\mu\vf_i$, the sequence converges.
  Thus, for every $\vx$ the value
  \[
    h_i(\vx_{-i}) = \lim_{k\to\infty} h_i^{(k)}(\vx_{-i})
  \]
  is well-defined, i.e., $h_i$ is a map from $[\vzero,\mu\vf_{-i}]$ to $[0,\mu f_i]$.

  If $f_i$ depends on at least one other variable except $X_i$,
  then $h_i$ is a non-constant power series in this variable with non-negative coefficients.
  For $\vx_{-i} \in [\vzero,\mu\vf_{-i})$ we thus always have
  \[
    h_i( \vx_{-i} ) < h_i( \mu\vf_{-i} )  = \mu\vf_i
  \]
  as $\vx_{-i} \prec \mu\vf_{-i}$.
\item[(e)]
  This follows immediately from (b).
\item[(f)]
  As $f_i$ is continuous, we have
 \[ f_i( h_i(\vx_{-i}), \vx_{-i} ) = f_i( \lim_{k\to\infty} h_i^{(k)}(\vx_{-i}), \vx_{-i}) = \lim_{k\to\infty} h_i^{(k+1)}(\vx_{-i}) = h_i(\vx_{-i}),
 \]
 where the last equality holds because of (b).
\item[(g)]
  Using induction similar to (a) replacing $\mu\vf$ by $\vx$,
  one gets $h_i^{(k)}(\vx_{-i}) \le x_i$ for all $k\in\N$ as $f_i(\vx_{-i}) = x_i$.
  Thus, $h_i(\vx_{-i}) \le x_i$ follows similarly to (d).
\item[(h)]
  By definition, we have $\mu \vf = \lim_{k\to\infty} \vf^k(\vzero)$.
  For $k=0$, we have
  \[ (\vf^{0}(\vzero))_i = 0 \le f_i(0,\mu\vf_{-i}) = h_i^{(0)}(\mu\vf_{-i}). \]
  We thus get by induction
  \[
    (\vf^{(k+1)}(\vzero))_i = f_i( \vf^{k}(\vzero) ) \le f_i( h_i^{(k)}(\mu\vf_{-i}), \mu\vf_{-i}) = h_i^{(k+1)}(\mu\vf_{-i}).
  \]
  Thus, we may conclude $\mu\vf_i \le h_i(\mu\vf_{-i})$.
  As $\mu\vf_i = f_i(\mu\vf)$, we get by virtue of (g) that $h_i(\mu\vf_{-i}) \le \mu\vf_i$, too.
\end{itemize}
\end{prf}

With Proposition~\ref{prop:h} at hand, we now can show Lemma~\ref{lem:h-diff}:

\begin{qlemma}{\ref{lem:h-diff}}
$h_i$ is continuously differentiable with
\[
    \pdat{h_i}{X_j}{\vx_-i} = \frac{\pdat{f_i}{X_j}{\vx}}{-\pdat{q_i}{X_i}{\vx}}
    = \frac{\pdat{q_i}{X_j}{\vx}}{-\pdat{q_i}{X_i}{\vx}} \text{ for } \vx \in S_i \text{ and } j \neq i.
\]
In particular, $\pd{h_i}{X_j}$ is monotonically increasing with $\vx$.
\end{qlemma}
\begin{proof}
By Lemma~\ref{lem:normal} the implicit function theorem is applicable for every $\vx\in S_i$.
We therefore find for every $\vx\in S_i$ a local parametrization
$h_{\vx} : U \mapsto V$ with $h_{\vx}(\vx_{-i}) = x_i$.
Thus $h_{\vx}(\vx_{-i})$ is the least non-negative solution of $q_i(X_i,\vx_{-i}) = 0$.
By continuity of $q_i$ it is now easily shown that for all $\vy_{-i} \in U$ it has to hold
that $h_{\vx}(\vy_{-i})$ is also the least non-negative solution of $q_i(X_i,\vy_{-i})=0$ (see below).
By uniqueness we therefore have $h_{\vx} = h_i$ and that $h_i$ is continuously differentiable
for all $\vx_{-i} \in [\vzero,\mu\vf_{-i})$.

For every $\vx_{-i}\in [\vzero,\mu\vf_{-i})$ we can solve the (at most) quadratic equation
$q_i(X_i,\vx_{-i}) = 0$. We already know that $h_i(\vx_{-i})$ is the least non-negative solution of
this equation. So, if there exists another solution, it has to be real, too.

Assume first that this equation has two distinct solutions for some fixed $\vx_{-i}\in [\vzero,\mu\vf_{-i})$.
Solving $q_i(X_i,\vx_{-i})=0$ thus leads to an expression of the form
\[
  \frac{-b(\vx_{-i}) \pm \sqrt{b(\vx_{-i})^2 - 4a\cdot c(\vx_{-i})}}{2a}
\]
for the solutions where $b,c$ are (at most) quadratic polynomials in $\vX_{-i}$, $c$ having non-negative coefficients,
and $a$ is a positive constant (leading coefficient of $X_i^2$ in $q_i(\vX)$).
As $b$ and $c$ are continuous, the discriminant $b(\cdot)^2 - 4a\cdot c(\cdot)$ stays positive
for some open ball around $\vx_{-i}$ included inside of $U$ (it is positive in $\vx_{-i}$ as we assume
that we have two distinct solutions).
By making $U$ smaller, we may assume that $U$ is this open ball.
One of the two solutions must then be the least nonnegative solution.
As $h_{\vx}$ is the least non-negative solution for $\vx_{-i}$, and $h_{\vx}$ is continuous,
this also has to hold for some open ball centered at $\vx_{-i}$. W.l.o.g., $U$ is this ball.
So, $h_{\vx}$ and $h_i$ coincide on $U$.

We turn to the case that $q_i(X_i,\vx_{-i})=0$ has only a single solution, i.e.\ $h_i(\vx_{-i})$.
Note that $q_i(\vX)$ is linear in $X_i$ if and only if $q_i(X_i,\vx_{-i})$ is linear in $X_i$.
Obviously, if $q_i$ linear in $X_i$, then $h_i$ and $h_{\vx}$ coincide on $U$.
Thus, consider the case that $q_i(\vX)$ is quadratic in $X_i$, but $q_i(X_i,\vx_{-i})$ has only
a single solution.
This means that $\vx_{-i}$ is a root of the discriminant, i.e.\ $b(\vx_{-i}) - 4ac(\vx_{-i})=0$.
As $h_i(\vy_{-i})$ is a solution of $q_i(X_i,\vy_{-i})=0$ for all $\vy_{-i}\in U$,
the discriminant is non-negative on $U$.
If it equal to zero on $U$, then we again have that $h_i$ is equal to $h_{\vx}$ on $U$.
Therefore assume that is positive in some point of $U$.
As the discriminant is continuous, the solutions change continuously with $\vx_{-i}$.
But this implies that for some $\vy_{-i}\in U$ there are at least two $y_i,y_i^\ast\in V$
such that $(\vy_{-i},y_i)$ and $(\vy_{-i},y_i^\ast)$ are both located on the quadric $q_i(\vX)=0$.
But this contradicts the uniqueness of $h_{\vx}$ guaranteed by the implicit function theorem.

Assume now that $\vx \in S_i$. We then have
\[
    q_i( \vx ) = q_i(\vx_{-i}, h_i(\vx_{-i})) = 0,
\]
or equivalently
\[
  f_i(\vx_{-i},h_i(\vx_{-i})) = h_i(\vx_{-i}).
\]
Calculating the gradient of both in $\vx$ yields
\[
    \nfix \cdot \Jpiat{\vx_{-i}} = \nhiat{\vx_{-i}}.
\]
For the Jacobian of $\vp_i$ we obtain
\[
 \Jpiat{\vx_{-i}} = \begin{pmatrix} \ve_1^\top \\ \vdots \\ \ve_{i-1}^{\top} \\ \nhiat{\vx_{-i}} \\ \ve_{i+1}^\top \\ \vdots \\ \ve_{n}^{\top} \end{pmatrix}.
\]
This leads to
\[
  \pdat{f_i}{X_j}{\vx} + \pdat{f_i}{X_i}{\vx} \cdot \pdat{h_i}{X_j}{\vx_{-i}} = \pdat{h_i}{X_j}{\vx_{-i}}
\]
which solved for $\pd{h_i}{X_j}$ yields
\[
    \pdat{h_i}{X_j}{\vx_{-i}} = \frac{\pdat{f_i}{X_j}{\vx}}{-\pdat{q_i}{X_i}{\vx}}.
\]
As $\pdat{q_i}{X_i}{\vx} < 0$ and both $\pd{f_i}{X_j}$ and $\pd{q_i}{X_i}$ monotonically increase with $\vx$,
it follows that $\pd{h_i}{X_j}$ also monotonically increases with $\vx$. Finally, for $j\neq i$ we have
that $\pd{q_i}{X_j} = \pd{f_i}{X_j}$ as $q_i = f_i - X_i$.
\qed
\end{proof}

\subsection{Proof of Lemma~\ref{lem:upconv}}$ $\\
\begin{qlemma}{\ref{lem:upconv}}
For all $\vx\in S_i$ we have
\[
    \forall \vy \in S_i \cap [\vx,\mu\vf] : \nqix \cdot (\vy-\vx) \le 0.
\]
In particular
\[
\forall \vy \in S_i \cap [\vx,\mu\vf] : y_i \ge x_i + \sum_{j\neq i} \pdat{h_i}{X_j}{\vx_{-i}} \cdot (y_j - x_j).
\]
\end{qlemma}
\begin{prf}
Let $\vx\in S_i$, i.e.\ $f_i(\vx) = x_i$.
We want to show that
\[
    \nqix \cdot( \vy - \vx ) \le 0
\]
for all $\vy\in S_i \cap [\vx,\mu\vf)$.
As $f_i$ is quadratic in $\vX$, we may write
\[
\begin{array}{lcl}
    0 & = & q_i(\vy) \\
      & = & - y_i + f_i(\vy)\\
      & = & - y_i + \underbrace{f_i(\vx)}_{= x_i} + \nfix \cdot (\vy - \vx) + \underbrace{(\vy-\vx)^\top \cdot A \cdot (\vy-\vx)}_{\ge 0} \\
      & \ge & -y_i + x_i + \nfix \cdot (\vy-\vx)\\
      & =   & \nfix\cdot (\vy-\vx) - \ve_i^\top \cdot(\vy-\vx)\\
      & = & \nqix\cdot(\vy-\vx)
\end{array}
\]
where $A$ is a symmetric square-matrix with non-negative components such that the quadric terms of $f_i$ are given by
$\vX^\top A \vX$.

The second claim is easily obtained by solving this inequality for $y_i$ and recalling that by Lemma~\ref{lem:h-diff}
we have $\pdat{h_i}{X_j}{\vx_{-i}} = \frac{\pdat{q_i}{X_j}{\vp_i(\vx_{-i})}}{-\pdat{q_i}{X_i}{\vp_i(\vx_{-i})}}$
and $\pdat{q_i}{X_i}{\vp_i(\vx_{-i})} < 0$.
\end{prf}

\subsection{Proof of Proposition~\ref{prop:R}}$ $\\
\begin{qproposition}{\ref{prop:R}}
It holds
\[
    \vx \in R \Leftrightarrow \vx\in[\vzero,\mu\vf) \wedge \vq(\vx) \ge \vzero.
\]
\end{qproposition}
\begin{prf}
Let $\vx \in R$ and $i\in\{1,\ldots,n\}$.
Consider the function
\[
    g(t) := q_i(\vp_{i}(\vx_{-i}) + t \ve_i).
\]
As $q_i$ is a quadratic polynomial in $\vX$ there exists a symmetric square-matrix $A$ with non-negative entries, a vector $\vb$,
and a constant $c$ such that
\[
    q_i(\vX) = \vX^\top A \vX + \vb^\top \vX + c.
\]
It then follows that
\[
    q_i(\vX+\vY) = q_i(\vX) + \nqiat{\vX} \vY + \vY^\top A \vY.
\]
With $q_i(\vp_{i}(\vx_{-i}))=0$ this implies
\[
 g(t) = \nqiat{\vp_{i}(\vx_{-i})} t\ve_i + t^2 \underbrace{\ve_i^\top A \ve_i}_{:=a \ge 0} = t\cdot \left( \pdat{q_i}{X_i}{\vp_{i}(\vx_{-i})} + a\cdot t\right).
\]
As $\vp_i(\vx_{-i}) \prec \mu\vf$ ($\vf$ is strongly connected and $\vx\in [\vzero,\mu\vf)$), we know
that $\pdat{q_i}{X_i}{\vp_{i}(\vx_{-i})} < 0$. Thus, $g(t)$ has at most two zeros, one at $0$, the other
for some $t^\ast \ge 0$.

For the direction $(\Rightarrow)$ we only have to show that $x_i \le h_i(\vx_{-i})$ implies that $q_i(\vx) \ge 0$.
This now easily follows as $x_i \le h_i(\vx_{-i})$ implies that there is a $t' \le 0$ with $p_{i}(\vx_{-i}) + t\ve_i = \vx$.
But for this $t' \le 0$ we have $q_i(\vx) = g(t') \ge 0$.

Consider therefore the other direction $(\Leftarrow)$, that is $\vx\in [\vzero,\mu\vf)$ with $\vq(\vx) \ge \vzero$.
Assume that $\vx\not\in R$, i.e., for at least one $i$ we have $x_i > h_i(\vx_{-i})$.
As $q_i(\vx)\ge 0$ there has to be a $t'' > 0$ with $p_{i}(\vx_{-i}) + t''\ve_i = \vx$ and $g(t'') \ge 0$.
This implies that $a>0$ has to hold as otherwise $g(t)$ would be linear in $t$ and negative for $t>0$.
But then the second root $t^\ast$ of $g(t)$ has to be positive.
Set $\vx^\ast = p_i(\vx_{-i}) + t^\ast \ve_i$ with $q_i(\vx^\ast) = 0$, too.

A calculation similar to the one from above leads to
\[
  g(t+t^\ast) = q_i(\vx^\ast + t\ve_i) = t\cdot \left( \pdat{q_i}{X_i}{\vx^\ast} + a\cdot t\right).
\]
It follows that $\pdat{q_i}{X_i}{\vx^\ast}$ has to be greater than zero for $-t^\ast$ to be a root (as $a>0$).
But we have shown that $\pdat{q_i}{X_i}{\vx} < 0$ for all $\vx\in[\vzero,\mu\vf)$.
%
\end{prf}

\subsection{Proof of Lemma~\ref{lem:tangents-invertable}}$ $\\
\begin{qlemma}{\ref{lem:tangents-invertable}}
Let $\vf$ be a clean and feasible scSPP.
Let $\vx^{(1)},\ldots,\vx^{(n)}\in [\vzero,\mu\vf)$.
Then the matrix
\[
    \begin{pmatrix}
      \nqoat{\vx^{(1)}}\\
      \vdots\\
      \nqnat{\vx^{(n)}}
      \end{pmatrix}
\]
is regular, i.e., the vectors $\{ \nqiat{\vx^{(i)}} | i = 1,\ldots,n \}$ are linearly independent.
\end{qlemma}
\begin{prf}
%
%
%
Define $\vx\in [\vzero,\mu\vf)$ by setting
\[
        x_i := \max\{ x^{(j)}_i \mid j = 1,\ldots,n \}.
\]
We then have $\vx^{(i)} \le \vx$ for all $i$, and $\vx\prec \mu\vf$.
As mentioned above, we therefore have that $\Jqx$ is regular with
\[
    \Jqx^{-1} = - \sum_{k\in\N} \Jfx^k.
\]
%
As $\vx^{(i)} \le \vx$ it follows that
\[
    \begin{pmatrix}
    \nfoat{\vx^{(1)}} \\
    \vdots\\
    \nfnat{\vx^{(n)}}
    \end{pmatrix}
    \le
    \Jfx.
\]
Hence, we also have
\[
  \sum_{k=0}^l \begin{pmatrix}
    \nfoat{\vx^{(1)}} \\
    \vdots\\
    \nfnat{\vx^{(n)}}
    \end{pmatrix}^l \le \sum_{k=0}^l \Jfx
\]
implying that
\[
  \begin{pmatrix}
    \nfoat{\vx^{(1)}} \\
    \vdots\\
    \nfnat{\vx^{(n)}}
    \end{pmatrix}^{\ast}
    \text{ and, thus, }
    \begin{pmatrix}
    \nqoat{\vx^{(1)}} \\
    \vdots\\
    \nqnat{\vx^{(n)}}
    \end{pmatrix}^{-1}
    \text{ exist.}
\]
So, the vectors $\{ \nqoat{\vx^{(1)}}, \ldots, \nqnat{\vx^{(n)}} \}$ have to be linearly independent.
\end{prf}

\ifthenelse{\equal{0}{1}}{
\subsection{Proof of Theorem~\ref{thm:tangent-method}}$ $\\
\begin{qtheorem}{\ref{thm:tangent-method}}
We have
\[
 \vx \le \Ta(\vx) \text{ and } \Ne(\vx) \le \Ta(\vx) \le \mu\vf.
\]
Further, for $\vy\in R$ with $\vx \le \vy$
\[
  \Ta(\vx) \le \Ta(\vy).
\]
\end{qtheorem}
\begin{proof}
As $\vx \le \Ne(\vx)$ holds, we only need to show that $\Ne(\vx) \le \Ta(\vx) \le \mu\vf$. We refer the reader to the proof of Theorem~\ref{thm:tang-approx} for
this.

We turn to the monotonicity of $\Ta$.
Let $\vy \in R$ with $\vx \le \vy$.
Assume first that $\vx$ and $\vy$ are located on the surface $S_i$, i.e.\
\[
 h_i(\vx_{-i}) = x_i \text{ and } h_i(\vy_{-i}) = y_i.
\]
The tangent $T_i|_{\vx}$ at $S_i$ in $\vx$ is spanned by the partial derivatives of $\vp_i$
in $\vx$. The part $T_i|_{\vx}\cap [\vx,\mu\vf]$ relevant for $\Ta(\vx)$ can therefore be parameterized by
\[
    \vx + \sum_{j\neq i} \pdat{\vp_{i}}{X_j}{\vx} \cdot ( u_j - x_j ) \text{ with } \vu_{-i} \in [\vx_{-i},\mu\vf_{-i}].
\]
Similarly for $T_i|_{\vy}$.

In particular, for $\vu_{-i}\in [\vy_{-i},\mu\vf_{-i}]$ both points on the tangents defined by $\vu_{-i}$
differ only in the $i$th coordinate being (the remaining coordinates are simply $\vu_{-i}$)
\[
    t_{\vy} = y_i + \sum_{j\neq i} \pdat{h_i}{X_j}{\vy} \cdot ( u_j - y_j) \text{, resp. }
    t_{\vx} = x_i + \sum_{j\neq i} \pdat{h_i}{X_j}{\vx} \cdot ( u_j - x_j).
\]
By Lemma~\ref{lem:upconv} we have
\[
    y_i \ge x_i + \sum_{j\neq i} \pdat{h_i}{X_j}{\vx} \cdot ( y_j - x_j).
\]
From Lemma~\ref{lem:h-diff} it follows that $\pdat{h_i}{X_j}{\vy} \ge \pdat{h_i}{X_j}{\vx}$.
Thus $t_{\vy} \ge t_{\vx}$ immediately follows.

Now for $\vx,\vy\in R$ with $\vx\le \vy$ we can apply this result to the tangents at $S_i$ in $\vp_i(\vx_{-i})$,
resp. $\vp_i(\vy_{-i})$, and $\Ta(\vx) \le \Ta(\vy)$ follows.
\qed
\end{proof}


\begin{qtheorem}{\ref{thm:tang-approx}}
Let $\vx\in R$.
For $i=1,\ldots,n$ fix some $\eta_i \in [x_i,h_i(\vx_{-i})]$, and set $\veta = (\eta_1,\ldots,\eta_n)$.
We then have
\[
\vx \le \Ne(\vx) = \Ta_{\vx}(\vx) \le \Ta_{\veta}(\vx) \le \Ta_{(h_1(\vx_{-1}),\ldots,h_n(\vx_{-n}))}(\vx) = \Ta(\vx) \le \mu\vf
\]
\end{qtheorem}
\begin{prf}
Set
\[ \vpi_i := (\vx_{-i},\eta_i) \text{ and } \vh := (h_1(\vx_{-1}),\ldots, h_n(\vx_{-n})). \]
We first show that $\vx \le \Ta_{\veta}(\vx)$:
\[
\begin{array}{lcl}
\Ta_{\veta}(\vx)
& = & \bigl( \nqiat{\vpi_i} \bigr))^{-1}_{i=1,\ldots,n} \cdot \bigl( \nqiat{\vpi_i} \cdot \vpi_i - q_i(\vpi_i) \bigr)_{i=1,\ldots,n}\\[2mm]
& = & \bigl( \nfiat{\vpi_i} \bigr)^\ast_{i=1,\ldots,n} \cdot \bigl( - \nqiat{\vpi_i} \cdot \vpi_i + q_i(\vpi_i) \bigr)_{i=1,\ldots,n}\\[2mm]
& = & \bigl( \nfiat{\vpi_i} \bigr)^\ast_{i=1,\ldots,n} \cdot \bigl( - \nqiat{\vpi_i} \cdot \left(\vx + (\eta_i-x_i) \cdot \ve_i \right) + q_i(\vpi_i) \bigr)_{i=1,\ldots,n}\\[2mm]
& = & \underbrace{\bigl( \nfiat{\vpi_i} \bigr)^\ast_{i=1,\ldots,n}}_{\ge 0 \text{ in every comp.}}
      \cdot \bigl( - \nqiat{\vpi_i} \cdot \vx - \underbrace{\pdat{q_i}{X_i}{\vpi_i}}_{<0} \cdot  \underbrace{(\eta_i-x_i)}_{\ge 0} + \underbrace{q_i(\vpi_i)}_{\ge 0} \bigr)_{i=1,\ldots,n}\\[2mm]
& \ge & \vx.
\end{array}
\]
$\Ta_{\veta}(\vx)$ is by definition the (unique) solution of the equation system defined by
\[
 \nqiat{\vpi_{i}} ( \vX - \vpi_{i} ) = -q_i(\vpi_i) \quad (i=1,\ldots,n).
\]
As $\Ta_{\veta}(\vx) \ge \vx$ we can also consider this system with the origin of the coordinate system
moved into $\vx$, i.e.\
\[
    \nqiat{\vpi_{i}} ( \vX + \vx - \vpi_{i}) = -q_i(\vpi_i) \quad (i=1,\ldots,n).
\]
We show that this system is equivalent to an SPP. For this, we
solve these equations for $X_i$:
\[
 \begin{array}{cl}
                 & \nqiat{\vpi_{i}} ( \vX + \vx - \vpi_{i}) = -q_i(\vpi_i)\\[2mm]
 \Leftrightarrow & \nqiat{\vpi_i} \vX = -q_i(\vpi_i) + \nqiat{\vpi_i} \underbrace{( \vpi_i - \vx )}_{=(\eta_i-x_i)\cdot \ve_i}\\[2mm]
 \Leftrightarrow & X_i = \sum_{j\neq i} \frac{\pdat{q_i}{X_j}{\vpi_i}}{-\pdat{q_i}{X_i}{\vpi_i}} \cdot X_j + \frac{q_i(\vpi_i)}{-\pdat{q_i}{X_i}{\vpi_i}} + (\eta_i-x_i).
 \end{array}
\]
Again, we have $\pdat{q_i}{X_i}{\vpi_i} < 0 \le
\pdat{q_i}{X_j}{\vpi_i}$ as $\vpi_i\in R$, and $\nqiat{\vpi_i}$
monotonically increases with $\eta_i$. Hence, the above linear
equation for $X_i$ is indeed a polynomial with non-negative
coefficients. Denote by $\vf_{\veta}$ the SPP defined by these
linear equations. We then have $\mu \vf_{\veta} =
\Ta_{\veta}(\vx)-\vx$ as the above equation system has
$\Ta_{\veta}(\vx)-\vx \ge \vzero$ as its unique solution. Further,
we know that the Kleene sequence
$\bigl(\vf_{\veta}^k(\vzero)\bigr)_{k\in\N}$ converges to
$\mu\vf_{\veta}$. We show that all coefficients of $\vf_{\veta}$
increase with $\veta \to \vh$. This is straight-forward for
\[
\frac{\pdat{q_i}{X_j}{\vpi_i}}{-\pdat{q_i}{X_i}{\vpi_i}}
\]
as $\pdat{q_i}{X_i}{\vpi_i} < 0 \le \pdat{q_i}{X_j}{\vpi_i}$, and all these terms increase with $\eta_i \to h_i(\vx_{-i})$.
Consider therefore
\[
  0\ge \frac{q_i(\vpi_i)}{-\pdat{q_i}{X_i}{\vpi_i}} + (\eta_i-x_i) = \frac{q_i(\vpi_i) - \pdat{q_i}{X_i}{\vpi_i} (\eta_i - x_i)}{-\pdat{q_i}{X_i}{\vpi_i}}.
\]
We show that this term increases with $\eta_i$. Set $\delta_i := \eta_i - x_i$.
We can find a non-negative, symmetric square-matrix $A$, a vector $\vb$, and constant $c$ such that
\[
 q_i(\vX) = \vX^\top A \vX + \vb^\top \vX + c \text{ and } \nqiat{\vX} = 2 \vX^\top A + \vb^\top.
\]
As $\vpi_i = \vx + \delta_i \ve_i$ we have
\[
    q_i(\vpi_i) = q_i(\vx + \delta_i \ve_i ) = q_i(\vx) + \pdat{q_i}{X_i}{\vx} \delta_i + \delta_i^2 A_{ii},
\]
and
\[
 \pdat{q_i}{X_i}{\vpi_i} \cdot \delta_i = \nqiat{\vx + \delta_i \ve_i} \delta_i\ve_i = \pdat{q_i}{X_i}{\vx} \delta_i + 2
 \delta_i^2 A_{ii}.
\]
This leads to
\[
\frac{q_i(\vpi_i) - \pdat{q_i}{X_i}{\vpi_i} \delta_i}{-\pdat{q_i}{X_i}{\vpi_i}} = \frac{ q_i(\vx) - \delta_i^2 A_{ii}}{-\pdat{q_i}{X_i}{\vx} - 2 \delta_i A_{ii}}.
\]
Deriving this w.r.t.~$\delta_i$ yields:
\[
\begin{array}{cl}
  & \frac{-2 A_{ii} \delta_i}{-\pdat{q_i}{X_i}{\vx} + 2 A_{ii} \delta_i} - \frac{q_i(\vx) - A_{ii} \delta_i^2}{( -\pdat{q_i}{X_i}{\vx} - 2 A_{ii} \delta_i)^2} (-2 A_{ii} )\\[2mm]
= & \frac{2 A_{ii} \pdat{q_i}{X_i}{\vx} \delta_i + 4 A_{ii}^2 \delta_i^2 + 2A_{ii} q_i(\vx) - 2 A_{ii}^2 \delta_i^2}{(-\pdat{q_i}{X_i}{\vx} - 2 A_{ii} \delta_i)^2}\\[2mm]
= & 2 A_{ii} \frac{ A_{ii} \delta_i^2 + \pdat{q_i}{X_i}{\vx} \delta_i + q_i(\vx)}{(-\pdat{q_i}{X_i}{\vx} - 2 A_{ii} \delta_i)^2}\\[2mm]
= & 2 A_{ii} \frac{ q_i(\vpi_i)}{(-\pdat{q_i}{X_i}{\vpi_i})^2}.
\end{array}
\]
As $q_i(\vpi_i) \ge 0$ and $A_{ii} \ge 0$, it follows that
\[
  \frac{q_i(\vpi_i)}{-\pdat{q_i}{X_i}{\vpi_i}} + (\eta_i-x_i)
\]
increases with $\eta_i \to h_i(\vx_{-i})$.
Thus, all coefficients of $\vf_{\veta}$ increase with $\eta_i \to h_i(\vx_{-i})$,
and so for any $\veta'\in [\veta, \vh]$ it follows that
\[
 \vf_{\veta}(\vy) \le \vf_{\veta'}(\vy) \text{ for all } \vy \ge \vzero,
\]
and
\[
 \Ta_{\veta}(\vx)-\vx = \mu\vf_{\veta} \le \mu\vf_{\veta'} = \Ta_{\veta'}(\vx) - \vx.
\]
As $\Ne(\vX) = \Ta_{\vx}(\vX)$ and $\Ta(\vX) = \Ta_{\vh}(\vX)$ we may therefore conclude that
\[
    \Ne(\vx) \le \Ta_{\veta}(\vx) \le \Ta_{\veta'}(\vx) \le \Ta(\vx).
\]
It remains to show that $\Ta(\vx) \le \mu\vf$. This is equivalent to showing
that $\mu \vf_{\vh} \le \mu \vf -\vx$. For $\vf_{\vh}(\vX)$ we have by definition and Lemma~\ref{lem:h-diff}
\[
  \bigl(\vf_{\vh}(\vX)\bigr)_i = \sum_{j\neq i} \frac{\pdat{q_i}{X_j}{\vp_{i}(\vx_{-i})}}{-\pdat{q_i}{X_i}{\vp_{i}(\vx_{-i})}} X_j + (h_i(\vx_{-i})-x_i)
  = \sum_{j\neq i} \pdat{h_i}{X_j}{\vx_{-i}} X_j + (h_i(\vx_{-i})-x_i).
\]
By virtue of Lemma~\ref{lem:upconv} it follows that $\mu\vf$ is above all the tangents, i.e.\
\[
 \vf_{\vh}(\mu\vf - \vx ) \le \mu\vf - \vx.
\]
By monotonicity of $\vf_{\vh}$ we also have
\[
    \vf_{\vh}(\vzero) \le \vf_{\vh}(\mu\vf-\vx).
\]
A straight-forward induction therefore shows that
\[
 \vf^k_{\vh}(\vzero) \le \mu\vf - \vx \quad (\forall k\in \N),
\]
and, thus,
\[
    \Ta(\vx) - \vx = \mu\vf_{\vh} \le \mu\vf -\vx.
\]
\end{prf}
}{ 
\ifthenelse{\equal{0}{1}}{
\subsection{Proof of Theorem~\ref{thm:tangent-method}}$ $\\
We want to show the following Theorem:
\begin{qtheorem}{\ref{thm:tangent-method}}
Let $\vf$ be a clean and feasible scSPP.
We have
\[
 \vx \le \Ta(\vx) \text{ and } \Ne(\vx) \le \Ta(\vx) \le \mu\vf.
\]
Further, for $\vy\in R$ with $\vx \le \vy$
\[
  \Ta(\vx) \le \Ta(\vy).
\]
\end{qtheorem}
We split the proof of this theorem into one part concerning the claim that the operator $\Ta$ improves a given $\vx\in R$ at least as much as the Newton operator $\Ne$ does, and another part concerning the claim that the operator $\Ta$ is monotone on $R$.
We start with the monotonicity claim:
\begin{proof} (Monotonicity of $\Ta$.)
%
Let $\vy \in R$ with $\vx \le \vy$.
Assume first that $\vx$ and $\vy$ are located on the surface $S_i$, i.e.\
\[
 h_i(\vx_{-i}) = x_i \text{ and } h_i(\vy_{-i}) = y_i.
\]
The tangent $T_i|_{\vx}$ at $S_i$ in $\vx$ is spanned by the partial derivatives of $\vp_i$
in $\vx$. The part $T_i|_{\vx}\cap [\vx,\mu\vf]$ relevant for $\Ta(\vx)$ can therefore be parameterized by
\[
    \vx + \sum_{j\neq i} \pdat{\vp_{i}}{X_j}{\vx} \cdot ( u_j - x_j ) \text{ with } \vu_{-i} \in [\vx_{-i},\mu\vf_{-i}].
\]
Similarly for $T_i|_{\vy}$.

In particular, for $\vu_{-i}\in [\vy_{-i},\mu\vf_{-i}]$ both points on the tangents defined by $\vu_{-i}$
differ only in the $i$th coordinate being (the remaining coordinates are simply $\vu_{-i}$)
\[
    t_{\vy} = y_i + \sum_{j\neq i} \pdat{h_i}{X_j}{\vy} \cdot ( u_j - y_j) \text{, resp. }
    t_{\vx} = x_i + \sum_{j\neq i} \pdat{h_i}{X_j}{\vx} \cdot ( u_j - x_j).
\]
By Lemma~\ref{lem:upconv} we have
\[
    y_i \ge x_i + \sum_{j\neq i} \pdat{h_i}{X_j}{\vx} \cdot ( y_j - x_j).
\]
From Lemma~\ref{lem:h-diff} it follows that $\pdat{h_i}{X_j}{\vy} \ge \pdat{h_i}{X_j}{\vx}$.
Thus $t_{\vy} \ge t_{\vx}$ immediately follows.

Now for $\vx,\vy\in R$ with $\vx\le \vy$ we can apply this result to the tangents at $S_i$ in $\vp_i(\vx_{-i})$,
resp.\ $\vp_i(\vy_{-i})$, and $\Ta(\vx) \le \Ta(\vy)$ follows.
\qed
\end{proof}

We turn to the proof of the first paart of Theorem~\ref{thm:tangent-method},
stating that $\Ne(\vx)\le \Ta(\vx)$ for $\vx\in R$. For technical reasons,
we first introduce a generalization of the operator $\Ta$
where we do not require to know the exact values $h_i(\vx_{-i})$, but only some under-approximation.
\begin{definition}\label{def:tang-approx}
Let $\vx\in R$.
For $i=1,\ldots,n$ fix some $\eta_i \in [x_i,h_i(\vx_{-i})]$, and set $\veta = (\eta_1,\ldots,\eta_n)$.
We then let $\Ta_{\veta}(\vx)$ denote the solution of
\[
 \nqiat{(\vx_{-i},\eta_i)}(\vX - (\vx_{-i},\eta_i)) = - q_i((\vx_{-i},\eta_i)) \quad (i=1,\ldots,n).
\]
\end{definition}
Note that we obtain the Newton operator $\Ne$ be setting $\veta:=\vx$. Similarly, we obtain the operator $\Ta$ by setting $\veta_i = h_i(\vx_{-i})$. We now show for the operator $\Ta_{\veta}$ a generalization of the first claim stated in Theorem~\ref{thm:tangent-method} thereby completing the proof of Theorem~\ref{thm:tangent-method}.
\begin{theorem}\label{thm:tang-approx}
Let $\vf$ be a clean and feasible scSPP.
Let $\vx\in R$.
For $i=1,\ldots,n$ fix some $\eta_i \in [x_i,h_i(\vx_{-i})]$, and set $\veta = (\eta_1,\ldots,\eta_n)$.
We then have
\[
\vx \le \Ne(\vx) \le \Ta_{\veta}(\vx) \le \Ta(\vx) \le \mu\vf
\]
\end{theorem}\begin{proof}
Set
\[ \vpi_i := (\vx_{-i},\eta_i) \text{ and } \vh := (h_1(\vx_{-1}),\ldots, h_n(\vx_{-n})). \]
We first show that $\vx \le \Ta_{\veta}(\vx)$:
\[
\begin{array}{lcl}
\Ta_{\veta}(\vx)
& = & \bigl( \nqiat{\vpi_i} \bigr))^{-1}_{i=1,\ldots,n} \cdot \bigl( \nqiat{\vpi_i} \cdot \vpi_i - q_i(\vpi_i) \bigr)_{i=1,\ldots,n}\\[2mm]
& = & \bigl( \nfiat{\vpi_i} \bigr)^\ast_{i=1,\ldots,n} \cdot \bigl( - \nqiat{\vpi_i} \cdot \vpi_i + q_i(\vpi_i) \bigr)_{i=1,\ldots,n}\\[2mm]
& = & \bigl( \nfiat{\vpi_i} \bigr)^\ast_{i=1,\ldots,n} \cdot \bigl( - \nqiat{\vpi_i} \cdot \left(\vx + (\eta_i-x_i) \cdot \ve_i \right) + q_i(\vpi_i) \bigr)_{i=1,\ldots,n}\\[2mm]
& = & \underbrace{\bigl( \nfiat{\vpi_i} \bigr)^\ast_{i=1,\ldots,n}}_{\ge 0 \text{ in every comp.}}
      \cdot \bigl( - \nqiat{\vpi_i} \cdot \vx - \underbrace{\pdat{q_i}{X_i}{\vpi_i}}_{<0} \cdot  \underbrace{(\eta_i-x_i)}_{\ge 0} + \underbrace{q_i(\vpi_i)}_{\ge 0} \bigr)_{i=1,\ldots,n}\\[2mm]
& \ge & \vx.
\end{array}
\]
$\Ta_{\veta}(\vx)$ is by definition the (unique) solution of the equation system defined by
\[
 \nqiat{\vpi_{i}} ( \vX - \vpi_{i} ) = -q_i(\vpi_i) \quad (i=1,\ldots,n).
\]
As $\Ta_{\veta}(\vx) \ge \vx$ we can also consider this system with the origin of the coordinate system
moved into $\vx$, i.e.\
\[
    \nqiat{\vpi_{i}} ( \vX + \vx - \vpi_{i}) = -q_i(\vpi_i) \quad (i=1,\ldots,n).
\]
We show that this system is equivalent to an SPP. For this, we
solve these equations for $X_i$:
\[
 \begin{array}{cl}
                 & \nqiat{\vpi_{i}} ( \vX + \vx - \vpi_{i}) = -q_i(\vpi_i)\\[2mm]
 \Leftrightarrow & \nqiat{\vpi_i} \vX = -q_i(\vpi_i) + \nqiat{\vpi_i} \underbrace{( \vpi_i - \vx )}_{=(\eta_i-x_i)\cdot \ve_i}\\[2mm]
 \Leftrightarrow & X_i = \sum_{j\neq i} \frac{\pdat{q_i}{X_j}{\vpi_i}}{-\pdat{q_i}{X_i}{\vpi_i}} \cdot X_j + \frac{q_i(\vpi_i)}{-\pdat{q_i}{X_i}{\vpi_i}} + (\eta_i-x_i).
 \end{array}
\]
Again, we have $\pdat{q_i}{X_i}{\vpi_i} < 0 \le
\pdat{q_i}{X_j}{\vpi_i}$ as $\vpi_i\in R$, and $\nqiat{\vpi_i}$
monotonically increases with $\eta_i$. Hence, the above linear
equation for $X_i$ is indeed a polynomial with non-negative
coefficients. Denote by $\vf_{\veta}$ the SPP defined by these
linear equations. We then have $\mu \vf_{\veta} =
\Ta_{\veta}(\vx)-\vx$ as the above equation system has
$\Ta_{\veta}(\vx)-\vx \ge \vzero$ as its unique solution. Further,
we know that the Kleene sequence
$\bigl(\vf_{\veta}^k(\vzero)\bigr)_{k\in\N}$ converges to
$\mu\vf_{\veta}$. We show that all coefficients of $\vf_{\veta}$
increase with $\veta \to \vh$. This is straight-forward for
\[
\frac{\pdat{q_i}{X_j}{\vpi_i}}{-\pdat{q_i}{X_i}{\vpi_i}}
\]
as $\pdat{q_i}{X_i}{\vpi_i} < 0 \le \pdat{q_i}{X_j}{\vpi_i}$, and all these terms increase with $\eta_i \to h_i(\vx_{-i})$.
Consider therefore
\[
  0\ge \frac{q_i(\vpi_i)}{-\pdat{q_i}{X_i}{\vpi_i}} + (\eta_i-x_i) = \frac{q_i(\vpi_i) - \pdat{q_i}{X_i}{\vpi_i} (\eta_i - x_i)}{-\pdat{q_i}{X_i}{\vpi_i}}.
\]
We show that this term increases with $\eta_i$. Set $\delta_i := \eta_i - x_i$.
We can find a non-negative, symmetric square-matrix $A$, a vector $\vb$, and constant $c$ such that
\[
 q_i(\vX) = \vX^\top A \vX + \vb^\top \vX + c \text{ and } \nqiat{\vX} = 2 \vX^\top A + \vb^\top.
\]
As $\vpi_i = \vx + \delta_i \ve_i$ we have
\[
    q_i(\vpi_i) = q_i(\vx + \delta_i \ve_i ) = q_i(\vx) + \pdat{q_i}{X_i}{\vx} \delta_i + \delta_i^2 A_{ii},
\]
and
\[
 \pdat{q_i}{X_i}{\vpi_i} \cdot \delta_i = \nqiat{\vx + \delta_i \ve_i} \delta_i\ve_i = \pdat{q_i}{X_i}{\vx} \delta_i + 2
 \delta_i^2 A_{ii}.
\]
This leads to
\[
\frac{q_i(\vpi_i) - \pdat{q_i}{X_i}{\vpi_i} \delta_i}{-\pdat{q_i}{X_i}{\vpi_i}} = \frac{ q_i(\vx) - \delta_i^2 A_{ii}}{-\pdat{q_i}{X_i}{\vx} - 2 \delta_i A_{ii}}.
\]
Taking the derivative w.r.t.~$\delta_i$ yields:
\[
\begin{array}{cl}
  & \frac{-2 A_{ii} \delta_i}{-\pdat{q_i}{X_i}{\vx} + 2 A_{ii} \delta_i} - \frac{q_i(\vx) - A_{ii} \delta_i^2}{( -\pdat{q_i}{X_i}{\vx} - 2 A_{ii} \delta_i)^2} (-2 A_{ii} )\\[2mm]
= & \frac{2 A_{ii} \pdat{q_i}{X_i}{\vx} \delta_i + 4 A_{ii}^2 \delta_i^2 + 2A_{ii} q_i(\vx) - 2 A_{ii}^2 \delta_i^2}{(-\pdat{q_i}{X_i}{\vx} - 2 A_{ii} \delta_i)^2}\\[2mm]
= & 2 A_{ii} \frac{ A_{ii} \delta_i^2 + \pdat{q_i}{X_i}{\vx} \delta_i + q_i(\vx)}{(-\pdat{q_i}{X_i}{\vx} - 2 A_{ii} \delta_i)^2}\\[2mm]
= & 2 A_{ii} \frac{ q_i(\vpi_i)}{(-\pdat{q_i}{X_i}{\vpi_i})^2}.
\end{array}
\]
As $q_i(\vpi_i) \ge 0$ and $A_{ii} \ge 0$, it follows that
\[
  \frac{q_i(\vpi_i)}{-\pdat{q_i}{X_i}{\vpi_i}} + (\eta_i-x_i)
\]
increases with $\eta_i \to h_i(\vx_{-i})$.
Thus, all coefficients of $\vf_{\veta}$ increase with $\eta_i \to h_i(\vx_{-i})$,
and so for any $\veta'\in [\veta, \vh]$ it follows that
\[
 \vf_{\veta}(\vy) \le \vf_{\veta'}(\vy) \text{ for all } \vy \ge \vzero,
\]
and
\[
 \Ta_{\veta}(\vx)-\vx = \mu\vf_{\veta} \le \mu\vf_{\veta'} = \Ta_{\veta'}(\vx) - \vx.
\]
As $\Ne(\vX) = \Ta_{\vx}(\vX)$ and $\Ta(\vX) = \Ta_{\vh}(\vX)$ we may therefore conclude that
\[
    \Ne(\vx) \le \Ta_{\veta}(\vx) \le \Ta_{\veta'}(\vx) \le \Ta(\vx).
\]
It remains to show that $\Ta(\vx) \le \mu\vf$. This is equivalent to showing
that $\mu \vf_{\vh} \le \mu \vf -\vx$. For $\vf_{\vh}(\vX)$ we have by definition and Lemma~\ref{lem:h-diff}
\[
  \bigl(\vf_{\vh}(\vX)\bigr)_i = \sum_{j\neq i} \frac{\pdat{q_i}{X_j}{\vp_{i}(\vx_{-i})}}{-\pdat{q_i}{X_i}{\vp_{i}(\vx_{-i})}} X_j + (h_i(\vx_{-i})-x_i)
  = \sum_{j\neq i} \pdat{h_i}{X_j}{\vx_{-i}} X_j + (h_i(\vx_{-i})-x_i).
\]
By virtue of Lemma~\ref{lem:upconv} it follows that $\mu\vf$ is above all the tangents, i.e.\
\[
 \vf_{\vh}(\mu\vf - \vx ) \le \mu\vf - \vx.
\]
By monotonicity of $\vf_{\vh}$ we also have
\[
    \vf_{\vh}(\vzero) \le \vf_{\vh}(\mu\vf-\vx).
\]
A straight-forward induction therefore shows that
\[
 \vf^k_{\vh}(\vzero) \le \mu\vf - \vx \quad (\forall k\in \N),
\]
and, thus,
\[
    \Ta(\vx) - \vx = \mu\vf_{\vh} \le \mu\vf -\vx.
\]
\end{proof}
}
}{
\subsection{Proof of Theorem~\ref{thm:tang-approx}}$ $\\
\begin{qtheorem}{\ref{thm:tang-approx}}
Let $\vf$ be a clean and feasible scSPP.
Let $\vx\in R$.
For $i=1,\ldots,n$ fix some $\eta_i \in [x_i,h_i(\vx_{-i})]$, and set $\veta = (\eta_1,\ldots,\eta_n)$.
We then have
\[
\vx \le \Ne(\vx) \le \Ta_{\veta}(\vx) \le \Ta(\vx) \le \mu\vf
\]
Further, the operator $\Ta$ is monotone on $R$, i.e., for any $\vy\in R$ with $\vx\le \vy$ it holds that $\Ta(\vx)\le \Ta(\vy)$.
\end{qtheorem}
\begin{proof}
Set
\[ \vpi_i := (\vx_{-i},\eta_i) \text{ and } \vh := (h_1(\vx_{-1}),\ldots, h_n(\vx_{-n})). \]
We first show that $\vx \le \Ta_{\veta}(\vx)$:
\[
\begin{array}{lcl}
\Ta_{\veta}(\vx)
& = & \bigl( \nqiat{\vpi_i} \bigr))^{-1}_{i=1,\ldots,n} \cdot \bigl( \nqiat{\vpi_i} \cdot \vpi_i - q_i(\vpi_i) \bigr)_{i=1,\ldots,n}\\[2mm]
& = & \bigl( \nfiat{\vpi_i} \bigr)^\ast_{i=1,\ldots,n} \cdot \bigl( - \nqiat{\vpi_i} \cdot \vpi_i + q_i(\vpi_i) \bigr)_{i=1,\ldots,n}\\[2mm]
& = & \bigl( \nfiat{\vpi_i} \bigr)^\ast_{i=1,\ldots,n} \cdot \bigl( - \nqiat{\vpi_i} \cdot \left(\vx + (\eta_i-x_i) \cdot \ve_i \right) + q_i(\vpi_i) \bigr)_{i=1,\ldots,n}\\[2mm]
& = & \underbrace{\bigl( \nfiat{\vpi_i} \bigr)^\ast_{i=1,\ldots,n}}_{\ge 0 \text{ in every comp.}}
      \cdot \bigl( - \nqiat{\vpi_i} \cdot \vx - \underbrace{\pdat{q_i}{X_i}{\vpi_i}}_{<0} \cdot  \underbrace{(\eta_i-x_i)}_{\ge 0} + \underbrace{q_i(\vpi_i)}_{\ge 0} \bigr)_{i=1,\ldots,n}\\[2mm]
& \ge & \vx.
\end{array}
\]
$\Ta_{\veta}(\vx)$ is by definition the (unique) solution of the equation system defined by
\[
 \nqiat{\vpi_{i}} ( \vX - \vpi_{i} ) = -q_i(\vpi_i) \quad (i=1,\ldots,n).
\]
As $\Ta_{\veta}(\vx) \ge \vx$ we can also consider this system with the origin of the coordinate system
moved into $\vx$, i.e.\
\[
    \nqiat{\vpi_{i}} ( \vX + \vx - \vpi_{i}) = -q_i(\vpi_i) \quad (i=1,\ldots,n).
\]
We show that this system is equivalent to an SPP. For this, we
solve these equations for $X_i$:
\[
 \begin{array}{cl}
                 & \nqiat{\vpi_{i}} ( \vX + \vx - \vpi_{i}) = -q_i(\vpi_i)\\[2mm]
 \Leftrightarrow & \nqiat{\vpi_i} \vX = -q_i(\vpi_i) + \nqiat{\vpi_i} \underbrace{( \vpi_i - \vx )}_{=(\eta_i-x_i)\cdot \ve_i}\\[2mm]
 \Leftrightarrow & X_i = \sum_{j\neq i} \frac{\pdat{q_i}{X_j}{\vpi_i}}{-\pdat{q_i}{X_i}{\vpi_i}} \cdot X_j + \frac{q_i(\vpi_i)}{-\pdat{q_i}{X_i}{\vpi_i}} + (\eta_i-x_i).
 \end{array}
\]
Again, we have $\pdat{q_i}{X_i}{\vpi_i} < 0 \le
\pdat{q_i}{X_j}{\vpi_i}$ as $\vpi_i\in R$, and $\nqiat{\vpi_i}$
monotonically increases with $\eta_i$. Hence, the above linear
equation for $X_i$ is indeed a polynomial with non-negative
coefficients. Denote by $\vf_{\veta}$ the SPP defined by these
linear equations. We then have $\mu \vf_{\veta} =
\Ta_{\veta}(\vx)-\vx$ as the above equation system has
$\Ta_{\veta}(\vx)-\vx \ge \vzero$ as its unique solution. Further,
we know that the Kleene sequence
$\bigl(\vf_{\veta}^k(\vzero)\bigr)_{k\in\N}$ converges to
$\mu\vf_{\veta}$. We show that all coefficients of $\vf_{\veta}$
increase with $\veta \to \vh$. This is straight-forward for
\[
\frac{\pdat{q_i}{X_j}{\vpi_i}}{-\pdat{q_i}{X_i}{\vpi_i}}
\]
as $\pdat{q_i}{X_i}{\vpi_i} < 0 \le \pdat{q_i}{X_j}{\vpi_i}$, and all these terms increase with $\eta_i \to h_i(\vx_{-i})$.
Consider therefore
\[
  0\ge \frac{q_i(\vpi_i)}{-\pdat{q_i}{X_i}{\vpi_i}} + (\eta_i-x_i) = \frac{q_i(\vpi_i) - \pdat{q_i}{X_i}{\vpi_i} (\eta_i - x_i)}{-\pdat{q_i}{X_i}{\vpi_i}}.
\]
We show that this term increases with $\eta_i$. Set $\delta_i := \eta_i - x_i$.
We can find a non-negative, symmetric square-matrix $A$, a vector $\vb$, and constant $c$ such that
\[
 q_i(\vX) = \vX^\top A \vX + \vb^\top \vX + c \text{ and } \nqiat{\vX} = 2 \vX^\top A + \vb^\top.
\]
As $\vpi_i = \vx + \delta_i \ve_i$ we have
\[
    q_i(\vpi_i) = q_i(\vx + \delta_i \ve_i ) = q_i(\vx) + \pdat{q_i}{X_i}{\vx} \delta_i + \delta_i^2 A_{ii},
\]
and
\[
 \pdat{q_i}{X_i}{\vpi_i} \cdot \delta_i = \nqiat{\vx + \delta_i \ve_i} \delta_i\ve_i = \pdat{q_i}{X_i}{\vx} \delta_i + 2
 \delta_i^2 A_{ii}.
\]
This leads to
\[
\frac{q_i(\vpi_i) - \pdat{q_i}{X_i}{\vpi_i} \delta_i}{-\pdat{q_i}{X_i}{\vpi_i}} = \frac{ q_i(\vx) - \delta_i^2 A_{ii}}{-\pdat{q_i}{X_i}{\vx} - 2 \delta_i A_{ii}}.
\]
Taking the derivative w.r.t.~$\delta_i$ yields:
\[
\begin{array}{cl}
  & \frac{-2 A_{ii} \delta_i}{-\pdat{q_i}{X_i}{\vx} + 2 A_{ii} \delta_i} - \frac{q_i(\vx) - A_{ii} \delta_i^2}{( -\pdat{q_i}{X_i}{\vx} - 2 A_{ii} \delta_i)^2} (-2 A_{ii} )\\[2mm]
= & \frac{2 A_{ii} \pdat{q_i}{X_i}{\vx} \delta_i + 4 A_{ii}^2 \delta_i^2 + 2A_{ii} q_i(\vx) - 2 A_{ii}^2 \delta_i^2}{(-\pdat{q_i}{X_i}{\vx} - 2 A_{ii} \delta_i)^2}\\[2mm]
= & 2 A_{ii} \frac{ A_{ii} \delta_i^2 + \pdat{q_i}{X_i}{\vx} \delta_i + q_i(\vx)}{(-\pdat{q_i}{X_i}{\vx} - 2 A_{ii} \delta_i)^2}\\[2mm]
= & 2 A_{ii} \frac{ q_i(\vpi_i)}{(-\pdat{q_i}{X_i}{\vpi_i})^2}.
\end{array}
\]
As $q_i(\vpi_i) \ge 0$ and $A_{ii} \ge 0$, it follows that
\[
  \frac{q_i(\vpi_i)}{-\pdat{q_i}{X_i}{\vpi_i}} + (\eta_i-x_i)
\]
increases with $\eta_i \to h_i(\vx_{-i})$.
Thus, all coefficients of $\vf_{\veta}$ increase with $\eta_i \to h_i(\vx_{-i})$,
and so for any $\veta'\in [\veta, \vh]$ it follows that
\[
 \vf_{\veta}(\vy) \le \vf_{\veta'}(\vy) \text{ for all } \vy \ge \vzero,
\]
and
\[
 \Ta_{\veta}(\vx)-\vx = \mu\vf_{\veta} \le \mu\vf_{\veta'} = \Ta_{\veta'}(\vx) - \vx.
\]
As $\Ne(\vX) = \Ta_{\vx}(\vX)$ and $\Ta(\vX) = \Ta_{\vh}(\vX)$ we may therefore conclude that
\[
    \Ne(\vx) \le \Ta_{\veta}(\vx) \le \Ta_{\veta'}(\vx) \le \Ta(\vx).
\]
It remains to show that $\Ta(\vx) \le \mu\vf$. This is equivalent to showing
that $\mu \vf_{\vh} \le \mu \vf -\vx$. For $\vf_{\vh}(\vX)$ we have by definition and Lemma~\ref{lem:h-diff}
\[
  \bigl(\vf_{\vh}(\vX)\bigr)_i = \sum_{j\neq i} \frac{\pdat{q_i}{X_j}{\vp_{i}(\vx_{-i})}}{-\pdat{q_i}{X_i}{\vp_{i}(\vx_{-i})}} X_j + (h_i(\vx_{-i})-x_i)
  = \sum_{j\neq i} \pdat{h_i}{X_j}{\vx_{-i}} X_j + (h_i(\vx_{-i})-x_i).
\]
By virtue of Lemma~\ref{lem:upconv} it follows that $\mu\vf$ is above all the tangents, i.e.\
\[
 \vf_{\vh}(\mu\vf - \vx ) \le \mu\vf - \vx.
\]
By monotonicity of $\vf_{\vh}$ we also have
\[
    \vf_{\vh}(\vzero) \le \vf_{\vh}(\mu\vf-\vx).
\]
A straight-forward induction therefore shows that
\[
 \vf^k_{\vh}(\vzero) \le \mu\vf - \vx \quad (\forall k\in \N),
\]
and, thus,
\[
    \Ta(\vx) - \vx = \mu\vf_{\vh} \le \mu\vf -\vx.
\]
We turn to the monotonicity of $\Ta$.
Let $\vy \in R$ with $\vx \le \vy$.
Assume that $\vx$ and $\vy$ are located on the surface $S_i$, i.e.\
\[
 h_i(\vx_{-i}) = x_i \text{ and } h_i(\vy_{-i}) = y_i.
\]
The tangent $T_i|_{\vx}$ at $S_i$ in $\vx$ is spanned by the partial derivatives of $\vp_i$
in $\vx$. The part $T_i|_{\vx}\cap [\vx,\mu\vf]$ relevant for $\Ta(\vx)$ can therefore be parameterized by
\[
    \vx + \sum_{j\neq i} \pdat{\vp_{i}}{X_j}{\vx} \cdot ( u_j - x_j ) \text{ with } \vu_{-i} \in [\vx_{-i},\mu\vf_{-i}].
\]
Similarly for $T_i|_{\vy}$.

In particular, for $\vu_{-i}\in [\vy_{-i},\mu\vf_{-i}]$ both points on the tangents defined by $\vu_{-i}$
differ only in the $i$th coordinate being (the remaining coordinates are simply $\vu_{-i}$)
\[
    t_{\vy} = y_i + \sum_{j\neq i} \pdat{h_i}{X_j}{\vy} \cdot ( u_j - y_j) \text{, resp. }
    t_{\vx} = x_i + \sum_{j\neq i} \pdat{h_i}{X_j}{\vx} \cdot ( u_j - x_j).
\]
By Lemma~\ref{lem:upconv} we have
\[
    y_i \ge x_i + \sum_{j\neq i} \pdat{h_i}{X_j}{\vx} \cdot ( y_j - x_j).
\]
From Lemma~\ref{lem:h-diff} it follows that $\pdat{h_i}{X_j}{\vy} \ge \pdat{h_i}{X_j}{\vx}$.
Thus $t_{\vy} \ge t_{\vx}$ immediately follows.

Now for $\vx,\vy\in R$ with $\vx\le \vy$ we can apply this result to the tangents at $S_i$ in $\vp_i(\vx_{-i})$,
resp.\ $\vp_i(\vy_{-i})$, and $\Ta(\vx) \le \Ta(\vy)$ follows.
\end{proof}
}